\documentclass[12pt]{article}
\usepackage[dvips]{graphicx}
\usepackage[sumlimits,]{amsmath}
\usepackage{amsthm}
\usepackage{amssymb}
\def \bfH { {\mathbf{H}}}

\def \calD { {\mathcal D}}
\def \calE { {\mathcal E}}
\def \calF { {\mathcal F}}

\def \calN { {\mathcal N}}
\def \calS { {\mathcal S}}
\def \calU { {\mathcal U}}
\def \d { {\,\mbox{d}}}

\def \dPxi { {\,\mbox{dP}_\xi}}
\def \ds { \,\mbox{d}s}
\def \dt { \,\mbox{d}t}
\def \dv { \,\mbox{d}v}
\def \dW { \,\mbox{d}W}
\def \dx { \, \mbox{d}x}
\def \dy { \, \mbox{d}y}

\def \iid { {\mbox{i.i.d.}}}

\def \Pxi { {\mbox{P}_\xi}}

\def \st {:\,}

\def \aNupper { {\overline a_N}}
\def \aNlower { {\underline a_N}}
\def \as { {\emph{a.s. }}}

\def \Aeps { {A_\eps}}

\def \Beps { {B_\eps}}

\def \Covalpha { {\mbox{Cov}_\alpha}}
\def \Covxi { {\mbox{Cov}_\xi}}

\def \eps { {\varepsilon}}
\def \epsint { {\lfloor\eps^{-1}\rfloor}}

\def \HNlower { \underline \bfH_N}
\def \HNupper { \overline \bfH_N}

\def \Ieps { {I_\eps}}
\def \Iueps { {I_\ueps}}
\def \Itilde { {\tilde I}}
\def \IZeps { {I_\Zeps}}
\def \noverN { {\frac{n}{N}}}
\def \qeps { {q_\eps}}

\def \range { {\mbox{Range}}}

\def \sbar { {\bar s}}

\def \SNhat { {\hat S_N}}
\def \sovereps { {\frac{s}{\eps}}}
\def \soverN { {\frac{s}{N}}}
\def \tovereps { {\frac{t}{\eps}}}

\def \ueps { {u_\eps}}

\def \veps { {v_\eps}}

\def \Yeps { {Y_\eps}}
\def \Zeps { {Z_\eps}}
\def \xovereps { {\frac{x}{\eps}}}

\newcommand{\brac}[1]{ \langle #1 \rangle}

\newcommand{\Exp}[1]{ \E\left\{ #1 \right\} } 
\newcommand{\Expxi}[1]{ \E_\xi\left\{ #1 \right\} } 
\newcommand{\Exppik}[2]{ \E_{\pi_{#1}}\left\{ #2 \right\} } 
\def \ddx { \frac{\mbox{d}}{\mbox{d} x}}

\def \E { {\mathbb E}}
\def \Exi { {\mathbb E}_\xi}

\def \g { |\,}

\def \p {\partial}

\def \Nat {{\mathbb N}}
\def \one { \mathbf{1}}

\def \Rn { {{\mathbb R}^n}}

\def \Rm {{\mathbb R}^m}
\def \Rone {{\mathbb R}}
\def \Rtwo {{\mathbb R}^2}

\def \Rfour {{\mathbb R}^4}

\def \S { {\mathcal{S}}}

\def \Zint {{\mathbb Z}}

\newtheorem{theorem}{Theorem}[section]
\newtheorem*{theorem*}{Theorem}

\newtheorem{proposition}{Proposition}[section]
\newtheorem{lemma}{Lemma}[section]

\theoremstyle{definition}
\newtheorem*{definition*}{Definition}
\newtheorem{definition}{Definition}[section]

\newtheorem*{algorithmold*}{Algorithm}
\newtheorem{assumptions}{Assumptions}[section]

\theoremstyle{remark}
\newtheorem*{remark*}{Remark}
\newtheorem{remark}{Remark}[section]

\newtheorem*{claim*}{Claim}

\newlength{\howlong}

\begin{document}
\title{Large Deviation Theory for a Homogenized and ``Corrected'' Elliptic ODE}
\author{Guillaume Bal \thanks{Department of Applied Physics and Applied Mathematics, Columbia University; email: ianlangmore@gmail.com},\and Roger Ghanem \thanks{University of Southern California},\and Ian Langmore \footnotemark[1]}
%\author{Guillaume Bal \thanks{Department of Applied Physics and Applied Mathematics, Columbia University}, \and Alexandre Jollivet \thanks{Laboratoire de Physique Th\'eorique et Mod\'elisation, CNRS UMR 8089/Universit\'e de Cergy-Pontoise}, \and Ian Langmore\footnotemark[1], \and Fran\c cois Monard\footnotemark[1]}
\maketitle
\abstract{
We study a one-dimensional elliptic problem with highly oscillatory random diffusion coefficient.  We derive a homogenized solution and a so-called Gaussian corrector.  We also prove a ``pointwise'' large deviation principle (LDP) for the full solution and approximate this LDP with a more tractable form.  Applications to uncertainty quantification are considered.} % Corrector and approximate LDP results similar to these are available whenever a solution can be written as a homogenized solution plus an oscillatory integral.}%These results allow one to access the limits of Gaussian correctors.  In general, the corrector does not capture the large deviation behavior.  Applications to uncertainty quantification are considered.}

\vspace{0.1in}
\noindent\textbf{Keywords:} Differential equations, probability theory, random coefficients, homogenization, applied mathematics, uncertainty quantification
%A large deviation principle is used to select between resolution levels in a one-dimensional elliptic problem.  At the lowest level, \emph{truncation}, the fine-scale randomness in a system is ignored.  The next level is \emph{homogenization}, where the coarse-scale is modified to account for fine-scale randomness.  More fine-scale effects can be added via \emph{Gaussian correctors}, which arise as a result of central-limit theorems.  At the highest level is \emph{full-simulation}.}
% arXiv modification line
\tableofcontents
\section{Introduction}
Partial differential equations whose (deterministic or random) coefficients have fine-scale structure are notoriously difficult to solve.  Here we consider the following elliptic problem:
Given probability space $(\Omega, \calS, P)$, find $\ueps\in L^2(\Omega;H_0^1([0,1]))$ satisfying
\begin{align}
  %-\nabla\cdot A(x,\xi,\xovereps,\omega)\nabla \ueps + q(x,\xi,\xovereps,\omega)\ueps&= f(x,\xi).
  -\ddx A(x,\xi,\xovereps,\omega)\ddx \ueps &= f(x),\quad x\in(0,1),\, \omega\in\Omega,\,\eps\ll1.
  \label{eq:basic_elliptic}
\end{align}
Here $\xi:\Omega\to\Rn$ is a random vector.  
In our case we assume either 
\begin{enumerate}
  \item[(i)]  $\Aeps(x) = A(x,\xi,\xovereps,\omega)$ $= a(x,\xi) + b(\xovereps,\omega)$
  \item[(ii)]  $\Aeps(x) = A(\xovereps)$, with $A(y)^{-1} = \gamma_{\lfloor y\rfloor}$ for a stationary dependent sequence $\gamma_k$ (which depends on $\xi$).  $\lfloor \cdot\rfloor$ denotes the ``floor'' function.
\end{enumerate}
Case (i) is motivated by a Karhunen-Lo\'eve expansion, which for $Q:[0,1]\times\Omega\to\Rone$ having continuous covariance $\Gamma:[0,1]^2\to\Rone$, $\Gamma(x,y):= \E Q(x)Q(y)$, is given by
\begin{align*}
  Q(x,\omega) &= \sum_{j=1}^\infty\sqrt{\lambda_j}\psi_j(\omega)h_j(x),\quad\mbox{where}\quad\Exp{\psi_i\psi_j} = \int h_i(x)h_j(x)\dx = \delta_{ij}.
\end{align*}
See \cite{Loeve_Probability_II,Ghanem_StochasticBook_1991}.
This is multi-scale, but could be approximated by two scales.  In case (i), the high frequency randomness $\xovereps$ is ``decoupled from $\xi$'' in the sense that, after conditioning on $\xi$, $A(x,\xi,y,\omega)$ is stationary in $y$.  So in the model $A=a+b$ we are assuming $b$ is a stationary random field.  For simplicity we always assume $f(x)$ is deterministic.

Case (ii) is an example of dependent media with ``short range correlations.''  Use of different kernels $h_k$ allows some flexibility in modeling.  Although we only consider the case where the final $\gamma_j$ are stationary, generalizations (using e.g. $h_j(x)$) would not be difficult.

In some PDE of interest (e.g. \eqref{eq:basic_elliptic}, or other elliptic equations \cite{Bal_Central_2008}, and also linear transport \cite{Bal_Homogenization2010}) the solution admits an expansion of the form
\begin{align}
  \ueps(x) &= u_0(x) + \veps(x) + R_\eps(x),
  \label{eq:expansion_general}
\end{align}
where $u_0$ is a \emph{homogenized} solution that is dominant in the limit $\eps\to0$ \cite{Jikov_1994_Homogenization_Book,Pavliotis_Multiscale}, the remainder $R_\eps$ is negligible in some sense, and $\veps$ is given by an oscillatory integral
\begin{align}
  \veps(x) &= \int G(x,y)q(y,\frac{y}{\eps})\dy.
  \label{eq:veps_general}
\end{align}
In these cases one expects (and can often prove) that $\eps^{-\alpha}(\ueps-u_0)$ converges in distribution to a Gaussian process $\eps^{\alpha}v(x)$ (often $\alpha=d/2$).  Thus, one is justified in approximating $\ueps$ by $u_0$ plus a Gaussian corrector:
\begin{align}
  \ueps &\approx u_0 + \eps^\alpha v(x).
  \label{eq:clt_general}
\end{align}
From an uncertainty quantification (UQ) perspective, this represents a significant simplification.  Computation of the \emph{homogenized} solution $u_0$ is much less expensive than $\ueps$.  The corrector $v(x)$ has an explicit form in terms of e.g. an It\=o integral.  This allows explicit calculation of the correlation function.  Moreover, draws from the random process $\eps^\alpha v(x)$ can be done with minimal (compared to calculation of $\ueps$) effort.  Another utility of corrector results is for validation of numerical homogenization schemes.  For example, it is known that the numerical homogenization techniques MsFEM and HMM give solutions $u^h_\eps$ that converge to the correct homogenization limit $u_0$ as $\eps,h\to0$.  The question as to whether $\eps^{-\alpha}(u^h_\eps-u_0)$ converges to the correct limit is explored (for \eqref{eq:basic_elliptic}) in \cite{Bal_Corrector_2010}.

As a downside, central limit approximations such as \eqref{eq:clt_general} are expected to work well only for moderate deviations, i.e. for $|\ueps - u_0| \sim O(\eps^\alpha)$.  Sometimes of interest in UQ applications are questions related to large deviations, e.g. $P[\ueps >\ell]$ for some $\ell\sim O(1)$.

Our main contribution is an investigation of the large-deviation behavior of $\ueps$, the solution to \eqref{eq:basic_elliptic}. 
As it turns out, it is possible to derive a large deviation principle (LDP) for the solution $\ueps(x)$ (for fixed $x$).  This gives asymptotic limits of e.g. $\eps \log P[\ueps(x)\geq\ell]$ for $\ell\sim O(1)$.  The resultant \emph{rate function} is given implicitly (see theorems \ref{theorem:LDP_for_parameterized_problems}, \ref{theorem:LDP_for_convolved_media}) as a solution to two (one convex and one non-convex) four-dimensional optimization problems.  We also derive an approximate LDP (proposition \ref{proposition:approximate_LDP}) that corresponds to the approximation $\ueps\approx u_0 + \veps$.  Since $\veps$ is given explicitly as an oscillatory integral, the rate function is ``more explicit'', being the result of a one-dimensional convex optimization problem.  We verify numerically that the approximate rate function works well when $\eps\ll1$ and $\ell\sim O(1)$ but ``not too large.''  This sort of approximate LDP should be available in other situations where the solution can be approximated by a homogenized term and an oscillatory integral.  Along the way we also derive a large deviation principle for some one dimensional oscillatory integrals, which appears to be new as well.

A secondary contribution is a generalization of homogenization and corrector results.  Homogenization results typically start with a uniformly (over all realizations) elliptic diffusion coefficient of the form $A(\xovereps,\omega)$.  In this case the homogenized tensor is constant.  Here we generalize these results slightly by allowing for non-constant low-frequency randomness (in the case (i)) and relaxing the uniform ellipticity requirement to \eqref{eq:ellipticity_condition}.  We do this by conditioning on the coarse-scale and bounding higher moments of $\Aeps^{-1}$.  These assumptions are more in line with those encountered in UQ.  We also prove almost sure convergence of the homogenized tensor.  This is motivated by the fact that in practice the homogenized tensor can be obtained by picking one high-frequency media realization and averaging over a domain of size $\rho$ \cite{Bourgeat_Approximations}.  Thus, it is nice to know that with probability one this realization converges as $\rho\to\infty$.

Our large deviation result allows us to determine (roughly) to what degree the Gaussian corrector captures the tail behavior of the solution.  This is useful in applications where one may consider replacing the full solution with the homogenized solution plus a Gaussian corrector.  Also, although we don't answer this here, questions such as ``does HMM capture the large-deviation behavior of the solution'' could potentially be answered in a manner similar to the question ``does HMM capture the moderate deviation behavior'' as discussed above.  Also of interest (and also not pursued further here) is the relation between large deviations and importance sampling of rare events.  As it turns out, a large deviations result can give ``asymptotically efficient'' importance functions \cite{Bucklew_IntroductionToRare_2004,Dieker_OnAsymptotically_2005}.

In section \ref{section:asymptotic_expansion} an asymptotic expansion of $\ueps$ is presented.  In section \ref{section:basis_reduction_1d}, theorems \ref{theorem:convergence} and \ref{theorem:corrector} give results on the homogenized convergence $\ueps\to u_0$, and the corrector characterizing moderate deviations of $\ueps-u_0$.  In section \ref{section:large_deviations} large deviations are considered.  After introducing the subject we derive a large deviation principle for two types of media, each with piecewise constant high-frequency parts (theorems \ref{theorem:LDP_for_parameterized_problems}, \ref{theorem:LDP_for_convolved_media}).  We next present our approximate LDP in \ref{subsection:approximate_rate_functions} and then numerical results in section \ref{section:example_problems}.  Proofs of the homogenization and corrector theorems, which are generalizations of known results, are 
% arXiv modification line
%not given here.  See \cite{Langmore_Large_ArXiv_2011} for full details.
relegated to sections \ref{subsection:homogenization_theorem_proof} and \ref{subsection:corrector_proof}.  

\section{Asymptotic expansion of the solution $\ueps$}
\label{section:asymptotic_expansion}
The boundary value problem \eqref{eq:basic_elliptic} may be integrated and the solution is 
\begin{align}
  \ueps(x) &= -\int_0^x\frac{F(s)}{\Aeps(s)}\ds + \frac{\brac{F/\Aeps}}{\brac{1/\Aeps}}\int_0^x\frac{1}{\Aeps(s)}\ds.
  %&= \brac{F/\Aeps}\left[ \int_0^x\left( \frac{1/\Aeps(s)}{\brac{1/\Aeps}} - \frac{F(s)/\Aeps(s)}{\brac{F/\Aeps}} \right)\ds \right].
  \label{eq:basic_elliptic_solution}
\end{align}
Here we define \emph{1-average} $\brac{\cdot}$ of a function $\phi:[0,1]\to\Rone$, and $F:[0,1]\to\Rone$ by
\begin{align*}
  \brac{\phi} :&= \int_0^1\phi(y)\dy,\qquad F(s) := \int_0^sf(t)\dt.
\end{align*}
Define the \emph{homogenized tensor} $A_0$ by
\begin{align*}
  A_0(x) :&= \left( \Expxi{A(x,0)^{-1}} \right)^{-1},
\end{align*}
where for $X:\Omega\to\Rone$, we denote conditional expectation by
\begin{align*}
  \Expxi{X} :&= \Exp{X\g \xi} = \int_\Omega X(\omega)\dPxi(\omega).
\end{align*}
Above the measure $\Pxi=P[\cdot\g\xi]$ is defined implicitly.

Now re-write the integrals appearing in \eqref{eq:basic_elliptic_solution} as 
\begin{align*}
  \int_0^x\frac{1}{\Aeps(s)} &= \int_0^x\frac{1}{A_0(s)}\ds + X^\eps_x,\quad
  \int_0^x\frac{F(s)}{\Aeps(s)} = \int_0^x\frac{F(s)}{A_0(s)}\ds + Y^\eps_x.
\end{align*}
Defining the \emph{homogenized solution} $u_0$ by the equation \eqref{eq:basic_elliptic_solution} with $A_0$ rather than $\Aeps$ (or equivalently the weak solution to \eqref{eq:basic_elliptic} with $A_0$ rather than $\Aeps$), we have the following expansion.
\begin{align}
  \begin{split}
    \ueps(x) &= u_0(x) + \veps(x) + R_\eps(x),\\
    \veps(x) :&= -Y_x^\eps + \left( Y_1^\eps - X_1^\eps\frac{\brac{F/A_0}}{\brac{1/A_0}} \right)\frac{1}{\brac{1/A_0}}\int_0^x\frac{1}{A_0(s)}\ds + X_x^\eps\frac{\brac{F/A_0}}{\brac{1/A_0}},\\
    R_\eps(x) :&= (X_1^\eps)^2\frac{\brac{F/\Aeps}}{\brac{1/A_0}^2\brac{1/\Aeps}}\int_0^x\frac{1}{\Aeps(x)}\ds\\
    &\quad- \frac{X_1^\eps}{\brac{1/A_0}^2}\left[ Y_1^\eps\int_0^x\frac{1}{\Aeps(s)}\ds + \brac{F/\Aeps}X_x^\eps \right] + \frac{Y_1^\eps X_x^\eps}{\brac{1/A_0}}.
  \end{split}
  \label{eq:asymptotic_expansion}
\end{align}
The deterministic homogenized solution $u_0$ is dominant in the limit $\eps\to0$.  As we will show, the remainder $R_\eps$ is $O(\eps)$ in $L^1(\Omega\times[0,1])$.  The term $\veps$ can be re-written
\begin{align}
  \begin{split}
    \veps(x) &= \int_0^1q_\eps(s)G(x,s)\ds,\qquad q_\eps(s) := \frac{1}{\Aeps(s)} - \frac{1}{A_0(s)},\\
    G(x,s) :&=\left\{
    \begin{matrix}
      \left( F(s)- \frac{\brac{F/A_0}}{\brac{1/A_0}} \right)\left( \frac{1}{\brac{1/A_0}}\int_0^x\frac{1}{A_0(t)}\dt - 1 \right),&\quad 0\leq s\leq x\\
      \left( F(s) - \frac{\brac{F/A_0}}{\brac{1/A_0}} \right)\frac{1}{\brac{1/A_0}} \int_0^x\frac{1}{A_0(t)}\dt ,&\quad x<s\leq1.
    \end{matrix}
    \right.
  \end{split}
  \label{eq:veps_rewritten}
\end{align}
Note that $G(x,s)$ is piecewise continuous in $s$ and so long as $A_0^{-1}$ is bounded, $G(x,s)$ is uniformly (in $s$) Lipschitz in $x$.
We also show that $\veps$ is $O(\sqrt{\eps})$ (in $L^2(\Omega\times[0,1])$) and has a limit that can be characterized well.  Note that these are slight generalizations of previous results.  In particular this limit was studied by \cite{Bourgeat_Approximations}, and in the case of media with long-range correlations in \cite{Bal_Random_2008}.  This paper adds the feature that the random media is allowed to be non-stationary in one variable, and has no uniform (in $\omega$) ellipticity lower bound.

\section{Homogenization and Gaussian Corrector}
\label{section:basis_reduction_1d}
Here we obtain homogenization and Gaussian corrector results that are a generalization of known results (e.g. \cite{Jikov_1994_Homogenization_Book,Pavliotis_Multiscale,Bourgeat_Estimates,Bal_Central_2008}) to the case of media that has no uniform (in $\omega$) upper or lower bounds.  
% arXiv modification line
%We refer the reader to \cite{Langmore_Large_ArXiv_2011} for the proofs.
Proofs are deferred till section \ref{section:homogenization_corrector_proofs}.

First, we assume $\Aeps(x)$ satisfies
\begin{align}
  0&<\nu_1(\omega,\eps)\leq \Aeps(x)\leq \nu_2(\omega, \eps)<\infty.
  \label{eq:ellipticity_condition}
\end{align}
%Our model problems will make assumptions on $\nu_1(\omega,\eps)$ ensuring $\ueps\in L^2(\Omega;H_0^1( (0,1)))$ (with a bound depending on $\eps$), and $\ueps\in L^2(\Omega\times(0,1))$ (with a bound independent of $\eps$).

We abuse notation by writing 
  $\Aeps(x) := A(x,\xovereps) = A(x,\xovereps, \omega)$.
The form $A(x,\xovereps,\omega)$ emphasizes that $A$ is a random field defined on a probability space $(\Omega, \S, P)$.  The form $A(x,\xi,\xovereps,\theta)$ emphasizes the dependence on a special random vector $\xi:\Omega\to\Rm$, and a (possibly infinite) sequence of random variables $\theta$.  Once we fix $\xi$, $A(x,y)$ exhibits some stationary in $y$ (although only weak-stationarity of $A^{-1}$ is needed for homogenization in one dimension).  

Note that functions such as $\Expxi{\Aeps}$ do not depend on $y$, so we write $\Expxi{Q}(x) = \Expxi{Q(x,0)}$ for functions $Q(x,\xi,y,\theta)$ whose conditional expectation does not depend on $y$.   It is necessary to make some ergodicity assumptions on the process (in $y$).  In one dimension, we require mean-ergodicity of $\Expxi{A^{-1}}$, which we express through decay of the covariance \eqref{eq:covariance_integral_bound}.
We also assume $\|f\|_{L^2}<\infty$.  

\subsection{Homogenization}
We assume weak stationarity of $A^{-1}$ and ellipticity of $A_0$.  In other words, we assume
\begin{align*}
  0&<c_1(\xi)\leq \Expxi{A^{-1}(x,y)}=\Expxi{A^{-1}(x,0)} = A_0(x)^{-1}\leq c_2<\infty,
\end{align*}
and that the conditional covariance 
\begin{align*}
    &\Covxi(x_1,x_2,z):=\Expxi{\left[ A^{-1}(x_1,y) - \Expxi{A^{-1}}(x_1) \right]
    \left[ A^{-1}(x_2,y+z)-\Expxi{A^{-1}}(x_2) \right]},
\end{align*}
is independent of $y\in[0,\infty)$.  We also assume
\begin{align}
  |\Covxi(x_1,x_2,z)|&\leq \Gamma(z)\quad\mbox{with}\quad |\Gamma\|_{L^1}=: C_{A^{-1}}<\infty.
  \label{eq:cov_pointwise_bound}
\end{align}
Note that this implies 
\begin{align}
  \begin{split}
    \eps^2\int_0^{\eps^{-1}}\int_0^{\eps^{-1}}\left|\Covxi(x_1,x_2,y-\tilde y)\dy\d\tilde y\right| &\leq \eps C_{A^{-1}}.
  \end{split}
  \label{eq:covariance_integral_bound}
\end{align}

This allows
\begin{theorem}
  \begin{align*}
    \sqrt{\Exi\|\ueps-u_0\|_{L^2([0,1])}} &\leq 3\sqrt{\eps}\|f\|_{L^2}\sqrt{C_{A^{-1}}}.
  \end{align*}
  Moreover, as $\eps\to0$, with $\xi$ fixed,
  \begin{align*}
    \ueps(x) - u_0(x) &\to0 \mbox{ for every } x, \quad \Pxi\quad\as
  \end{align*}
  and if the ellipticity bound \eqref{eq:ellipticity_condition} is independent of $\eps$,
  \begin{align*}
    \|\ueps-u_0\|_{L^2}\to0, \quad \Pxi\quad\as
  \end{align*}
  \label{theorem:convergence}
\end{theorem}

\subsection{Gaussian corrector}
To quantify the rate at which the random coefficient decorrelates, we introduce the following mixing condition (referred to in the literature as $\rho$ mixing).
  \begin{assumptions}[Mixing Condition]
    For two Borel sets $A,B\subset\Rone$, let $\calS_A$, $\calS_B$ be the sub sigma algebras generated by $A(x,y)$, $x\in[0,1]$, $y\in A$ and $y\in B$ respectively.  We assume the existance of non-negative bounded and decreasing $\varphi:[0,\infty)\to\Rone$ such that $\varphi^{1/3}\in L^1$ that also satisfies
    \begin{align*}
      \sup\left\{ |\E[\eta_a\eta_b]|\st \eta_a\in\calS_A, \eta_b\in\calS_B, \E\eta_a^2=\E\eta_b^2=1, \E\eta_a=\E\eta_b=0\right\} &\leq \varphi(\d(A,B)).
    \end{align*}
    \label{assumptions:mixing_condition}
  \end{assumptions}
  After conditioning on $\xi$ (e.g. fixing one realization of $a(x,\xi)$ if $A=a+b$) we are able to partially characterize the limiting distribution of $\ueps-u_0$. 
\begin{theorem}
  If $A(x,y)$ is stationary in $y$, and $(s_1,s_2)\mapsto\Covxi(s_1,s_2,0)$ is continuous at $(x,x)$, $A^{-1}$ satisfies the mixing condition \eqref{assumptions:mixing_condition}, and 
  \begin{align*}
    \Expxi{\left( \frac{1}{A(x,y)} - \Expxi{\frac{1}{A(x,y)}} \right)^6}&\leq C_\xi<\infty,
  \end{align*}
  then
  \begin{align*}
    \frac{\ueps(x)-u_0(x)}{\sqrt{\eps}} &\xrightarrow{dist.} v(x) := \int_0^1G(x,t)\sigma(t)\dW_t,\\
    %&\mbox{\textbf{For my notes}}\\
    %&=\int_0^x\left( \frac{\brac{F/A_0}}{\brac{1/A_0}} - F(t) \right)\sigma(t)\dW_t \\
    %&\quad- \left[ \frac{1}{\brac{1/A_0}}\int_0^x\frac{1}{A_0(t)}\dt \right]\int_0^1\left( \frac{\brac{F/A_0}}{\brac{1/A_0}} - F(t) \right)\sigma(t)\dW_t.
    %\\&\mbox{\textbf{End of for my notes}}
  \end{align*}
  where $W_t$ is a one-dimensional Brownian motion, $G$ is given by \eqref{eq:veps_rewritten}, and
  \begin{align*}
    \sigma^2(t) &= \int_{-\infty}^\infty \Covxi(t,t,q)\d q.
  \end{align*}
  \label{theorem:corrector}
\end{theorem}
\begin{remark}
  The continuity on $\Covxi$ is equivalent to mean-square continuity of $x\mapsto A(x,0)^{-1}$ (see e.g. \cite{Loeve_Probability_II}).  This means more-or-less that the media varies slowly with respect to $x$.  This is where scale-separation comes in.
\end{remark}
\begin{remark}
  The variance of the above expression is then given by 
  \begin{align}
    \begin{split}
      &\int_0^1 G(x,t)^2\sigma^2(t)\dt.
      %\int_0^{x\wedge y}&Z(t)^2\dt - H(y)\int_0^xZ(t)^2\dt - H(x)\int_0^y Z(t)^2\dt + H(x)H(y)\int_0^1Z(t)^2\dt,\\
      %Z(t) :&= \left( \frac{\brac{F/A_0}}{\brac{1/A_0}} - F(t) \right)\sigma(t),\quad
      %H(x) :=  \frac{1}{\brac{1/A_0}}\int_0^x\frac{1}{A_0(t)}\dt.
  %\\&\textbf{For my notes}\\
  %&= H(x)^2\int_0^1(F(s)-R)^2\sigma^2(s)\ds + (1-2H(x))\int_0^x(F(s)-R)\sigma^2(s)\ds
  %\\&\textbf{End of for my notes}
    \end{split}
    \label{eq:corrector_covariance}
  \end{align}
\end{remark}
\begin{remark}
  The lack of a uniform (in $\omega$) ellipticity lower bound is dealt with through the bounds on sixth moments and the mixing conditions.
\end{remark}
%Note that if $A(x,\xovereps)=A(\xovereps)$, then $\sigma$ and $A_0$ are constant so that the limiting distribution is
%\begin{align*}
%  -\sigma\int_0^xF(t)\dW_t + x\sigma\left( \int_0^1F(t)\dW_t - \brac{F}W_1 \right) + W_x\sigma\brac{F}.
%\end{align*}
%\begin{align*}
%  \sigma\int_0^x\left( \brac{F}-F(s) \right)\dW_s - x\sigma\int_0^1\left( \brac{F}-F(s) \right)\dW_s.
%\end{align*}

\section{Large Deviations}
\label{section:large_deviations}
%The principle question we wish to answer here is whether a Gaussian corrector or a full simulation should be used.  Gaussian random variables have symmetric, light tails.  So they might not always capture large behavior in the solution $\ueps$.  
%Gaussian random variables have symmetric, light tails that extend to infinity.  So they might not always capture large behavior in the solution $\ueps$.%Large deviations theory gives asymptotic probabilities such as $\lim_{\eps\to0}\eps \log P[\ueps\geq\ell]$.  
Here we derive a rigorous ``pointwise'' large-deviations result for $\ueps$, and an approximate rate function.  
\subsection{Moderate and large deviations for our problem}
\label{subsection:moderate_and_large_deviations}
For small enough $\eps$, the homogenized solution $u_0$ captures the bulk of the solution $\ueps$.  The corrector attempts to capture some statistics of the term $\ueps-u_0 = \veps + R_\eps$.  Our corrector result shows that
\begin{align}
  \begin{split}
    \frac{\ueps(x) - u_0(x)}{\sqrt{\eps}}:&= \frac{1}{\sqrt{\eps}}\int_0^1 G(x,t) q(t,\tovereps) +\frac{R_\eps(x)}{\sqrt{\eps}}\\
    &\xrightarrow{dist.} v(x) := \int_0^1 G(x,t)\sigma(t)\dW_t.
  \end{split}
  \label{eq:simplified_oscilliatory_integral}
\end{align}
By definition this means that, for any $\ell>0$
\begin{align}
  \Pxi\left[ \frac{\ueps(x)-u_0(x)}{\sqrt{\eps}}\geq \ell \right] &\to \Pxi[v(x)\geq\ell].
  \label{eq:moderate_deviation}
\end{align}
A relevant question is \emph{whether or not}
\begin{align}
  \Pxi[\ueps(x)\geq\ell] &\approx \Pxi[u_0(x) + \sqrt{\eps}v(\eps)\geq\ell].
  \label{eq:large_deviation}
\end{align}
The inequality $\ueps\geq\ell$ (for $\ell>u_0$) is a \emph{large deviation} since \eqref{eq:simplified_oscilliatory_integral} (or a law of large numbers result) shows that $\ueps$ concentrates near $u_0$.  On the other hand, $\eps^{-1/2}(\ueps-u_0)\geq\ell$ is a \emph{moderate deviation}.  Generally speaking, \eqref{eq:large_deviation} \emph{does not} hold.  Instead, from \eqref{eq:moderate_deviation} we can only rigorously infer something about moderate deviations, namely $\Pxi[\ueps(x)\geq\sqrt{\eps}\ell]\approx \Pxi[u_0(x)+\sqrt{\eps}v(x)\geq\sqrt{\eps}\ell]$.

%The proof of theorem \ref{theorem:corrector} shows that $\Exi\|R_\eps\|_{L^1}\lesssim\eps$.  This sort of expansion is available in higher dimensions in some elliptic and transport problems \cite{Bal_Central_2008,Bal_Homogenization2010}.  In all cases, we have $L^2$ control over the remainder.  This is not enough to obtain a rigorous large deviations result.  Indeed, assuming only $\E\|X\|_{L^1}\lesssim\eps^d$ one can do no better than Chebyshev's inequality
%\begin{align*}
%  P[|X|>\ell]&\leq \frac{\E\|X\|_{L^1}}{\ell} \lesssim \eps^d,
%\end{align*}
%which falls far short of exponential decay.
%Nonetheless, we proceed with a large deviation study using the approximation $\ueps(x) \approx u_0(x) + \veps(x)$, and test numerically its validity.  The plan is to calculate the term $u_0(x)$ explicitly, and then bound the term $\veps(x)$.  

For simplicity, we consider large deviations at only one fixed $x\in(0,1)$ and often change the notation to $u_0$, $\ueps$, $\veps$ (dropping the $x$ dependence).

\begin{definition}[Rate functions]
  A \emph{rate function} $I$ is a lower semicontinuous mapping (such that for all $\alpha\in[0,\infty)$, the sub-level set $\Psi_I(\alpha):=\{x\st I(x)\leq\alpha\}$ is closed) $I:\Rn\to[0,\infty]$.  A \emph{good rate function} is a rate function for which all the sub-level sets $\Psi_I(\alpha)$ are compact.  
\end{definition}

\begin{definition}
  We say that a family of random vectors $\Yeps\in\Rn$ satisfy a large deviations principle (LDP) with rate function $I$ if for all $\Gamma\subset\Rn$
  \begin{align*}
    -\inf_{y\in\Gamma^\circ}I(y) &\leq \liminf_{\eps\to0}\eps \log P[\Yeps\in\Gamma]\leq\limsup_{\eps\to0}\eps\log P[\Yeps\in\Gamma]
    \leq-\inf_{y\in\bar\Gamma}I(y).
  \end{align*}
  Above, $\Gamma^\circ$, $\bar\Gamma$ denote the interior and closure of $\Gamma$.
  \label{definition:LDP}
\end{definition}
Considering now our solution $\ueps\in\Rone$, by choosing $A=[\ell,\infty)$ we obtain limiting upper and lower bounds on $\Pxi[\ueps\geq\ell]$.  We also often have, for $\ell>\lim\Exi\ueps$, %If $I(\ell)$ is non-decreasing and continuous at $\ell>a=\lim\Exi\ueps$ (usually the case), then it follows that (for $\ell>a$)
\begin{align}
  \lim_{\eps\to0}\eps\log P[\ueps\geq\ell] = -I(\ell).
  \label{eq:LDP_equality}
\end{align}
We will find a rate function $\Iueps$ for $\ueps$.  We also find a rate function $\Itilde$ for $u_0+\veps$, which proves to be more tractable.  Simulations show that $\Iueps\approx\Itilde$ for $\ell\sim O(1)$ but not too large (see figures \ref{fig:parameterized_LDP_theory_mild}, \ref{fig:parameterized_LDP_theory_wild}, and \ref{fig:convolved_LDP1}).

The rate-function $\Ieps$ gives the exponential rate of convergence in the sense that given $\ell\in\Rone$, $\delta>0$, there exists $\eps_0$ such that for $\eps<\eps_0$
\begin{align}
  \Pxi[\ueps\geq\ell] &\leq e^{-(I(\ell) - \delta)/\eps}.
  \label{eq:LDP_exponential_rate}
\end{align}
As an example, consider the corrector in our 1-d problem, $\sqrt{\eps}v$.  It is Gaussian with mean zero and variance equal to $\eps C_c$ for some $C_c$ (given by \eqref{eq:corrector_covariance}).  Therefore $\ueps\approx u_0+\sqrt{\eps}v\sim\calN(u_0,\eps C_c)$.  Keeping the first term in an asymptotic expansion of the complementary error function, we have that for $\ell>u_0$,
\begin{align*}
  \Pxi[u_0 + \sqrt{\eps}v\geq\ell] %&= \frac{1}{\sqrt{2\pi}\ell}\int_{\ell-u_0(x)}^\infty e^{-x^2/(2C_c\eps)}\dx \\
  &\sim \frac{\sqrt{C_c\eps}}{\sqrt{2\pi}(\ell-u_0)}e^{-(\ell-u_0)^2/(2C_c\eps)}.
\end{align*}
Thus, when
\begin{align}
  \log(C_c\eps/(\ell-u_0)^2) \ll \frac{(\ell-u_0)^2}{C_c\eps},
  \label{eq:asymptotic_validity_test}
\end{align}
we will have $\eps \log \Pxi[u_0+\sqrt{\eps}v\geq\ell]\approx -(\ell-u_0)^2/(2C_c)$.  Note that this is a ``small $\eps$ and large $|\ell-u_0|$'' condition, as it should be.
Comparing this to \eqref{eq:LDP_exponential_rate} we see that 
the Gaussian corrector captures the asymptotic tail behavior when 
\begin{align}
  \frac{(\ell-u_0)^2}{2C_c} &\approx I(\ell).
  \label{eq:tail_condition}
\end{align}
If \eqref{eq:tail_condition} does not hold, the corrector cannot capture the tail behavior of $\ueps$.  A cautionary note is in order here.  The large deviations result captures the exponential rate of decay, and important algebraic factors in $\eps$ are not captured.  So, for finite $\eps$ the rate function can be used for comparative purposes, not to estimate the true tail.
It should be noted that so-called concentration inequalities provide a number of upper bounds on sums of random variables, often in pre-asymptotic regimes.  Often these require fewer assumptions but do not claim to be tight.  For example, \emph{Chernoff's bounding method} is given in \eqref{eq:LD_upper_bound_always_true}.  See \cite{Boucheron_Concentration} for a survey, and \cite{Lucas_Rigorous} for an application to uncertainty quantification.  For independent media, we derive rigorous upper bounds for finite $\eps$ (e.g. \eqref{eq:LD_upper_bound_always_true}).  For dependent media similar bounds are available but are more complex than the asymptotic bounds.  Another reason for using asymptotic bounds is that a ``fair'' comparison such as \eqref{eq:tail_condition} can be made.

\subsection{Independent sums and basic definitions}
\label{subsection:independent_sums_and_basic_definitions}
It is instructive to start here.  Let $X_n$ be random variables and define
\begin{align*}
  \SNhat :&= \frac{1}{N}\sum_{n=1}^NX_n.
\end{align*}
Now, for $\lambda\geq0$,
\begin{align}
  P[\SNhat\geq\ell] &= P[N\SNhat\geq N\ell] = \Exp{\one_{N\SNhat\geq N\ell}} \leq e^{-N\lambda\ell}\Exp{e^{\lambda N\SNhat}}.
  \label{eq:LD_upper_bound_preliminary}
\end{align}
With vector-valued random variables in mind, we define
\begin{definition}
  The \emph{logarithmic moment generating} function for a random variable $Y$ is defined as
  \begin{align*}
    \Lambda(Y, \lambda) = \Lambda_Y(\lambda):&= \log \Exp{e^{\lambda\cdot Y}}.
  \end{align*}
  \label{definition:logarithmic_MGF}
\end{definition}
\eqref{eq:LD_upper_bound_preliminary} allows us to conclude
\begin{align}
  \frac{1}{N}\log P[\SNhat\geq\ell] &\leq -\sup_{\lambda\geq0} \left[ \lambda\ell - \frac{1}{N}\Lambda(N\SNhat, \lambda) \right].
  \label{eq:LD_upper_bound_always_true}
\end{align}
The above (\emph{Chernoff's}) bound holds without any assumptions.  Suppose the $X_n$ are i.i.d.  This leads to

\begin{align*}
  \Exp{e^{\lambda N\SNhat}} &= \prod_{n=1}^N \Exp{e^{\lambda X_1}},
\end{align*}
and thus
\begin{align}
  \frac{1}{N}\Lambda(N\SNhat, \lambda N) &= \frac{1}{N}\sum_{n=1}^N \log \Exp{e^{\lambda X_1}} = \Lambda_X(\lambda).
  \label{eq:derivation_of_Lambda}
\end{align}
Note that $\Lambda_X$ is convex since by H\"older's inequality
\begin{align*}
  &\Lambda_X(t\lambda_1 + (1-t)\lambda_2) = \log\Exp{(e^{\lambda_1 X})^t(e^{\lambda_2X})^{(1-t)}}\\
  &\quad\leq \log\left\{ \Exp{e^{\lambda_1X}}^t\Exp{\lambda_2X}^{(1-t)} \right\} = t\Lambda_X(\lambda_1) + (1-t)\Lambda_X(\lambda_2).
\end{align*}
\begin{definition}
  The \emph{Frenchel-Legendre} transform of $\Lambda(Y,\cdot)$ is defined by 
  \begin{align*}
    \Lambda^\ast(Y,\ell) = \Lambda^\ast_Y(\ell):&= \sup_{\lambda\in\Rn}\left[ \lambda\cdot\ell - \Lambda_Y(\lambda) \right]. 
  \end{align*}
\end{definition}
Inserting \eqref{eq:derivation_of_Lambda} back into \eqref{eq:LD_upper_bound_always_true} (taking into account negative $\lambda$) we have a large deviation upper bound with rate function $\Lambda^\ast_X(\cdot)$.  The ``trick'' is to obtain a lower bound and thus show that this upper bound is tight in the limit $N\to\infty$.  This indeed is the case and the result is 
\begin{theorem}[Cram\'er, \cite{Dembo_Large}]
  The sum $\SNhat$ satisfies the LDP with good convex rate function $\Lambda^\ast_X(\cdot)$.  Moreover, \eqref{eq:LDP_equality} holds.
  \label{theorem:cramer}
\end{theorem}

Since $\Lambda_X(0)=0$, we always have $\Lambda^\ast_X(\ell)\geq0$, i.e. we never have exponential growth.  Jensen's inequality shows that $\Lambda_X(\lambda) \geq \lambda \E X$, so for $\ell= \E X$ we also have $\lambda\ell - \Lambda_X(\lambda )\leq0$.  Therefore $\Lambda^\ast_X(\E X)=0$.  This makes sense in view of the law of large numbers.  %Also note that since $\E\veps=0$, the definition of $\Lambda^\ast$ gives us $0=\Lambda^\ast(0)=\sup_\lambda[-\Lambda(\lambda)]$, and therefore $\inf_\lambda\Lambda(\lambda)=0$.

\subsection{The G\"artner-Ellis theorem and contraction principle}
\label{subsection:Gartner_Ellis}
\begin{definition}
  For any function $H$ taking values in $(-\infty,\infty]$, we define $\calD_H :=\{x\st H(x)<\infty\}$.  
\end{definition}
The G\"artner-Ellis theorem will be used to prove results for our model problems.
\begin{definition}
  Convex $\Lambda:\Rn\to(-\infty,\infty]$ is called \emph{essentially smooth} if $\calD_\Lambda^\circ$ is nonempty, $\Lambda$ is differentiable throughout $\calD_\Lambda^\circ$, and $\Lambda$ is \emph{steep}.  Steep means that $\lim_{j\to\infty}|\nabla\Lambda(\lambda_j)|=\infty$ whenever $\{\lambda_j\}$ is a sequence in $\calD_\Lambda^\circ$ converging to a boundary point of $\calD_\Lambda^\circ$.
  \label{definition:essentially_smooth}
\end{definition}
\begin{theorem}[G\"artner-Ellis \cite{Dembo_Large}]
  Suppose 
  \begin{align*}
    \Lambda(\lambda) :&= \lim_{\eps\to0}\eps\log\Exp{e^{\eps^{-1}\lambda\cdot\Zeps}}
  \end{align*}
  exists as an extended real number.  Furthermore suppose that $\Lambda$ is essentially smooth, lower semicontinuous and that the origin belongs to the interior of $\calD_\Lambda$.  Then $\Zeps$ satisfies an LDP with good convex rate function $\Lambda^\ast$ defined by
  \begin{align*}
    \Lambda^\ast(\ell) :&= \sup_{\lambda\in\Rn}\left[ \lambda\cdot\ell - \Lambda(\lambda) \right].
  \end{align*}
  \label{theorem:GE}
\end{theorem}
\begin{remark*}
  If $n=1$ and $\Lambda^\ast$ is finite in a neighborhood of $\ell>a:=\lim\E\Zeps$ then convex $\Lambda^\ast$ is non-decreasing on $[a,\infty)$ and continuous in this neighborhood.  Therefore \eqref{eq:LDP_equality} holds.
\end{remark*}
Notice that the G\"artner-Ellis theorem does not require independence.  Indeed, one can use it to prove an LDP for mixing random variables \cite{Dembo_Large,Bryc_Large_1992}.  

While the G\"artner-Ellis theorem allows us to obtain an LDP for oscillatory integrals, we need the contraction principle for functions of those integrals such as \ref{eq:basic_elliptic_solution}.
\begin{theorem}[Contraction principle]
  Suppose $f:\Rn\to\Rm$ is continuous and $I:\Rn\to[0,\infty]$ is a good rate function for the family of random variables $\Zeps$ and associated measures $\mu_\eps$ ($\mu_\eps(A)=P[\Zeps\in A]$).  For $y\in\Rm$, define
  \begin{align*}
    I'(y) :&= \inf\{I(x)\st x\in\Rn, y=f(x)\}.
  \end{align*}
  Then $I'$ is a good rate function controlling the LDP associated with the measures $\mu_\eps\circ f^{-1}$ ($\mu_\eps\circ f^{-1}(B) = P[f(\Zeps)\in B]$).
  \label{theorem:contraction_principle}
\end{theorem}
In other words, with $Y_\eps:=f(\Zeps)$, the rate at which $\Zeps$ concentrates away from $\ell$ will be determined by the point in $f^{-1}(\{\ell\})$ holding the most mass.

\subsection{Large deviations for $\ueps$}
\label{subsection:large_deviations_for_functions_of_oscillatory_integrals}
In light of \eqref{eq:basic_elliptic_solution}, we have
\begin{align}
  \begin{split}
    \ueps(x) &= g(\Zeps),\qquad g:\Rone^2\times(0,\infty)^2\to\Rone,\quad g(z)=g(z_1,z_2,z_3,z_4)=-z_1 + z_2\frac{z_3}{z_4},\\
    \Zeps &= (\Zeps_1,\dots,\Zeps_4),\quad \Zeps_i = \int_0^1\frac{H_i(s)}{\Aeps(s)}\ds,\\
    H_1 &= F\one_{(0,x)},\quad H_2=F,\quad H_3=\one_{(0,x)},\quad H_4 = 1.
  \end{split}
  \label{eq:Zeps_definition}
\end{align}
Also set
\begin{align*}
  \bfH :&= (H_1,H_2,H_3,H_4),
\end{align*}
so that $\lambda\cdot\Zeps = \int \lambda\cdot\bfH/\Aeps$.

Using the G\"artner-Ellis theorem we will find rate function $\IZeps$ for $\Zeps$, and then by the contraction principle the rate function for $\ueps$ is 
\begin{align}
  \Iueps(\ell) &= \inf_{z\in g^{-1}\{\ell\}}\IZeps(z).
  \label{eq:Iueps_in_terms_of_IZeps}
\end{align}

The real work involved here is in characterizing the limiting Cram\'er functional $\lim\eps\log\Exi\exp\{\eps^{-1}\lambda\cdot \Zeps\}$.  We will do this for the media types (i) and (ii) mentioned in the introduction.

\subsubsection{Parameterized, independent, uniformly elliptic media}
\label{subsubsection:parameterized_media}
We assume here that the high frequency media is piecewise constant and independent:
\begin{align}
  \begin{split}
    A(x,y) &= a(x,\xi) + b(y,\theta),\quad\mbox{with every realization $a(x,\xi)$ continuous},\\
    b(y,\theta) &= \nu_b\sum_{j=1}^\infty \theta_j \one_{[n-1,n)}(y),\quad \theta_i\sim_\iid \pi_\theta,\quad |\theta|\leq1,\\
    0&<\nu_1\leq A(x,y)\leq\nu_2.
  \end{split}
  \label{eq:parameterized_media_definition}
\end{align}

Considered as a discrete process (at points centered at $n\eps$), the field is stationary and ergodic.  It is not truly stationary since the correlation $\Exi A(x,y)A(x,z)$ depends on more than the difference $y-z$.  Nonetheless we apply our theorems and obtain results that are validated by simulation.
The correlation condition is satisfied trivially since $A(x_1,y_1)$, and $A(x_2,y_2)$ are conditionally independent whenever $|y_1-y_2|>1$.  
Note also that the low and high frequency parts are in the form of Karhunen-Lo\'eve expansions, but the total field is not.  

Since this solution \eqref{eq:basic_elliptic_solution} involves $\Aeps^{-1}$, it is not surprising that we need to define
\begin{align}
  V_\alpha :&= \frac{1}{\alpha + \nu_b\theta} \sim \pi_{V_\alpha}(v) = \frac{1}{\nu_b}\frac{1}{v^2}\pi_\theta\left( \frac{1}{\nu_b}\left( \frac{1}{v}-\alpha \right) \right),
  \label{eq:Valpha_definition}
\end{align}

The main result for this media is
\begin{theorem}
  With $g$, $\Zeps$, $\bfH(s)$, $V_\alpha$ given by \eqref{eq:Zeps_definition}, \eqref{eq:parameterized_media_definition}, \eqref{eq:Valpha_definition}, define
\begin{align*}
  \Lambda(\lambda) :&= \int_0^1\Lambda(V_{a(s,\xi)},\lambda\cdot \bfH(s))\ds,
  %\lim\eps\Lambda(\eps^{-1}\brac{F/\Aeps},\lambda) &= \Lambda(\lambda) := \int_0^1 \log\Expxi{\exp\left\{\frac{\lambda F(s)}{a(s) + b(0)}\right\}}\ds.
\end{align*}
then $\Lambda\in C^\infty(\Rfour)$ is a convex function such that when $\Aeps$ is defined by \eqref{eq:parameterized_media_definition}
\begin{align*}
  \lim\eps\log \Exi e^{\eps^{-1}\lambda\cdot\Zeps} &= \Lambda(\lambda).% := \int_0^1 \log\Expxi{\exp\left\{\frac{\lambda F(s)}{a(s) + b(0)}\right\}}\ds.
\end{align*}
Moreover, for fixed $\xi$, $\Zeps$ satisfies a large deviation principle with good convex rate function 
\begin{align*}
  \Lambda^\ast(\ell) :&= \sup_{\lambda\in\Rfour}\left[ \lambda\cdot\ell - \Lambda(\lambda) \right],
\end{align*}
and $\ueps$ satisfies a large deviation principle with good rate function
\begin{align*}
  \Iueps(\ell) :&= \inf_{z\in g^{-1}\{\ell\}}\Lambda^\ast(z).
\end{align*}
\label{theorem:LDP_for_parameterized_problems}
\end{theorem}

\begin{proof}
We will condition on $\xi$ and then approximate $a(x)$ by a piecewise constant function so that $a(x)$ appears only as a parameter.  
Now
\begin{align*}
  \Expxi{e^{\eps^{-1}\lambda\cdot\Zeps}} &= \Expxi{\exp\left\{ \eps^{-1}\int_0^1\frac{\lambda\cdot \bfH(s)}{\Aeps(s)}\ds \right\}}.
\end{align*}
So for $N<\eps^{-1}<N+1\in\Nat$
\begin{align*}
  \eps\log\Expxi{e^{\eps^{-1}\lambda\cdot \Zeps}} &= \eps N\frac{1}{N}\sum_{n=1}^N \log\Expxi{e^{\lambda\cdot X_N(n)}} + \eps \log\Expxi{e^{\lambda\cdot Y_N}},\\
  X_N(n) :&=\int_{n-1}^n \frac{\bfH(\soverN)}{a(\soverN) + \nu_b\theta},\quad Y_N := \int_N^{\eps^{-1}}\frac{\bfH(\soverN)}{a(\soverN)+\nu_b\theta}.
\end{align*}
Since $\nu_2^{-1}\leq(a+\nu_b\theta)^{-1}\leq\nu_1^{-1}$ we also have upper and lower bounds on $\Expxi{\exp\left\{ \lambda\cdot Y_N \right\}}$.  Therefore $\eps\log\Expxi{\exp\left\{ \lambda\cdot Y_N \right\}}\to0$.  Using also the fact $\eps N\to1$ we have
\begin{align*}
  \lim_{\eps\to0}\eps\log\Expxi{e^{\eps^{-1}\lambda\cdot \Zeps}} &= \lim_{N\to\infty} \frac{1}{N}\sum_{n=1}^N \log\Expxi{e^{\lambda\cdot X_N(n)}}. 
\end{align*}
We henceforth study the limit on the right.

We can approximate 
\begin{align*}
  X_N(n) &\leq \frac{\HNupper(n)}{\aNlower(n) + \nu_b\theta_n} \sim \HNupper(n)V_{\aNlower(n)}, \\
  \HNupper(n) :&= \max_{n-1\leq s<n}\bfH(\soverN), \quad \aNlower(n) := \min_{n-1\leq s<n}a(\soverN).
\end{align*}
We thus have
\begin{align}
  \eps\log\Exp{\lambda\veps/\eps} &\leq \frac{1}{N} \sum_{n=1}^N \Lambda(V_{\aNlower(n)}, \lambda\cdot\HNupper(n)).
  \label{eq:independent_media_pre_asymptotic_upper_bound2}
\end{align}
Similarly we can choose $\HNlower(n)$ and $\aNupper(n)$ to provide a lower bound.  Together they yield
\begin{align*}
  \frac{1}{N} \sum_{n=1}^{N}\Lambda(V_{\overline a_{N-1}(n)}, \lambda\cdot\underline \bfH_{N-1}(n)) &\leq \frac{1}{N}\sum_{n=1}^N\log\Expxi{e^{\lambda\cdot X_N(n)}}\\
  &\leq \frac{1}{N}\sum_{n=1}^N\Lambda(V_{\aNlower(n)}, \lambda\cdot\HNupper(n)).
\end{align*}

The ellipticity bounds imply that $(\alpha,\lambda)\mapsto \Lambda(V_\alpha,\lambda)$ is $C^\infty$ for $\alpha$ taking values in the closure of $\range(a)$.  
Therefore, 
\begin{align*}
  \Lambda(V_{\aNlower(n)}, \lambda\cdot\HNupper(n)) &\to \Lambda(V_{a(\noverN)}, \lambda\cdot \bfH(\noverN)),\quad N\to\infty.
\end{align*}
And thus (after extending $\aNlower$, $\HNupper$ to be piecewise constant), and using the continuity of $s\mapsto a(s,\xi)$ and $\bfH(s)$ we have
\begin{align*}
  \frac{1}{N}\sum_{n=1}^N\Lambda(V_{\aNlower(n)}, \lambda\cdot\HNupper(n)) &=\int_0^1 \Lambda(V_{\aNlower(sN)}, \lambda\cdot\HNupper(sN))\ds \\
  &\to \int_0^1 \Lambda(V_{a(s)}, \lambda\cdot \bfH(s))\ds.
\end{align*}
The same holds for the lower bound.  We have thus shown
\begin{align}
  \lim\eps\log e^{\eps^{-1}\lambda\cdot\Zeps} &= \int_0^1\Lambda(V_{a(s)}, \lambda\cdot \bfH(s))\ds.
  \label{eq:cramer_functional_12}
\end{align}
Since $\Exp{\exp\left\{ \lambda V_\alpha \right\}}$ is finite for all $\alpha$, the hypothesis of the G\"artner-Ellis theorem are trivially satisfied.  Recalling the definition of $V_\alpha$, $\Lambda$, we see that the theorem is proved.
\end{proof}

\subsubsection{Convolved media with no uniform (in $\omega$) lower bound}
\label{subsubsection:convolved_media}
Here we obtain a large deviation principle for dependent media given by a convolution of random variables that, while being positive, have no uniform lower bound.  Convolution provides a convenient way to generate dependencies.  

To avoid additional technicalities, we restrict our attention to families of functions indexed by $\eps$ such that $\eps^{-1}\in\Nat$.  Starting from this result, using the notion of \emph{exponential equivalence}, and adding an assumption of a finite logarithmic moment generating function, it would be possible to prove an LDP for general $\eps\in(0,\infty)$.

Define
\begin{align}
  \begin{split}
    \Aeps(x) &= A(\sovereps), \quad\mbox{where}\quad 
    \frac{1}{A(s)} := \sum_{n=1}^\infty \one_{[n-1, n)}(s) \gamma_n, \\
    \gamma_n :&= \sum_{m=-\infty}^\infty h_{n-m}\beta_m,\qquad
    h_n\geq0, \quad\|h\|_1 := \sum_kh_k < \infty,
  \end{split}
  \label{eq:Aeps_convolved}
\end{align}
and the $\{\beta_m\}_{m=-\infty}^\infty$ are non-negative i.i.d. random variables each depending on the same random vector (parameter) $\xi$.
The random variables $\gamma_n$ are well defined so long as the characteristic function $\phi_M(t):=\prod_{|k|<M}\Exp{e^{ith_{k}\beta}}$ has a continuous limit $\phi(t)$.  We take this as an assumption and proceed.

In this case, the large deviation principles for $\Zeps$, $\ueps$ are a direct result of lemmas \ref{lemma:technical_limit}, \ref{lemma:limiting_cramer_functional} (below), the G\"artner-Ellis theorem, and the contraction principle.
\begin{theorem}
  With $g$, $\Zeps$, and $\bfH(s)$ given by \eqref{eq:Zeps_definition}, and \eqref{eq:parameterized_media_definition}, define
  \begin{align*}
    \Lambda(\lambda) :&= \int_0^1\Lambda_{\beta(\xi)}(\|h\|_1\lambda\cdot \bfH(s))\ds,
    %\lim\eps\Lambda(\eps^{-1}\brac{F/\Aeps},\lambda) &= \Lambda(\lambda) := \int_0^1 \log\Expxi{\exp\left\{\frac{\lambda F(s)}{a(s) + b(0)}\right\}}\ds.
  \end{align*}
  then $\Lambda\in C^\infty(\Rfour)$ is a convex function such that when $\Aeps$ is defined by \eqref{eq:Aeps_convolved}
  \begin{align*}
    \lim\eps\log \Exi e^{\eps^{-1}\lambda\cdot\Zeps} &= \Lambda(\lambda).% := \int_0^1 \log\Expxi{\exp\left\{\frac{\lambda F(s)}{a(s) + b(0)}\right\}}\ds.
  \end{align*}
  If in addition $\Lambda$ is steep (see proposition \ref{proposition:steepness_criteria}), then for fixed $\xi$, $\Zeps$ satisfies a large deviation principle with good convex rate function 
  \begin{align*}
    \Lambda^\ast(\ell) :&= \sup_{\lambda\in\Rfour}\left[ \lambda\cdot\ell - \Lambda(\lambda) \right],
  \end{align*}
  and $\ueps$ satisfies a large deviation principle with good rate function
  \begin{align*}
    \Iueps(\ell) :&= \inf_{z\in g^{-1}\{\ell\}}\Lambda^\ast(z).
  \end{align*}
  \label{theorem:LDP_for_convolved_media}
\end{theorem}

\begin{lemma}
  Let $\gamma_n := \sum_kh_{n-k}\beta_k$ with non-negative $\{h_k\}\in\ell^1$ and non-negative $\beta_k$ \iid with the same law as $\beta$.  For $G\in L^\infty([0,1])$, define
  \begin{align*}
    \SNhat :&= \int_0^1 G(s)\gamma_{\lfloor sN\rfloor}\ds.
  \end{align*}
  Then for $\lambda\in\Rone$ the following limit exists in $(-\infty,\infty]$
  \begin{align*}
    \lim \frac{1}{N}\Lambda(N\SNhat,\lambda) &= \int_0^1\Lambda_\beta(\lambda G(s)\|h\|_1)\ds,
  \end{align*}
  \label{lemma:technical_limit}
  with $\Lambda_\beta$ as in definition \ref{definition:logarithmic_MGF}.
\end{lemma}
\begin{proof}[Proof of lemma \ref{lemma:technical_limit}]
For a result involving sums of convolved stationary random variables see \cite{Burton_Large_1990}.  The key difference is that here we allow the moment generating function to be infinite, and the term $G(s)$ makes the sum non-stationary.  This necessitates a new proof technique that relies on non-negativity of the $\beta_k$ and the $h_k$.

\begin{align*}
  N\SNhat = N\int_0^1G(s)\gamma_{\lfloor sN\rfloor}\ds &= \sum_{n=1}^N\left[ \int_{n-1}^nG(s/N)\ds \right]\gamma_n\\
  &= \sum_{n=1}^N\left[ \int_{n-1}^nG(s/N)\ds \right]\sum_{k\in\Zint}\beta_kh_{n-k}.
\end{align*}
Making the substitution $j=n-k$ we have 
\begin{align*}
  N\SNhat&= \sum_{k\in\Zint}\beta_kH_{N,k},\quad H_{N,k} := \sum_{j=1-k}^{N-k}h_j\int_{j+k-1}^{j+k}G(s/N)\ds.
\end{align*}
This is a sum of independent random variables.  Since $\beta_k\sim\beta$,
\begin{align*}
  \frac{1}{N}\Lambda(N\SNhat,\lambda) &= \frac{1}{N}\sum_{k\in\Zint}\Lambda_\beta(\lambda H_{N,k}) =\frac{1}{N}\int_{-\infty}^\infty \Lambda_\beta(\lambda H_{N,\lfloor s\rfloor})\ds.
\end{align*}
Using \eqref{eq:chi_squared_MGF} and changing $s\mapsto Ns$ we have
\begin{align*}
  \frac{1}{N}\Lambda(N\SNhat,\lambda) &= \int_{-\infty}^\infty \varphi_N(s)\ds,\qquad
  \varphi_N(s) := \Lambda_\beta(\lambda H_{N,\lfloor sN\rfloor}).
\end{align*}
The key to our proof is the fact that $H_{N,k}$ itself can be written as an expectation
\begin{align*}
  H_{N,k} &= \int_0^NG(s/N)h_{\lfloor s-k+1\rfloor}\ds = \Exppik{k}{\|h\|_1G(\cdot/N)\one_{[0,N]}},\\
  \Exppik{k}{f(S)} :&= \int_{-\infty}^\infty f(s)\frac{h_{\lfloor s-k+1\rfloor}}{\|h\|_1}\ds.
\end{align*}

Since $\pi_k$ is a density that concentrates near $s=k-1$, for almost every $s\in[0,1]$, $\Exppik{\lfloor sN\rfloor}{\|h\|_1G(\cdot/N)\one_{[0,N]}}\to \|h\|_1 G(s)$ (this is a result on the Lebesgue set for integrable functions, see e.g. \cite{Folland_RealAnalysis} theorem 3.20 or \cite{Dembo_Large} theorem C.13).  Also,
\begin{align}
  |H_{N,\lfloor sN\rfloor}|&\leq \|G\|_\infty \sum_{j=\lfloor 1-sN\rfloor}^{\lfloor N+1-sN\rfloor}h_j.
  \label{eq:HNk_bound}
\end{align}
Since this tends to zero for $s\notin[0,1]$, we have
\begin{align}
  \varphi_N(s) &\to\varphi(s):=\left\{
  \begin{matrix}
    \Lambda_\beta\left(\lambda\|h\|_1G(s) \right),& \emph{a.e.}\,\, s\in[0,1],\\
    0,& \emph{a.e.}\,\, s\notin[0,1].
  \end{matrix}
  \right.
  \label{eq:varphi_pointwise_limit2}
\end{align}
To prove the lemma, we thus have to show $\int\varphi_N\to\int\varphi$.

We first obtain an upper bound in cases where $\Lambda(\lambda)<\infty$.  When $\Lambda(\lambda)=\infty$, the lower bound we derive will be infinite, and thus the upper bound as well.
\begin{align}
  \begin{split}
    \int_{-\infty}^\infty\varphi_N(s)\ds &= \int_{-\infty}^\infty\Lambda_\beta\left( \Exppik{\lfloor sN\rfloor}{ \lambda\|h\|_1G(\cdot/N)\one_{[0,N]}} \right)\ds\\
    &\leq \int_{-\infty}^\infty \Exppik{\lfloor sN\rfloor}{\Lambda_\beta\left( \lambda\|h\|_1G(\cdot/N) \right)}\ds\\
    &= \int_{-\infty}^\infty\left[\int_0^N\Lambda_\beta\left( \lambda\|h\|_1G(t/N) \right)\pi_{\lfloor sN\rfloor}(t)\dt \right]\ds\\
    &= \frac{1}{N}\int_{-\infty}^\infty\left[\int_0^N\Lambda_\beta\left( \lambda\|h\|_1G(t/N) \right)\pi_{\lfloor s\rfloor}(t)\dt \right]\ds\\
    &= \int_{-\infty}^\infty\left[\int_0^1\Lambda_\beta(\lambda\|h\|_1G(t))\pi_{\lfloor s\rfloor}(tN)\dt \right]\ds\\
    &= \int_0^1 \Lambda_\beta(\lambda\|h\|_1G(t))\left( \int_{-\infty}^\infty \pi_{\lfloor s\rfloor}(tN)\ds \right) \dt\\
    &= \int_0^1 \Lambda_\beta(\lambda\|h\|_1G(t))\dt
    = \int_0^1 \varphi(t)\dt.
  \end{split}
  \label{eq:integral_upper_bound2}
\end{align}
Above, the inequality is due to Jensen's inequality and the convexity of $\Lambda_\beta$.  The change of integration order is justified by Fubini since $\varphi\in L^1$.

We now obtain a lower bound.  We will show that $\varphi_N\geq\psi$, with $\psi\in L^1$.  Then a corollary of Fatou's lemma gives us
\begin{align}
  \liminf \int \varphi_N(s)\ds &\geq \int \varphi(s)\ds = \Lambda(\lambda).
  \label{eq:Fatou}
\end{align}
With $M(t):=\Exp{e^{t\beta}}$, we note that
\begin{align}
  \varphi_N(s) &=\log M(\lambda H_{N,\lfloor sN\rfloor}) \geq \log M(-|\lambda H_{N,\lfloor sN\rfloor}|) = -\log 1/M(-|\lambda H_{N,\lfloor sN\rfloor}|).
  \label{eq:varphi_N_lower_bound}
\end{align}
Also, using $M(-|\lambda H_{N,\lfloor sN\rfloor}|)<1$ and \eqref{eq:HNk_bound} we have
\begin{align}
  \begin{split}
    &\log 1/M(-|\lambda H_{N,\lfloor sN\rfloor}|) = \log\left( 1 + \frac{1 - M(-|\lambda H_{N,\lfloor sN\rfloor}|)}{M(-|\lambda H_{N,\lfloor sN\rfloor}|)} \right)\\
    &\quad\leq \frac{|1-M(-|\lambda H_{N,\lfloor sN\rfloor}|)|}{M(-|\lambda H_{N,\lfloor sN\rfloor}|)}
    \leq \frac{|1-M(-|\lambda H_{N,\lfloor sN\rfloor}|)|}{e^{\Lambda_\beta(-\lambda\|h\|_1\|G\|_\infty)}}\\
    &\quad\leq \frac{\E|1-e^{-|\lambda H_{N,\lfloor sN\rfloor}|}|}{e^{\Lambda_\beta(-\lambda\|h\|_1\|G\|_\infty)}}
    \leq \frac{|\lambda H_{N,\lfloor sN\rfloor}|}{e^{\Lambda_\beta(-\lambda\|h\|_1\|G\|_\infty)}}\\
    &\leq e^{-\Lambda_\beta(-\lambda\|h\|_1\|G\|_\infty)}\|G\|_\infty \sum_{j=\lfloor 1-sN\rfloor}^{\lfloor N+1-sN\rfloor}h_j.
  \end{split}
  \label{eq:log_one_over_M_bound2}
\end{align}

So,
\begin{align}
  \log 1/M(-|\lambda H_{N,\lfloor sN\rfloor}|) &\leq e^{-\Lambda_\beta(-\lambda\|h\|_1\|G\|_\infty)}\|G\|_\infty \|h\|_1.
  \label{eq:psi_bound1}
\end{align}
This bound works for all $s$.  However, when $s$ is away from $[0,1]$ the summation is over the tails of $h_j$ and we can do better.  Specifically there exists $N_0$, $N_1$ such that for $N\geq N_1\geq N_0$ and $s\notin[-1,2]$
\begin{align*}
  \sum_{j=\lfloor 1-sN\rfloor}^{\lfloor N+1-sN\rfloor}h_j &\leq \sum_{j=\lfloor 1-sN_0\rfloor}^{\lfloor N_0+1-sN_0\rfloor}h_j.
\end{align*}
Therefore, using \eqref{eq:log_one_over_M_bound2}, for $s\notin[-1,2]$, $N>N_1$,
\begin{align}
  \log1/M(-|\lambda H_{N,\lfloor sN\rfloor}|) &\leq e^{-\Lambda_\beta(-\lambda\|h\|_1\|G\|_\infty)}\|G\|_\infty \sum_{j=\lfloor 1-sN_0\rfloor}^{\lfloor N_0+1-sN_0\rfloor}h_j.
  \label{eq:psi_bound2}
\end{align}
So define
\begin{align*}
  \psi(s) :&= \displaystyle\left\{
  \begin{matrix}
  -e^{-\Lambda_\beta(-\lambda\|h\|_1\|G\|_\infty)}\|G\|_\infty \|h\|_1,& \quad s\in[-1,2]\\
  -e^{-\Lambda_\beta(-\lambda\|h\|_1\|G\|_\infty)}\|G\|_\infty \sum_{j=\lfloor 1-sN_0\rfloor}^{\lfloor N_0+1-sN_0\rfloor}h_j,&\quad s\notin[-1,2].
  \end{matrix}
  \right.
\end{align*}
Then for $N>N_1$, \eqref{eq:varphi_N_lower_bound}, \eqref{eq:psi_bound1}, and \eqref{eq:psi_bound2} show that $\varphi_N\geq\psi$.  To show $\psi\in L^1$, we note that
\begin{align*}
  \int_{s\notin[-1,2]}|\psi(s)|\ds &\leq C\int_{-\infty}^\infty \sum_{j=\lfloor 1-sN_0\rfloor}^{\lfloor N_0+1-sN_0\rfloor}h_j\ds = \frac{C}{N_0}\int_{-\infty}^\infty \sum_{j=\lfloor 1-s\rfloor}^{\lfloor N_0+1-s\rfloor}h_j\ds\\
  &= C\|h\|_1<\infty.
\end{align*}
We thus obtain \eqref{eq:Fatou} and the proof is complete.
\end{proof}
Lemma \ref{lemma:technical_limit} allows us to obtain a limiting Cram\'er functional for the convolved media.
\begin{lemma}
  With $\bfH$ defined by \eqref{eq:Zeps_definition} set
  \begin{align*}
    \Lambda(\lambda):&= \int_0^1\Lambda_\beta(\|h\|_1\lambda\cdot \bfH(s))\ds
  \end{align*}
  Then 
  \begin{enumerate}
    \item[(i)] 
      Restricting attention to $\eps$ such that $\eps^{-1}\in\Nat$,
      \begin{align*}
        \lim_{\eps\to0}\eps\log\Exi e^{\eps^{-1}\lambda\cdot\Zeps} = \Lambda(\lambda).
      \end{align*}
    \item[(ii)] $\Lambda$ is lower semicontinuous.
    \item[(iii)] Let $b\in(-\infty,\infty]$ be the number satisfying $\{\lambda\st \E\exp\{\lambda \beta\}<\infty\}=(-\infty,b)$ or $=(-\infty,b]$, and let 
    $\bfH_\lambda$ be the value of $\bfH(s)$ that maximizes $\lambda\cdot\bfH(s)$.  Then 
      \begin{align*}
        \{\lambda\st \|h\|_1\lambda\cdot \bfH_\lambda< b\}&\subset\calD_\Lambda\subset\{\lambda\st \|h\|_1\lambda\cdot \bfH_\lambda\leq b\},
      \end{align*}
    with $\Lambda'$ finite throughout $\calD_\Lambda^\circ$.  
  \item[(iv)] If $\Lambda$ is steep, then $\Lambda$ is essentially-smooth.
  \end{enumerate}
  \label{lemma:limiting_cramer_functional}
\end{lemma}
\begin{proof}[Proof of lemma \ref{lemma:limiting_cramer_functional}]
  Write
  \begin{align*}
    \eps\log\Exi e^{\eps^{-1}\lambda\cdot\Zeps} :&= \eps\log\Expxi{\exp\left\{\eps^{-1}\int_0^1\frac{\lambda\cdot \bfH(s)}{\Aeps(s)}ds}\right\}.
  \end{align*}
  Applying lemma \ref{lemma:technical_limit} with $\lambda=1$ and $G = \lambda\cdot \bfH$ we obtain that $\Lambda$ is indeed the form for the limiting Cram\'er functional, thus proving (i).  

  To prove lower semi-continuity, we note that whenever $\lambda_n\to\lambda$ and $\Lambda(\lambda_n)\leq\alpha$ then
  \begin{align*}
    \alpha&\geq \lim\inf\Lambda(\lambda_n)=\lim\inf\int_0^1\Lambda_\beta(\|h\|_1\lambda_n\cdot \bfH(s))\ds \\
    &\geq \int_0^1\Lambda_\beta(\|h\|_1\lambda\cdot \bfH(s)) = \Lambda( \lambda),
  \end{align*}
  by a corollary of Fatou's lemma (using the fact that the integrand is bounded below).

  To show (iii), suppose first $\|h\|_1\lambda\cdot \bfH_\lambda<b$.  Then $\Lambda_\beta(\|h\|_1\lambda\cdot \bfH(s))$ is bounded and differentiable and therefore $\Lambda(\lambda)$ is too.  If on the other hand $b<\|h\|_1\lambda\cdot \bfH_\lambda$, then since $\bfH$ is continuous, $\Lambda_\beta(\|h\|_1\lambda\cdot \bfH(s))$ equals $+\infty$ on a set of positive measure, hence $\Lambda(\lambda)=\infty$.  We thus have our bounds on $\calD_\Lambda$ and it follows that $\calD_\Lambda^\circ$ is non-empty and $\Lambda'$ exists in $\calD_\Lambda^\circ$.

  Lastly, (iv) follows from (ii), the assumption of steepness, (iii), and the definition of essential smoothness.
  
 % Steepness is also assured.  Assume $\lambda_n\to\lambda\in\p\calD_\Lambda$ and $\Lambda(\lambda_n)\to\infty$ (with $|\lambda|$ finite).  Then consider $L(t):=\Lambda(t\lambda)$.  $L$ is twice differentiable and convex so that
 % \begin{align*}
 %   \frac{\lambda}{|\lambda|}\cdot\nabla\Lambda(t\lambda)=\frac{t}{|\lambda|}L'(t) &\geq\frac{1}{|\lambda|}\int_0^t L'(s)\ds = \Lambda(t\lambda) - \Lambda(0) = \Lambda(t\lambda)\to\infty,
 %   %\Lambda'(\lambda_n)\geq\int_0^{\lambda_n}\Lambda'(s)\ds &= \Lambda(\lambda_n)-\Lambda(0) = \Lambda(\lambda_n)\to\infty,
 %   %\Lambda'(\lambda_n) &\geq C\int_0^1\frac{1}{1-2\lambda_n\|h\|_1G(s)}\ds \geq C\int_0^1\log\frac{1}{1-2\lambda_n\|h\|_1G(s)}\ds\to\infty.
 % \end{align*}
 % as $t\to1$.  This implies $\lim_{t\to1}|\nabla\Lambda(t\lambda)| \to\infty$ as well.  Since $\nabla\Lambda$ is continuous inside $\calD_\Lambda$ we have $|\nabla\Lambda(\lambda_n)|\to\infty$ as well, showing that
 % $\Lambda$ is essentially smooth.  
\end{proof}

As hinted at by lemma \ref{lemma:limiting_cramer_functional}, steepness of $\Lambda$ is a condition that needs extra work to check.  We formulate a necessary and sufficient condition below, and then three sufficient conditions that are easy to check.
\begin{proposition}[Steepness criteria]
  Extend $\Lambda_\beta'$ to map $\Rone\to[0,\infty]$ by setting $\Lambda_\beta'(t)=\infty$ whenever $\Lambda_\beta(t)=\infty$.  Then define $K_i:\Rtwo\to\Rone$ by
  \begin{align*}
    K_1(\eta) :&= \int_0^x\Lambda_\beta'(F(s)\eta_1 + \eta_2)\ds,\quad K_2(\eta) := \int_x^1\Lambda_\beta'(F(s)\eta_1 + \eta_2)\ds.
  \end{align*}
  Then $\Lambda$ defined in lemma \ref{lemma:limiting_cramer_functional} is steep if and only if $\Lambda_\beta$ is steep and for every $\eta\in\p\calD_{K_i}^\circ$, $K_i(\eta)=\infty$, $i=1,2$.

  Moreover, with $F_M:=\max_sF(s)$, $F_m:= \min_sF(s)$, $\Lambda$ is steep whenever one of the following sufficient conditions hold:
  \begin{enumerate}
    \item $\Lambda_\beta$ is finite everywhere
    %\item $\Lambda_\beta$ is steep and $F\equiv F_M$ and $F\equiv F_m$ on open sets
    \item $F$ is piecewise $C^2$ and $\calD_{\Lambda_\beta}=(-\infty,b)$.
    \item Let $\{s_1,\dots,s_n\}$ be the points where $F(s_i)=F_M$ or $F(s_i)=F_m$.  Then there exist neighborhoods $N_i = (s_i-\delta,s_i+\delta)\cap\{s\st F_m<F(s)<F_M\}$ such that on $N_i$, $F$ admits an expansion of the form
      \begin{align*}
        F(s) :&= F(s_i) + c(s-s_i)^r + R(s-s_i),\\
        \frac{R(s-s_i)}{(s-s_i)^r}&\to0,\quad \frac{R'(s-s_i)}{(s-s_i)^{r-1}}\to0,\quad N_i\ni s\to s_i.
      \end{align*}
      Then with $r:=\min\{r_1,\dots,r_n\}$, $\Lambda$, $b=\p\calD_{\Lambda_\beta}$, $\Lambda$ is steep if 
      \begin{align*}
        \int_{b-1}^b\frac{\Lambda_\beta'(t)}{(b-t)^{(r-1)/r}}\dt = \infty.
      \end{align*}
  \end{enumerate}
  \label{proposition:steepness_criteria}
\end{proposition}
\begin{remark}
  Since $\beta\geq0$, we have $\Lambda_\beta'\geq0$, and therefore $K_i:[0,\infty]$ are well defined as integrals of functions taking values in $[0,\infty]$.
\end{remark}
\begin{proof}
  We note that the first condition ``$\Lambda_\beta$ is finite everywhere'' trivially implies that $\Lambda$ is steep.  
  If $\Lambda_\beta$ is not steep, then it is easy to construct an example showing that $\Lambda$ is not steep either.  So from now on we assume $\Lambda_\beta$ is steep but $\calD_{\Lambda_\beta}\neq\Rone$.  %This (and the non-negativity of $\beta$) imply then that $\calD_{\Lambda_\beta}=(-\infty,b)$, for some $b>0$.  
  
  We now show the necessary and sufficient condition involving the $K_i$.  Recall
  \begin{align*}
    \bfH :&= (F\one_{(0,x)},\, F,\, \one_{(0,x)},\, 1).
  \end{align*}
  To show steepness we must fix $\lambda\in\p\calD_\Lambda^\circ$, let $\calD_\Lambda^\circ\ni\lambda^n\to\lambda$ and show $|\nabla\Lambda(\lambda^n)|\to\infty$.  

  Define the function $\Gamma:\Rfour\to[0,\infty]$ by
  \begin{align*}
    \Gamma(\lambda) :&= \int_0^1\Lambda_\beta'(\|h\|_1\lambda\cdot\bfH(s))\ds.
  \end{align*}
  We claim that $\Lambda$ is steep if and only if $\Gamma(\lambda)=\infty$ for all $\lambda\in\p\calD_\Lambda^\circ$.  Indeed, if $\Gamma(\lambda)<\infty$ for some $\lambda\in\p\calD_\Lambda^\circ$, then since for $t\in(0,1)$ $t\mapsto\Lambda_\beta'(\|h\|_1t\lambda\cdot\bfH(s))$ is finite and non-decreasing (by convexity), 
  \begin{align*}
    |\nabla\Lambda(t\lambda)|&\leq \||\bfH|\|_{L^\infty}\Gamma(t\lambda)\leq \||\bfH|\|_{L^\infty}\Gamma(\lambda)<\infty,
  \end{align*}
  and therefore $\Lambda$ is not steep.  
  % the set \{t\lambda\} is a set of $\lambda^n$ approaching a boundary point.
  On the other hand suppose $\Gamma(\lambda)=\infty$ for all boundary points $\lambda$, then choose one along with a sequence $\calD_\Lambda^\circ\ni\lambda^n\to\lambda$, then by Fatou
  \begin{align*}
    \lim\inf|\nabla\Lambda(\lambda^n)|&\geq \lim\inf\Gamma(\lambda^n)\geq \Gamma(\lambda)=\infty.
  \end{align*}
  Therefore $\Lambda$ is steep.
  
  We now show that $\calD_\Gamma^\circ=\calD_\Lambda^\circ$.  
  Note that $\lambda\in\calD_\Gamma^\circ$ implies $\Lambda_\beta'(\|h\|_1\lambda\cdot\bfH(s))$ is bounded for all $s\in(0,1)$, and $\lambda\in (\lambda-\delta,\lambda+\delta)$ for some $\delta>0$.  Since $\calD_{\Lambda_\beta}^\circ = \calD_{\Lambda_\beta'}^\circ$, the same holds for $\Lambda_\beta(\|h\|_1\lambda\cdot\bfH(s))$.  Hence $\calD_\Gamma^\circ\subset\calD_\Lambda^\circ$.  A similar argument shows $\calD_\Lambda^\circ\subset\calD_\Gamma^\circ$.
  %First note that whenever $t\in\calD_{\Lambda_\beta}^\circ=\calD_{\Lambda_\beta'}^\circ$, 
  %\begin{align*}
  %  \Lambda_\beta(t) &= \int_0^t\Lambda_\beta'(s)\ds \leq t\Lambda_\beta'(t)\leq Ct\Lambda_\beta(t+\delta).
  %\end{align*}
  %The first inequality by convexity and the second holding for small enough $\delta$ by the definition of $\Lambda_\beta$.  The equality $\calD_\Gamma^\circ=\calD_\Lambda^\circ$ follows easily.

  We now have $\Lambda$ is steep if and only if $\Gamma(\lambda)=\infty$ for all $\lambda\in\p\calD_\Lambda^\circ=\p\calD_\Gamma^\circ$.  Our next step is to change variables to simplify this boundary.  To that end, note that $\Gamma(\lambda)= K_1(\|h\|_1(\lambda_1+\lambda_2), \|h\|_1(\lambda_3+\lambda_4)) + K_2(\|h\|_1(\lambda_2), \|h\|_1(\lambda_4))$, so by a change of variables $\eta = (\lambda_1+\lambda_2, \lambda_3+\lambda_4, \lambda_2, \lambda_4)$ we have
  \begin{align*}
    \Gamma(\lambda) &= K_1(\|h\|_1\eta_1,\|h\|_1\eta_2) + K_2(\|h\|_1\eta_3,\|h\|_1\eta_4).
  \end{align*}
  Changing variables again $\eta\mapsto\eta/\|h\|_1$ and taking note of the non-negativitiy of the $K_i$, we see that $\Lambda$ is steep if and only if $K_i(\eta)=\infty$ for all $\Rtwo\ni\eta\in\p\calD_{K_i}$.

  Having proved the necessary and sufficient condition, we use this to show the three sufficient conditions.  The first has already been shown.  

  As for the third sufficient condition, choose an extremal point $\sbar\in(0,x)$ and assume $\eta\in\calD_{K_1}$.  Without loss of generality, assume we have the given expansion in the open set $(\sbar-\delta,\sbar)$.  Now
  $K_1(\eta)=\infty$ if for all $\delta>0$
  \begin{align}
    \int_{\sbar-\delta}^\sbar \Lambda_\beta'(F(s)\eta_1+\eta_2)\ds=\infty.
    \label{eq:K_integral_condition}
  \end{align}
  We shall reduce this condition to the type stated in the proposition. 
 
  Our expansion gives us
  \begin{align*}
    t(s) :&= F(s)\eta_1 + \eta_2 = b - c|s-\sbar|^r + R(s-\sbar),
  \end{align*}
  for a new positive constant $c$ and a new function $R$ (differing from the old $R$ by a constant).  After possibly shrinking $\delta$ we can solve for $s(t)$.
  \begin{align*}
    s(t) &= \sbar - c^{-1/r}(b+R(s-\sbar)-t)^{1/r}.
  \end{align*}
  Therefore, \eqref{eq:K_integral_condition} holds if and only if
  \begin{align}
    \int_{\sbar-\delta}^\sbar\Lambda_\beta'(F(s)\eta_1+\eta_2)\ds &= \int_{t(\sbar-\delta)}^b\Lambda_\beta'(t)s'(t)\dt=\infty.
    \label{eq:lambdaprime_condition}
  \end{align}
  Differentiating we have
  \begin{align*}
    s'(t)\left( 1 - \frac{R'(s-\sbar)}{c^{1/r}(b+R(s-\sbar)-t)^{(r-1)/r}} \right) &= \frac{1}{c^{1/r}(b-t)^{(r-1)/r}}\frac{(b-t)^{(r-1)/r}}{(b+R(s-\sbar)-t)^{(r-1)/r}}.
  \end{align*}
  Noting that $b+R(s-\sbar)-t = c(\sbar-s)^r$, and using our hypothesis on $R$, we have positive $c_1$, $c_2$ such that
  \begin{align}
    \frac{c_1}{(b-t)^{(r-1)/r}}&\leq s'(t)\leq\frac{c_2}{(b-t)^{(r-1)/r}}.
    \label{eq:sprime_equivalent}
  \end{align}
  Due to \eqref{eq:sprime_equivalent}, \eqref{eq:lambdaprime_condition} is equivalent to
  \begin{align*}
    \int_{b-1}^b\frac{\Lambda_\beta'(t)}{(b-t)^{(r-1)/r}}\dt=\infty.
  \end{align*}
  Sufficient condition three then follows by considering all possible such points $\sbar$.

  Sufficient condition 2 follows from 3 since given 2 we have the expansion in 3 with $r\geq1$ and therefore 
  \begin{align*}
    \int_{b-1}^b\frac{\Lambda_\beta'(t)}{(b-t)^{(r-1)/r}}\dt &\leq \int_{b-1}^b\Lambda_\beta'(t)\dt = \lim_{\delta\nearrow b} \int_{b-1}^{\delta}\Lambda_\beta'(t)\dt \\
    &=\lim_{\delta\nearrow b}\Lambda_\beta(\delta) - \Lambda_\beta(b-1)=\infty,
  \end{align*}
  in light of our assumption $\calD_{\Lambda_\beta} = (-\infty,b)$ and the lower-semicontinuity of $\Lambda_\beta$, which follows from Fatou's lemma.
\end{proof}

%\begin{proof}
%Corollary \ref{corollary:limiting_cramer_functional} gives us the limiting logarithmic moment generating function.  The upper bound for precompact $\Gamma$ is a direct result of theorem 4.5.3 in \cite{Dembo_Large}.  The extension to all $\Gamma$ follows from the exponential tightness of the measure $P[\Iepshat\in\cdot]$.  The measure is exponentially tight since $0\in\calD_\Lambda^\circ$.  See the proof of theorem 2.3.6 in \cite{Dembo_Large}.  
  
%Combining corollary \ref{corollary:limiting_cramer_functional} with the G\"artner-Ellis theorem the result is immediate.  %Once we assume that $\Lambda$ is essentially smooth and lower semicontinuous, we meet the hypothesis of the G\"artner-Ellis theorem and the second conclusion follows.  %First, clearly $\calD_\Lambda$ has nonempty interior, throughout which $\Lambda'$ is finite.  $\Lambda$ is also steep by hypothesis.  Thus $\Lambda$ is essentially smooth.  Since it is also lower semicontinuous, the G\"artner-Ellis theorem applies.
%\end{proof}

\section{LDP approximations and generalizations}
\subsection{Rate functions for approximate solutions}
\label{subsection:approximate_rate_functions}
To compute (na\"ively) the rate functions for $\ueps$ requires first a four-dimensional (convex) optimization to obtain $\Lambda^\ast$, and then another four dimensional optimization to obtain $I_\ueps$.  Our corrector theory shows that one can rigorously approximate $\ueps=u_0+\veps$ in an $\sqrt{\eps}$ neighborhood of $u_0$.  Motivated by this we consider the rate function for $u_0+\veps$.  

However, a large deviation will necessarily take us outside the $\sqrt{\eps}$ neighborhood, so more discussion is in order.  Looking at terms in the expansion \eqref{eq:asymptotic_expansion} $\ueps = u_0 + \veps + R_\eps$ we see that when terms of the form
\begin{align*}
  \frac{1}{\brac{H_k/A_0}}\int_0^1\frac{H_i(s)}{\Aeps(s)}\ds
\end{align*}
are not too large, we can approximate $\ueps\approx u_0 + \veps$.  An exact rate function for $u_0+\veps$ can then be calculated.  We call this our \emph{approximate rate function} $\Itilde(\ell)$.  Note that this is equivalent to approximating $g$ (from \eqref{eq:Zeps_definition}) by some map $\tilde g$ and using the contraction principle.  It could be argued then that the inverse images $g^{-1}\{\ell\}\approx \tilde g^{-1}\{\ell\}$ when $\ell$ is sufficiently small (or some other better conditions).  Then continuity of the rate function $\Lambda^\ast(z)$ shows that $\Iueps(\ell)\approx\Itilde(\ell)$.  Since we can also represent $\veps(x)$ (for fixed $x$) in terms of an integral against a single function $G(x,s)$, we can obtain a rate function for $u_0+\veps$ without using the contraction principle.  This gives us a more explicit form, and since rate functions are unique (lemma 4.4.1 \cite{Dembo_Large}) these methods give the same result.

We present now the LDP for $u_0+\veps$.  The proof is a simplified version of the LDP proof for $\ueps$.  
\begin{proposition}
  With $G$, $\veps$ given by \eqref{eq:veps_rewritten}, we have:
  \begin{align*}
    \lim\eps\log \Exi e^{\eps^{-1}\lambda(u_0+\veps)} &= \Lambda(\lambda).% := \int_0^1 \log\Expxi{\exp\left\{\frac{\lambda F(s)}{a(s) + b(0)}\right\}}\ds.
  \end{align*}
  Where, when $\Aeps$ is the parameterized media defined by \eqref{eq:parameterized_media_definition}, and $V_\alpha$ given by \eqref{eq:Valpha_definition}, 
  \begin{align*}
    \Lambda(\lambda) :&= \lambda u_0+\int_0^1\left[-\lambda\frac{G(s)}{A_0(s)}+ \Lambda(V_{a(s)},\lambda G(s))\right]\ds,
    %\lim\eps\Lambda(\eps^{-1}\brac{F/\Aeps},\lambda) &= \Lambda(\lambda) := \int_0^1 \log\Expxi{\exp\left\{\frac{\lambda F(s)}{a(s) + b(0)}\right\}}\ds.
  \end{align*}
  and when $\Aeps$ is the convolved media defined by \eqref{eq:Aeps_convolved},
  \begin{align*}
    \Lambda(\lambda) :&= \lambda u_0+\int_0^1\left[-\lambda\frac{G(s)}{A_0(s)}\ds + \Lambda_\beta(\|h\|_1\lambda G(s))\right]\ds.
    %\lim\eps\Lambda(\eps^{-1}\brac{F/\Aeps},\lambda) &= \Lambda(\lambda) := \int_0^1 \log\Expxi{\exp\left\{\frac{\lambda F(s)}{a(s) + b(0)}\right\}}\ds.
  \end{align*}
  In either case $\Lambda$ is convex and whenever $\Lambda$ is steep $u_0+\veps$ satisfies a large deviation principle with good convex rate function 
  \begin{align*}
    \Itilde(\ell) :&= \sup_{\lambda\in\Rone}\left[ \lambda\ell - \Lambda(\lambda) \right].
  \end{align*}
  Moreover, whenever $\Itilde$ is finite in a neighborhood of $\ell>\lim\E\ueps$ then \eqref{eq:LDP_equality} holds.
  \label{proposition:approximate_LDP}
\end{proposition}
\begin{remark}
  The steepness criteria in proposition \ref{proposition:steepness_criteria} also apply here with $G(x,s)$ (written $G(s)$ when we fix $x$) replacing $F(s)$.
\end{remark}

\subsection{Generalizations to incompletely characterized media}
Gaussian corrector results, e.g. theorem \ref{theorem:corrector}, require knowledge of the first two moments of the random media (along with other ``niceness'' assumptions such as mixing and bounds on higher moments).
The question arises:  How much must be known about the random media for a large deviation result?  Here we partially answer this question and leave a thread open for future work.

LDP results for mixing random variables are available \cite{Dembo_Large,Bryc_Large_1992}.  These give existence but not an explicit form for limiting Cram\'er functional $\Lambda$, and hence only existence of an LDP (instead of an explicit form).  Restricting our attention to the case of parameterized media (section \ref{subsubsection:parameterized_media}) or convolved media (section \ref{subsubsection:convolved_media}), we ask, ``how general can the assumptions on the $\theta_k$ or $\beta_k$ be?''  In general, we cannot expect the moment generating function of our media to be known.  Notice that 
\begin{align}
  \tilde\Lambda \geq\Lambda\quad &\Rightarrow \quad\sup_{\lambda\in\Rone}\left[ \lambda\ell - \tilde\Lambda(\lambda) \right] \leq \sup_{\lambda\in\Rone}\left[ \lambda\ell - \Lambda(\lambda) \right].
  \label{eq:approximate_LDP_upper_bound}
\end{align}
In other words, if we can obtain upper bounds for the moment generating function, then we can obtain a lower bound on the rate function.  So a strategy would be:  First, write the limiting logarithmic moment generating function $\Lambda(\lambda) := \lim_{\eps\to0}\eps\Lambda(\eps^{-1}\veps,\lambda)$ in terms of the logarithmic moment generating function of the random media (e.g. $\Lambda_\beta$ in lemma \ref{lemma:limiting_cramer_functional}).  Second, find upper bounds for $\Lambda_\beta$ using e.g. Bennett's inequality (Lemma 2.4.1 in \cite{Dembo_Large}) (this bounds the moment generating function of a bounded random variable in terms of its mean and variance).  Third, a rigorous limiting upper bound is now available via \eqref{eq:approximate_LDP_upper_bound}.

\section{Model Problems}
\label{section:example_problems}
Here we explore two specific examples and give numerical results.  In both cases the right hand side $f(x) \equiv1$ for $x\in(0.45, 0.55)$ and $f(x)\equiv0$ elsewhere.  Hence $F(s)$ is piecewise smooth.  Therefore, using sufficient steepness condition 3 (proposition \ref{proposition:steepness_criteria}) our logarithmic moment generating functions will be steep (since $s\mapsto G(x,s)$ is piecewise smooth).

\subsection{Numerical results for parameterized media}
\label{subsection:independent_media}
Here we use a field that fits into the framework of section \ref{subsubsection:parameterized_media}.
This gives some control over the large deviations.

Let $\xi=(\xi_0,\dots,\xi_7)$ be $\sim_\iid\calU[-1,1]$.
We then set
\begin{align*}
  a(x,\xi) &= \max\left\{ 1 + 2\frac{1-0.75}{0.75}\sum_{m=0}^{7} \xi_mr^m\sin[(2m+1)\pi x],\,\, \frac{17}{32} \right\}.
\end{align*}

Next, let $\theta = (\theta_1,\dots,\theta_m,\dots)$ be an infinite collection of identically distributed independent rvs $\theta_i\sim\calU[-1,1]$ (which are also independent of the $\xi_i$).  Put
\begin{align*}
  b(y,\theta) &= \frac{1}{2}\sum_{m=0}^\infty \theta_m\one_{m\leq y<m+1}(y).
\end{align*}
In other words, $b(y,\theta)=\theta_m$ when $m<y<m+1$. 

We are ensured of the ellipticity condition
\begin{align*}
  \frac{1}{32} &\leq A(x,\xovereps) \leq \frac{7}{2}.
\end{align*}
The resultant media is pictured in figure \ref{fig:parameterized_diffusion_coefficient_comparision}.

\begin{figure}[h]
  \includegraphics[width=0.5\textwidth]{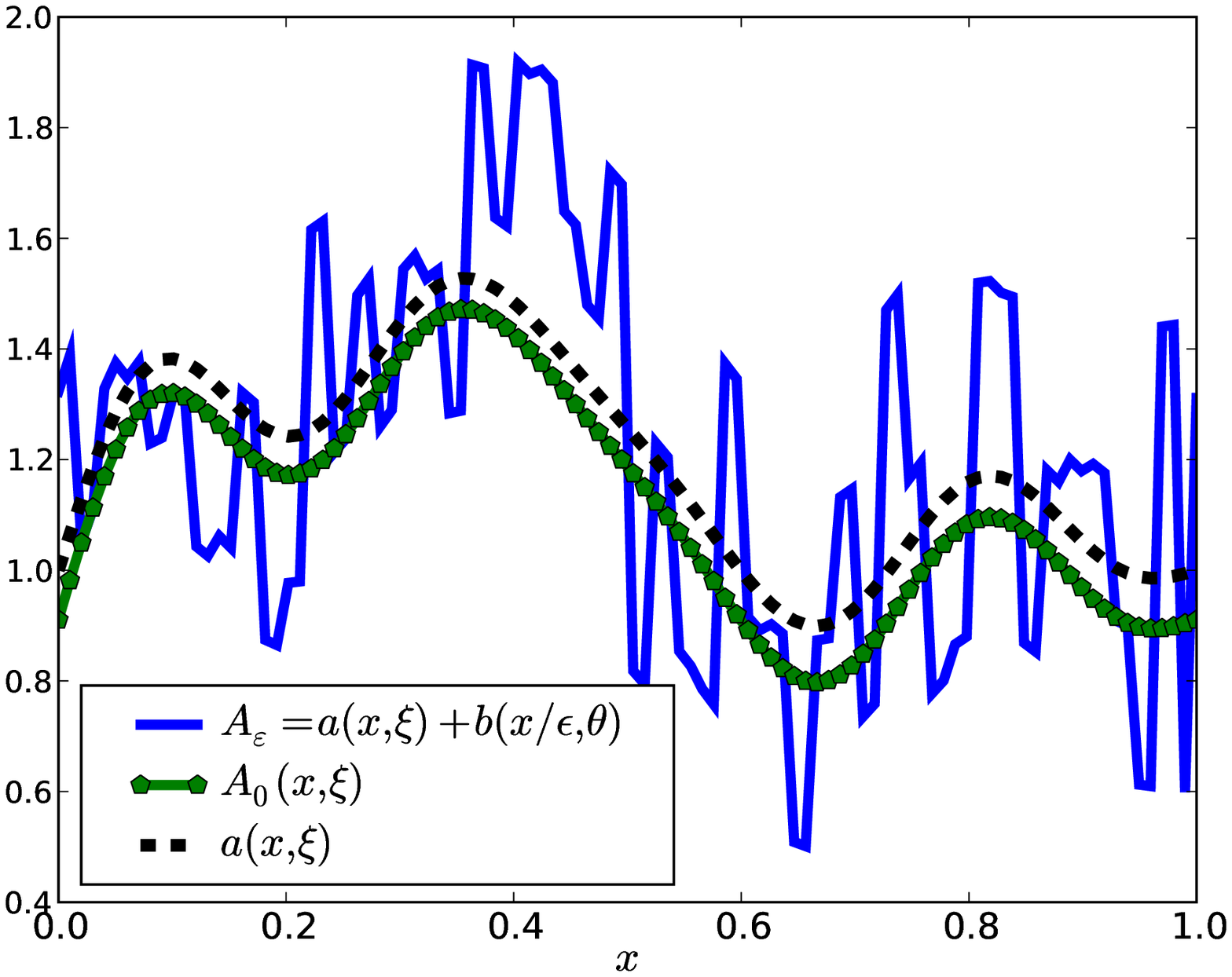}
  \includegraphics[width=0.5\textwidth]{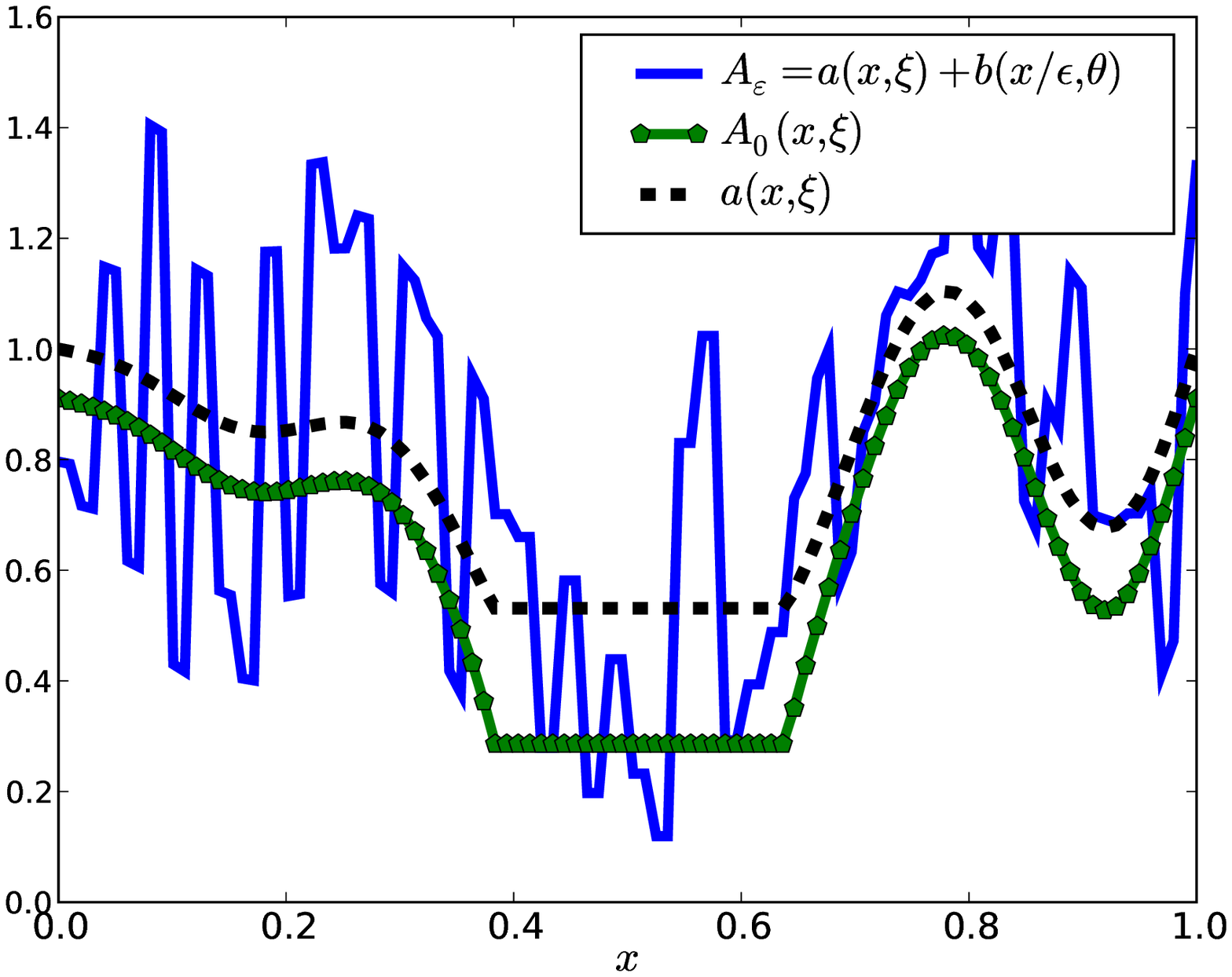}
  \caption{Comparison of two realizations of parameterized diffusion coefficients (section \ref{subsection:independent_media}).  The homogenized coefficient differs significantly only when $a(x,\xi)$ is small.  In both cases, $\eps = 1/50$.  On the left $a(x,\xi)$ is such that no values of $\theta_i\sim\calU[-1,1]$ bring $\Aeps$ close to zero.  This realization of $a(x,\xi)$ is quite typical and is referred to as our ``mild'' coefficient.  On the right we have a less common ``wild'' coefficient.}
  %\label{fig:diffusion_coefficient_comparision_seed100}
  %\caption{Comparison of one realization of diffusion coefficients.  The homogenized coefficient differs significantly only when $a(x,\xi)$ is small.  In this figure $1/\eps = 50$.  Note that since $a(x,\xi)$ is small near $x=0.5$, we expect large variance/deviations from $u_0$.  Because of that we refer to this as our ``wild'' coefficient}
  \label{fig:parameterized_diffusion_coefficient_comparision}
\end{figure}

Following as in section \ref{subsubsection:parameterized_media} we characterize
$V_\alpha := (\alpha + \theta/2)^{-1}$ with $\theta\sim\calU[-1,1]$.  We have an explicit density for $V_\alpha$,
\begin{align}
  \pi_{V_\alpha}(v) &= \one_{\frac{1}{\alpha+1/2}<V_\alpha<\frac{1}{\alpha-1/2}}(v)\frac{1}{v^2}.
  \label{eq:v_alpha_density}
\end{align}
Hence,
\begin{align*}
  \Exp{V_\alpha} :&= \log\left( \frac{\alpha+1/2}{\alpha-1/2} \right),\quad
  \Exp{V_\alpha^2} := \frac{1}{\alpha^2 - 1/4}.
\end{align*}
Therefore, 
\begin{align*}
  A_0(x,\xi) &= \Expxi{\frac{1}{\alpha + b(0)}}^{-1}\bigg\vert_{\alpha=a(x,\xi)} = \left[\log\left( \frac{\alpha+1/2}{\alpha-1/2} \right) \right]^{-1}\bigg\vert_{\alpha=a(x,\xi)},\\
  \Covalpha(\tau) &= \left\{
  \begin{matrix}
    \Exp{V_\alpha^2} - \Exp{V_\alpha}^2,& 0\leq\tau<1\\
    0,& \mbox{ otherwise},
  \end{matrix}
  \right.\\
  \sigma^2(t) &= \Covalpha\big\vert_{\alpha=\alpha(t,\xi)}.
\end{align*}
These calculations are enough to give us explicit integrals defining the homogenized term $u_0$ and the corrector $\veps$.  Namely, $u_0$ solves
\begin{align*}
  \ddx A_0(x)\ddx u_0 &= f(x),
\end{align*}
and the corrector is given by theorem \ref{theorem:corrector} with $\sigma^2(t)$ as above.

The large deviations result for $\ueps$ is given by theorem \ref{theorem:LDP_for_parameterized_problems}, and our approximate rate function by \ref{proposition:approximate_LDP}.  In particular, the limiting Cram\'er functional (theorem \ref{theorem:LDP_for_parameterized_problems}) is given by
\begin{align*}
  \Lambda(\lambda) :&= \int_0^1 \log\int_{(a(s)+1/2)^{-1}}^{(a(s)-1/2)^{-1}}\frac{e^{\lambda\cdot\bfH(s)v}}{v^2}\dv \ds,
  %\lim\eps\Lambda(\eps^{-1}\brac{F/\Aeps},\lambda) &= \Lambda(\lambda) := \int_0^1 \log\Expxi{\exp\left\{\frac{\lambda F(s)}{a(s) + b(0)}\right\}}\ds.
\end{align*}
and for the approximate rate-function, 
\begin{align*}
  \Lambda(\lambda) :&= \lambda u_0 +\int_0^1 \left[-\lambda\frac{G(s)}{A_0(s)}+ \log\int_{(a(s)+1/2)^{-1}}^{(a(s)-1/2)^{-1}}\frac{e^{\lambda G(s)v}}{v^2}\dv  \right]\ds.
  %\lim\eps\Lambda(\eps^{-1}\brac{F/\Aeps},\lambda) &= \Lambda(\lambda) := \int_0^1 \log\Expxi{\exp\left\{\frac{\lambda F(s)}{a(s) + b(0)}\right\}}\ds.
\end{align*}
%\textbf{For my notes}
%\begin{align*}
%  \int_a^b \frac{e^{\lambda v}}{v^2}\dv &= \frac{e^{\lambda a}}{a} - \frac{e^{\lambda b}}{b} + \lambda\left[ Ei(\lambda b) - Ei(\lambda a) \right],\\
%  Ei(x) :&= \int_{-\infty}^x \frac{e^t}{t}\dt.
%\end{align*}

We define empirical rate functions $\calE(\ueps(0.5), \ell)$, $\calE(u_0(0.5)+\sqrt{\eps}v(0.5),\ell)$ (for $\ueps(0.5)$ and $u_0(0.5) + \sqrt{\eps}v(0.5)$ respectively) as follows.  With $\{X_1,\dots,X_N\}$ a set of samples from $W$ ($W=\ueps$ or $W=u_0+\sqrt{\eps}v$), we set
\begin{align}
  \begin{split}
    \widehat{W} :&= \lim_{N\to\infty}\frac{1}{N}\sum_{j=1}^N X_j,\\
  \calE(W,\ell) :&= \left\{
  \begin{matrix}
    \eps\log \frac{1}{N}\sum_{j=1}^N\one_{X\geq\ell}(X_j),&\quad \ell > \widehat{W}\\
    \eps\log \frac{1}{N}\sum_{j=1}^N\one_{X\leq\ell}(X_j),& \ell<\widehat{W}.
  \end{matrix}
  \right.
\end{split}
\label{eq:empirical_rate_function}
\end{align}
Note that $\lim_{N\to\infty}\calE(u_0+\sqrt{\eps}v,\ell)$ has an explicit expression, and we use this in place of \eqref{eq:empirical_rate_function} to compute $\calE(u_0 + \sqrt{\eps}v,\ell)$.
Since $X_j\geq\ell$ is a rare event we cannot compute $\calE$ by direct sampling.  A crude importance sampling technique was used whereby the $\theta_i\sim\calU[-1,1]$ were replaced by with (scaled and shifted) Bradford random variates, 
\begin{align*}
  \pi_{Brad}(\theta) &= \frac{c}{2\log(1+c)\left( 1 + \frac{c}{2}(\theta+1) \right)}\one_{|x|<1}(\theta).
\end{align*}
As $c>0$ increases, the draws $\theta_i$ are more likely to concentrate near $-1$.  This gives a smaller diffusion coefficient and hence larger solution.  Calculation of $\calE$ must then be re-weighted by the factor $\pi_{Uniform}(\theta_i)/\pi_{Brad}(\theta_i)$.  See e.g. \cite{Caflisch_Acta_Monte,Robert_MonteCarloBook_2004} for an overview of importance sampling.
%We verify theorem \ref{theorem:convergence} in figure \ref{fig:homogenized_convergence_many_correctors} (left).  We also plot (for one realization of $\xi$) the corresponding homogenized solution and many different realizations of the corrector (figure \ref{fig:homogenized_convergence_many_correctors} right).
\begin{figure}[hp!]
  \includegraphics[width=0.5\textwidth]{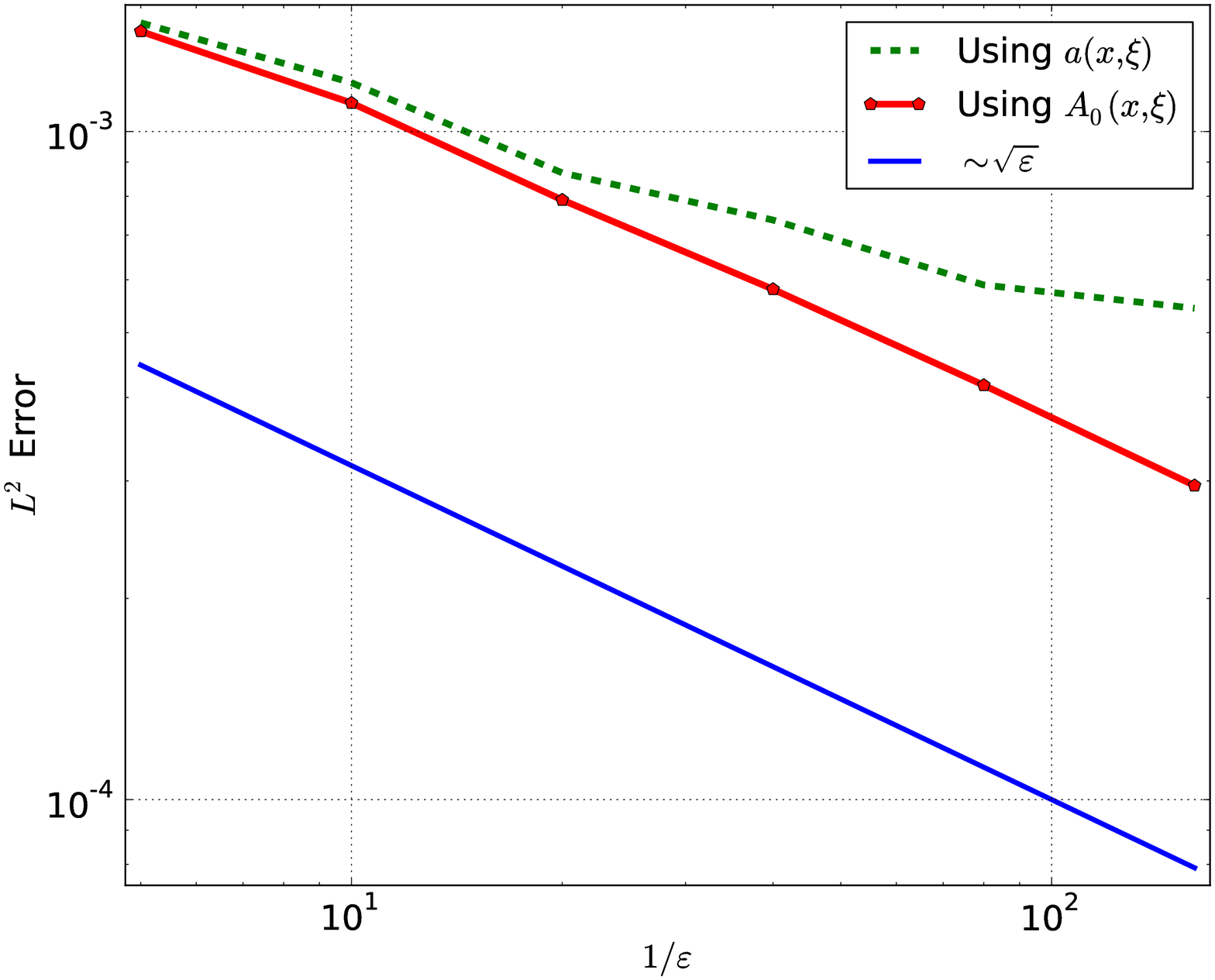}
  \includegraphics[width=0.5\textwidth]{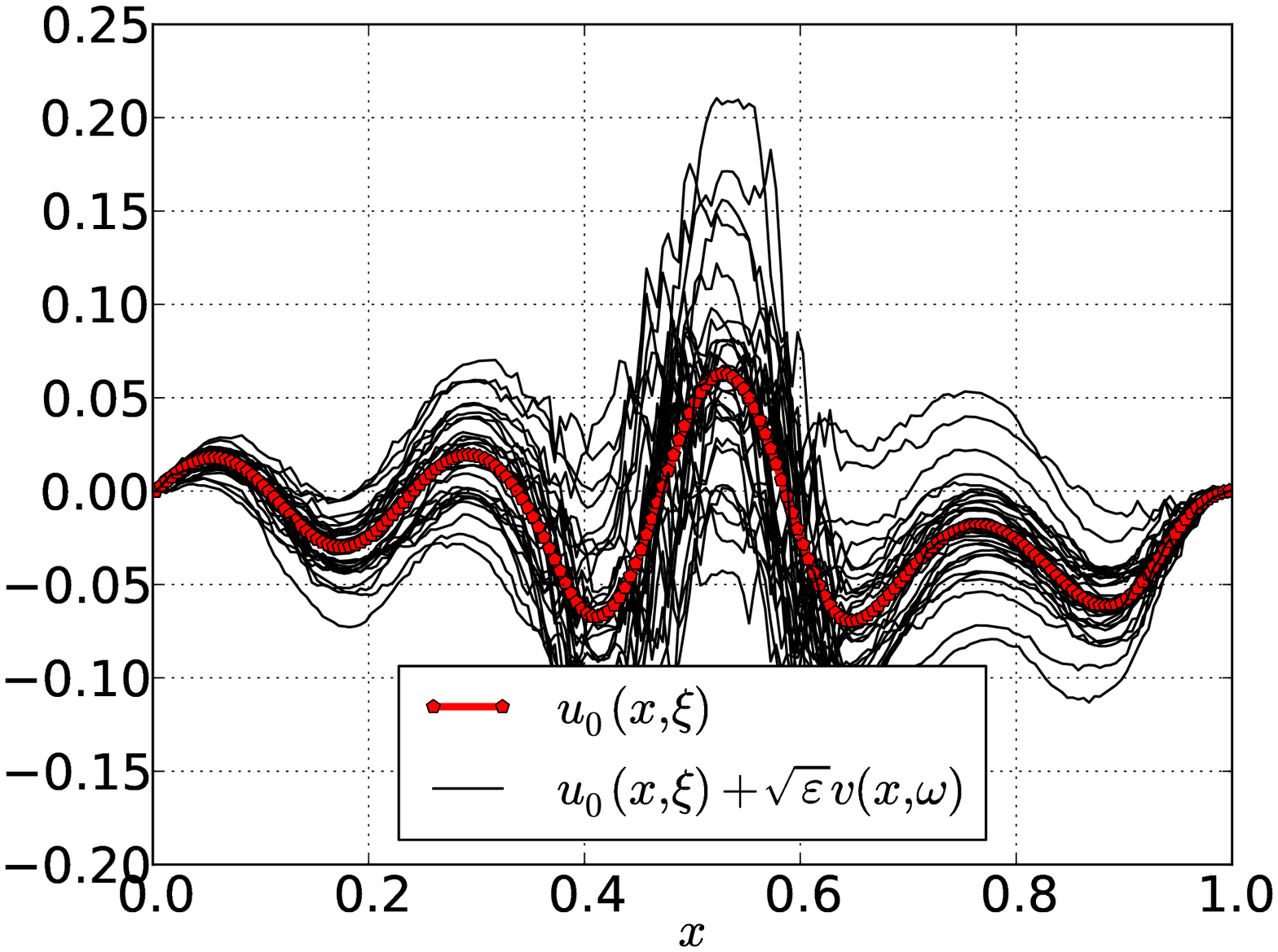}
  \caption{Left:  $L^2$ convergence of $\Exp{\|\ueps-u_0\|}$ and $\Exp{\|\ueps-u_a\|}$ where $u_a$ is the result of a truncated coefficient expansion (using $a(x,\xi)$ only).  This verifies theorem \ref{theorem:convergence}.
  Right: Homogenized solution $u_0(x,\xi)$ and many realization of the corrected solution $u_0(x,\xi)+ \sqrt{\eps}v(x,\omega)$.  For all realizations $\xi$ was fixed at the same value as the wild coefficient (figure \ref{fig:parameterized_diffusion_coefficient_comparision} right)}
  \label{fig:homogenized_convergence_many_correctors}
\end{figure}

In figure \ref{fig:parameterized_diffusion_coefficient_comparision} two realizations of the parameterized media are shown.  We will fix the low frequency part $a(x,\xi)$ and study the behavior of the solution over different realizations of $b(\xovereps,\theta)$.  The ``mild'' medium (left) has $a(x,\xi)$ far from zero, so no matter what $b(\xovereps,\theta)$ is the solution is small.  The ``wild'' medium has a section of very small $a(x,\xi)$.  In all cases, one notes that the homogenized coefficient $A_0(x,\xi)$ differs from the low-frequency coefficient $a(x,\xi)$ most when the medium is small.  In this case, $A_0<a$ in an attempt to affect a large jump in the solution $u_0$ to approximate the often large solution $\ueps$ (although $b(y,\theta)$ is symmetric about $0$, the resultant solution $\ueps$ is not symmetric about $u_0$).

We verify theorem \ref{theorem:corrector} in figures \ref{fig:parameterized_CLT} and \ref{fig:corrector_variance}.   In figure \ref{fig:parameterized_CLT} one can see that the pdf of the corrected solution (at the fixed point $x=1/2$) agrees well with the true solution so long as $\eps$ is small enough.  The fit is worse for the ``wild'' medium, and in particular the true pdf shows an asymmetry that the Gaussian corrector cannot have.  In figure \ref{fig:corrector_variance}, one sees that when $\eps\approx1/50$ (or smaller) the corrector captures the variance quite well.
\begin{figure}[hp!]
  \includegraphics[width=0.5\textwidth]{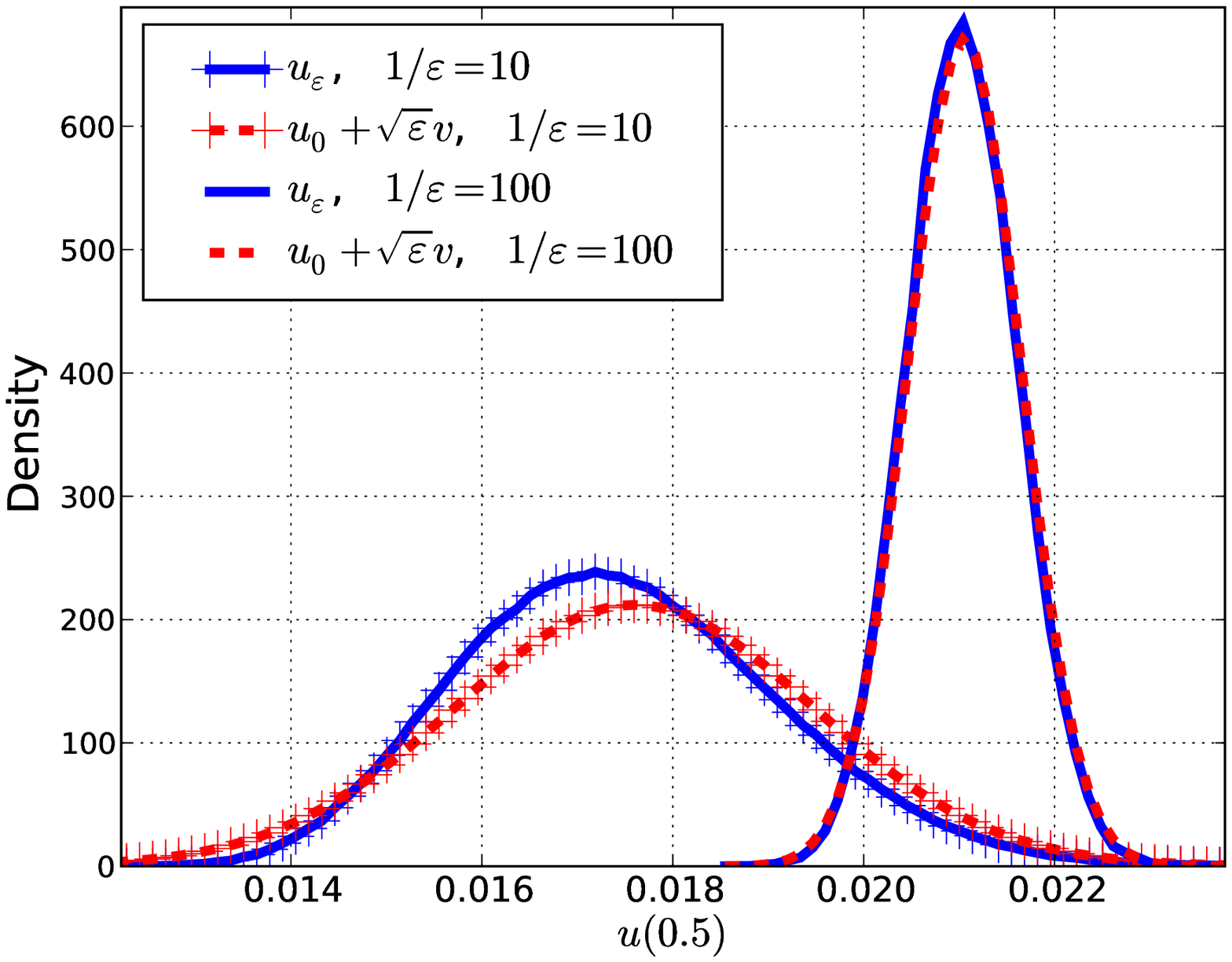}
  \includegraphics[width=0.5\textwidth]{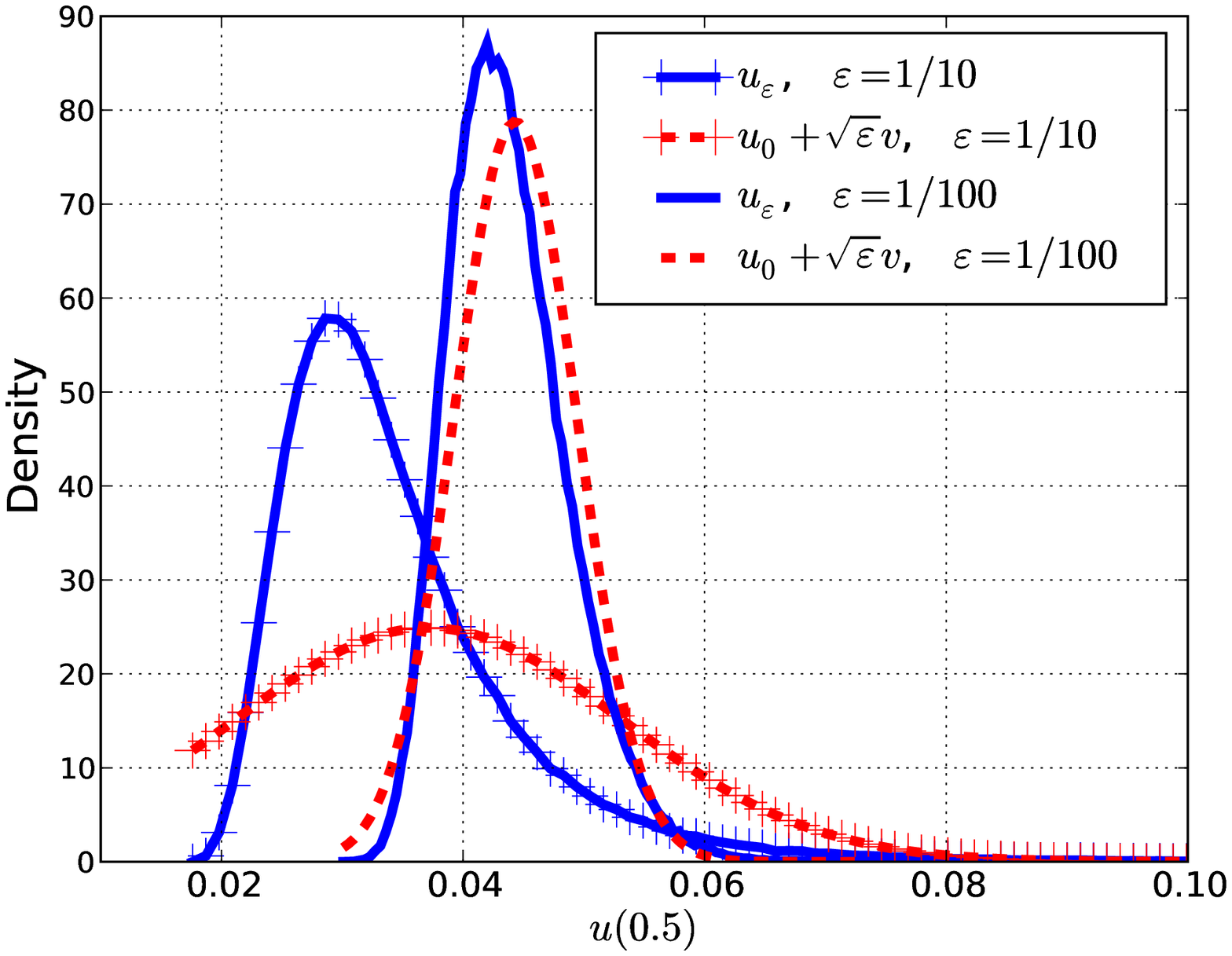}
  \caption{Left:  Observed density of $u(0.5)$ for the mild medium (figure \ref{fig:parameterized_diffusion_coefficient_comparision} left).  Plot shows that the corrector captures the bulk of the variance quite well.  Right: The wild medium (figure \ref{fig:parameterized_diffusion_coefficient_comparision} right) is shown.  Results are not as good.}
  \label{fig:parameterized_CLT}
\end{figure}
\begin{figure}[h!]
  \includegraphics[width=\textwidth]{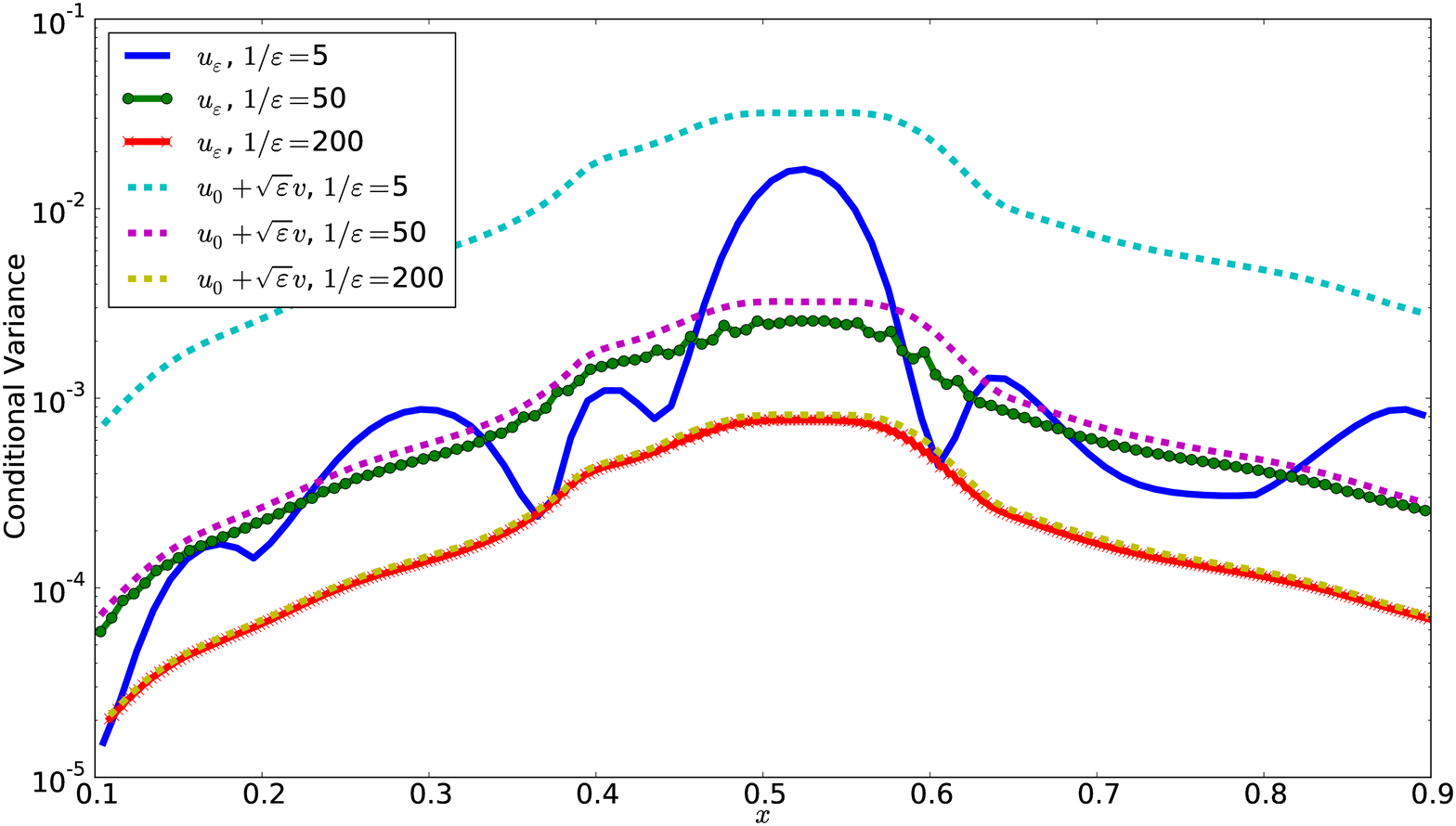}
    \caption{Verification of theorem \ref{theorem:corrector}.  Conditional variance $\Expxi{\ueps(x,\omega)^2}$ and $\Expxi{(u_0(x) + \sqrt{\eps}v(x,\omega))^2}$ for various $\eps$.  For all realizations $\xi$ was fixed at the same value as the wild coefficient (figure \ref{fig:parameterized_diffusion_coefficient_comparision} left).  Plots show good agreement when $1/\eps\geq50$.  Note also that $\Expxi{(u_0(x)+\sqrt{\eps}v(x))^2}$ is available explicitly via \eqref{eq:corrector_covariance}.  
    \label{fig:corrector_variance}}
\end{figure}
Rate functions for the mild medium are compared in figure \ref{fig:parameterized_LDP_theory_mild}.  One can see that the corrector and true rate function are almost indistinguishable until the true solution saturates around $0.35$.  The approximate rate function also works well up until $u(0.5)\approx0.3$.  In this case one could use the approximate rate function to see \emph{a-priori} that the corrector stands a chance of capturing the large-deviation behavior well.  The case is different for the wild medium LDP results in figure \ref{fig:parameterized_LDP_theory_wild}.  Here one can see that the corrector rate function separates from the true rate function fairly early on.  While the fit between the approximate rate function $\Itilde$ and the true rate function $\calE(\ueps,\cdot)$ is not perfect, one could still tell, using only $\Itilde$, that the Gaussian corrector stands little chance of capturing the large deviation behavior.  

It should be noted that since for the mild medium, the maximum possible value of $\ueps(0.5)$ was approximately $0.35$, we consider $\ell=0.03$ a large deviation.  The scale is harder to set with the wild medium since our sampling could not achieve results near the maximum.  However, one does note (figure \ref{fig:parameterized_LDP_theory_wild}) that by $\ell\approx0.6$ the empirical rate function differs from a Gaussian rate function by quite a bit.  For that reason, we consider $\ell>0.6$ to be a large deviation.  It is important to note however that the approximate rate function, being based on a linearization, does differ from the true rate function for large enough $\ell$. 
\begin{figure}[hp!]
  \includegraphics[width=0.5\textwidth]{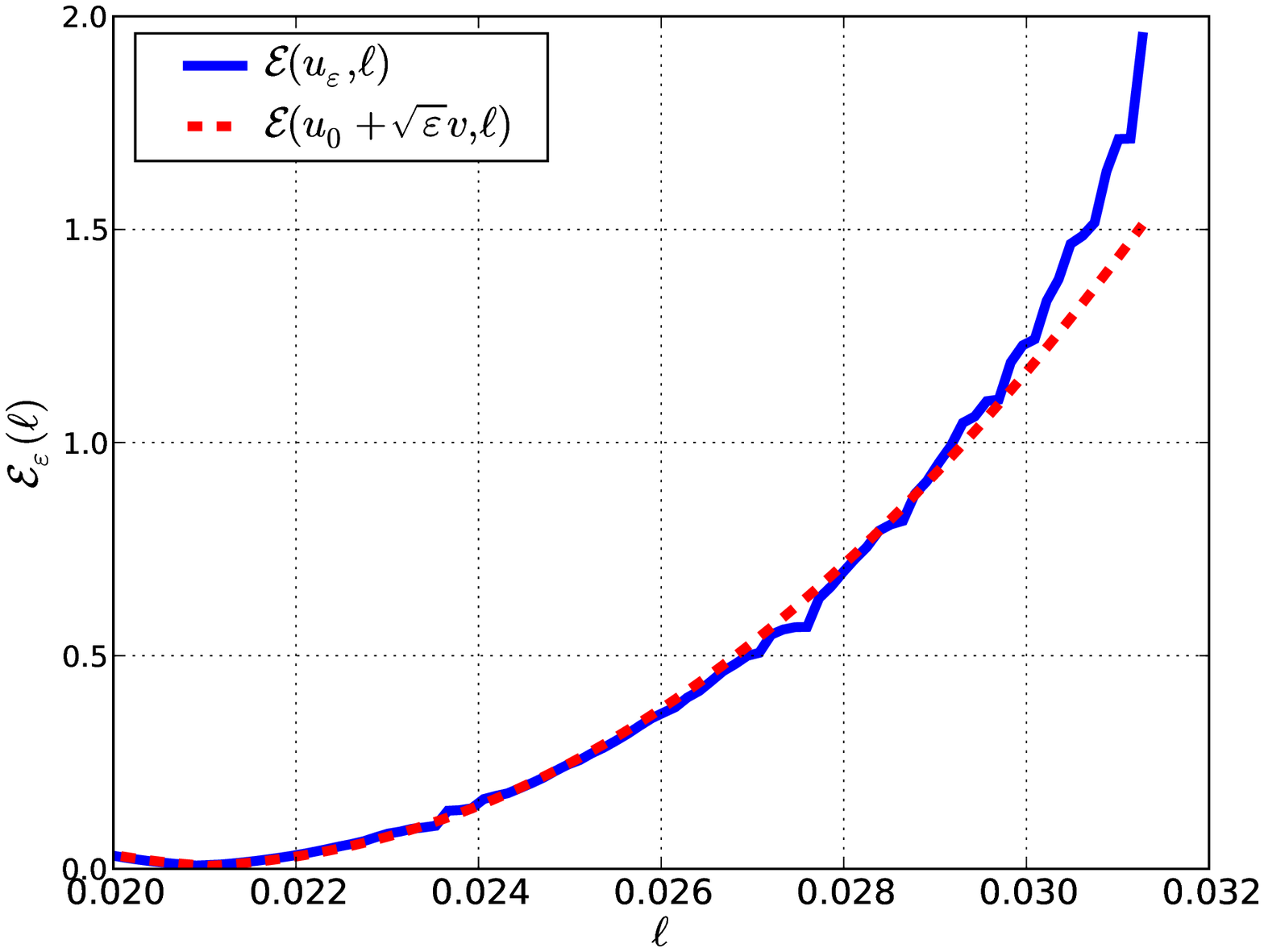}
  \includegraphics[width=0.5\textwidth]{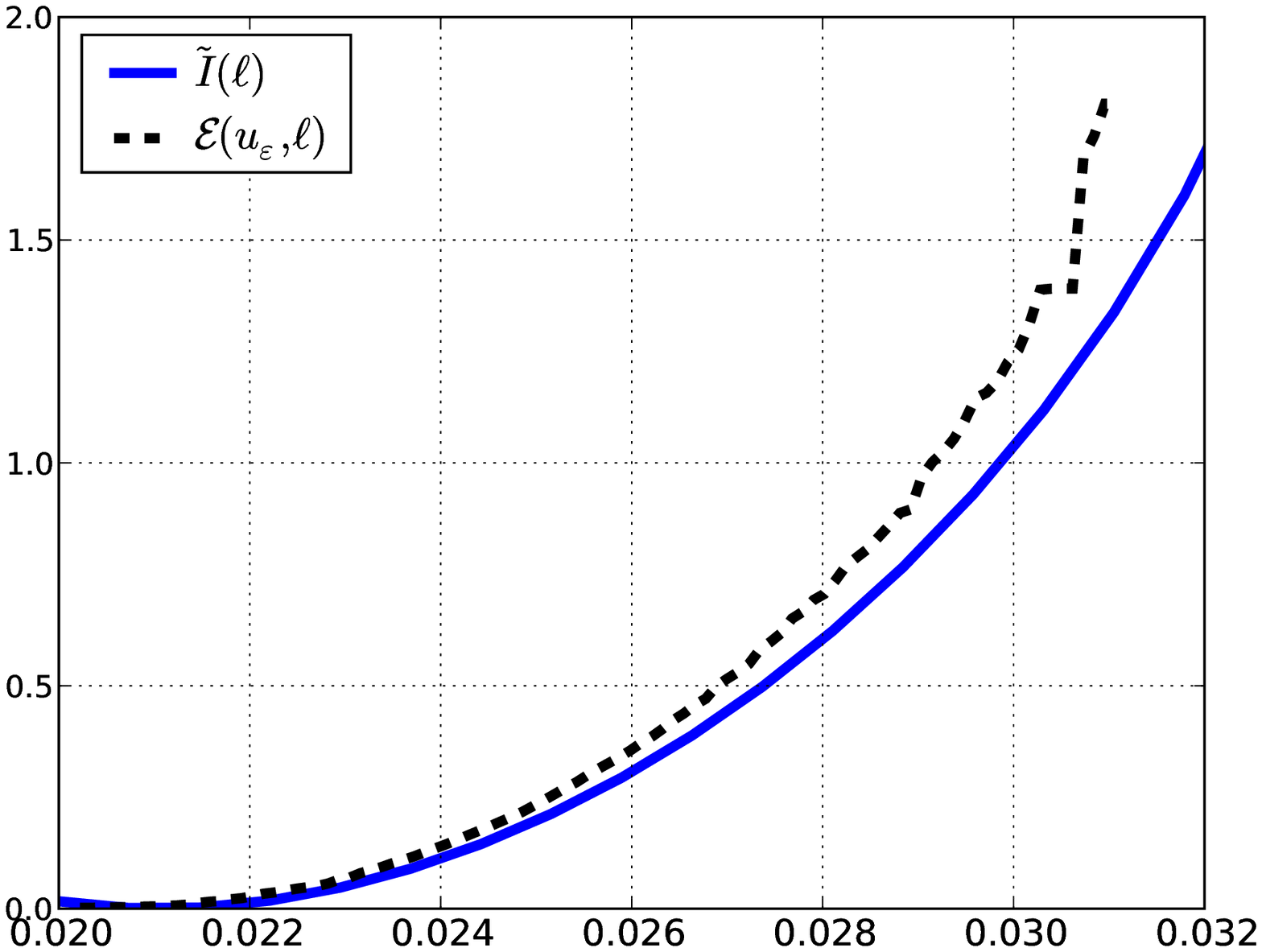}
  \caption{Left: Comparison of empirical rate function for $\ueps$, $\calE(\ueps,\ell)$, with the corrector rate function $(\ell-u_0)^2/(2C_c)$ in the case of the mild medium (figure \ref{fig:parameterized_diffusion_coefficient_comparision} left).  In this case, the corrector captures the large deviation behavior well up until $\approx0.3$ when the true solution nears its theoretical upper bound.
  Right: Comparison of theoretical approximate rate function $\Itilde(\ell)$ with empirical rate function $\calE$ with the mild medium and $\eps=1/100$.  The approximate rate function works quite well for values $\approx 0.30$.  Since the largest possible solution is $\ueps(0.5)\approx0.35$ we consider $0.03\sim O(1)$.}
  \label{fig:parameterized_LDP_theory_mild}
\end{figure}
\begin{figure}[hp!]
  \includegraphics[width=0.5\textwidth]{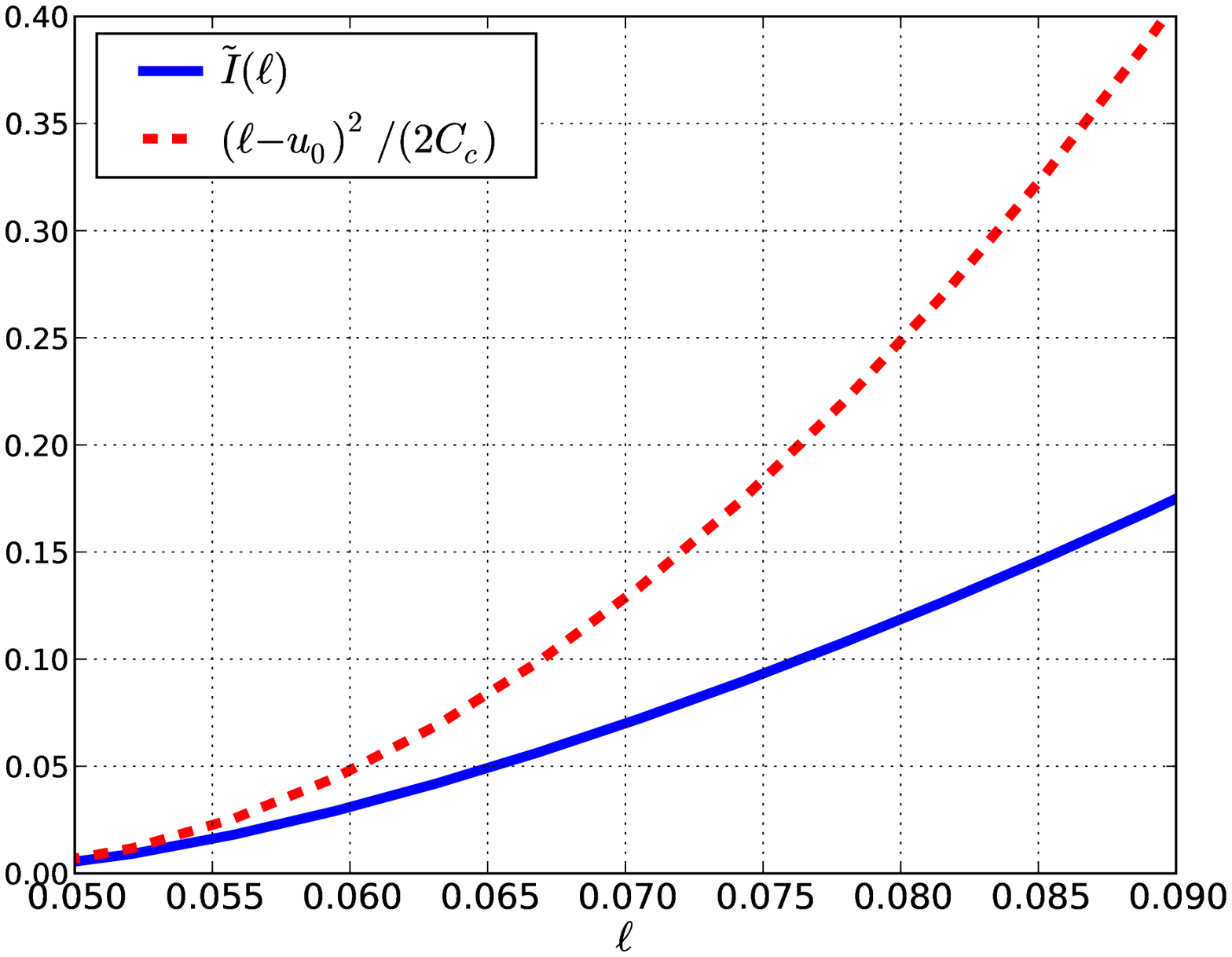}
  \includegraphics[width=0.5\textwidth]{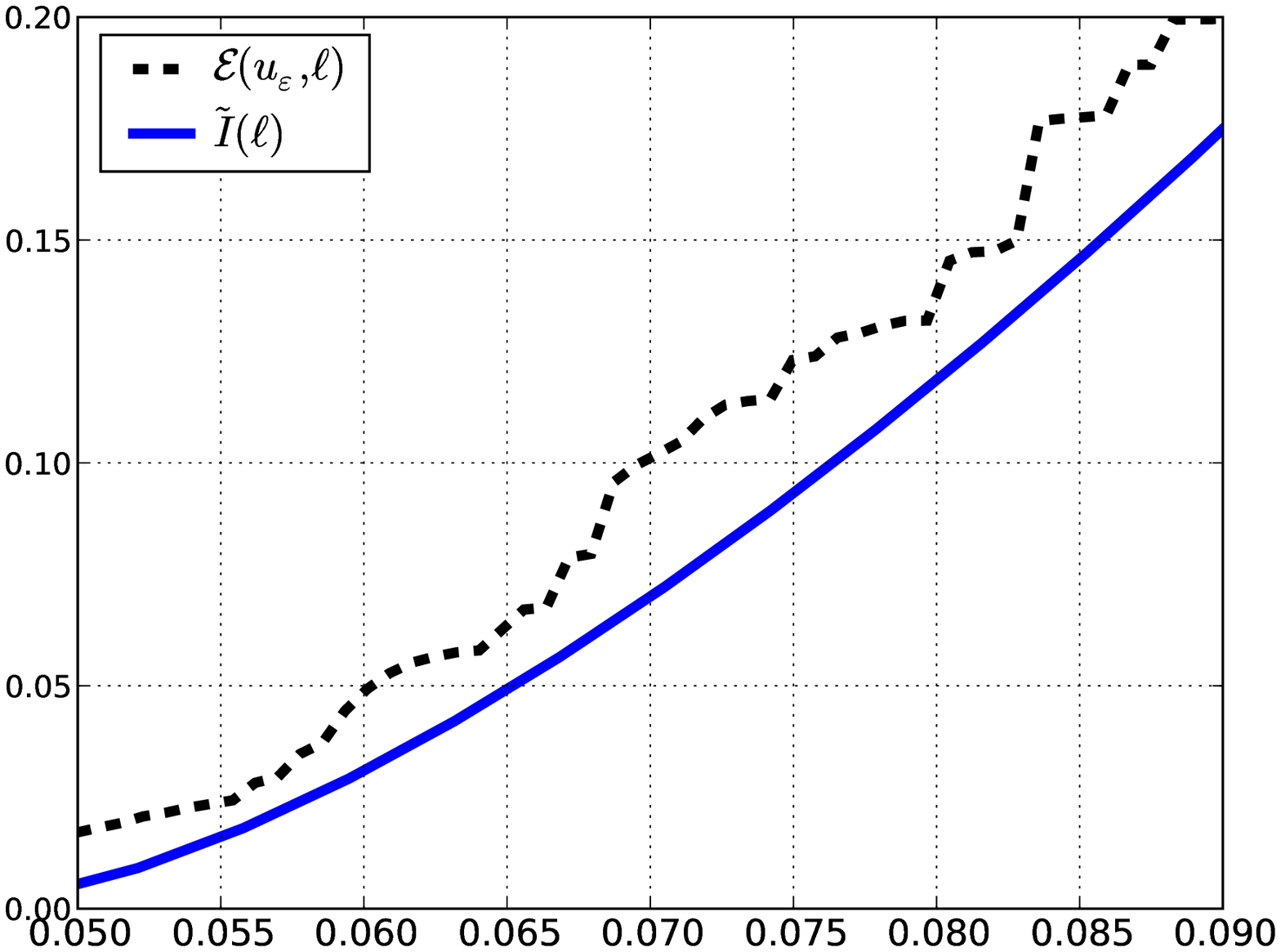}
  \caption{Left: Comparison of empirical rate function for $\ueps$, $\calE(\ueps,\ell)$, with the corrector rate function $\calE(u_0+\sqrt{\eps}v,\ell)$ in the case of the wild medium (figure \ref{fig:parameterized_diffusion_coefficient_comparision} right).  
  Right: Comparison of theoretical approximate rate function $\Itilde(\ell)$ with empirical rate function $\calE(\ueps,\ell)$ with the wild medium and $\eps=1/100$.}
  \label{fig:parameterized_LDP_theory_wild}
\end{figure}

\subsection{Numerical results for convolved media}
\label{subsection:numerical_convolved_media}
Here we implement a particular case of the media described in section \ref{subsubsection:convolved_media}. 
With $\eps^{-1}\in\Nat$, we define
\begin{align*}
  \Aeps(x) &= A(\sovereps), \quad\mbox{where}\quad 
  \frac{1}{A(s)} := \sum_{n=1}^\infty \one_{[n-1, n)}(s) \gamma_n, \\
  \gamma_n :&= \sum_{m=-\infty}^\infty h_{n-m}\beta_m,\\
  h_n&\geq0, \quad\|h\|_1 := \sum_kh_k < \infty,
\end{align*}
and the $\{\beta_m\}_{m=-\infty}^\infty$ are i.i.d. chi-squared random variables with $\xi$ degrees of freedom.  This means
\begin{align*}
  \beta_1&\sim \pi_\beta(\beta\g \xi) =  \frac{\beta^{\xi/2-1}e^{-\beta/2}}{\Gamma(\xi/2)2^{\xi/2}}.
\end{align*}
\begin{figure}[h!]
  \includegraphics[width=0.5\textwidth]{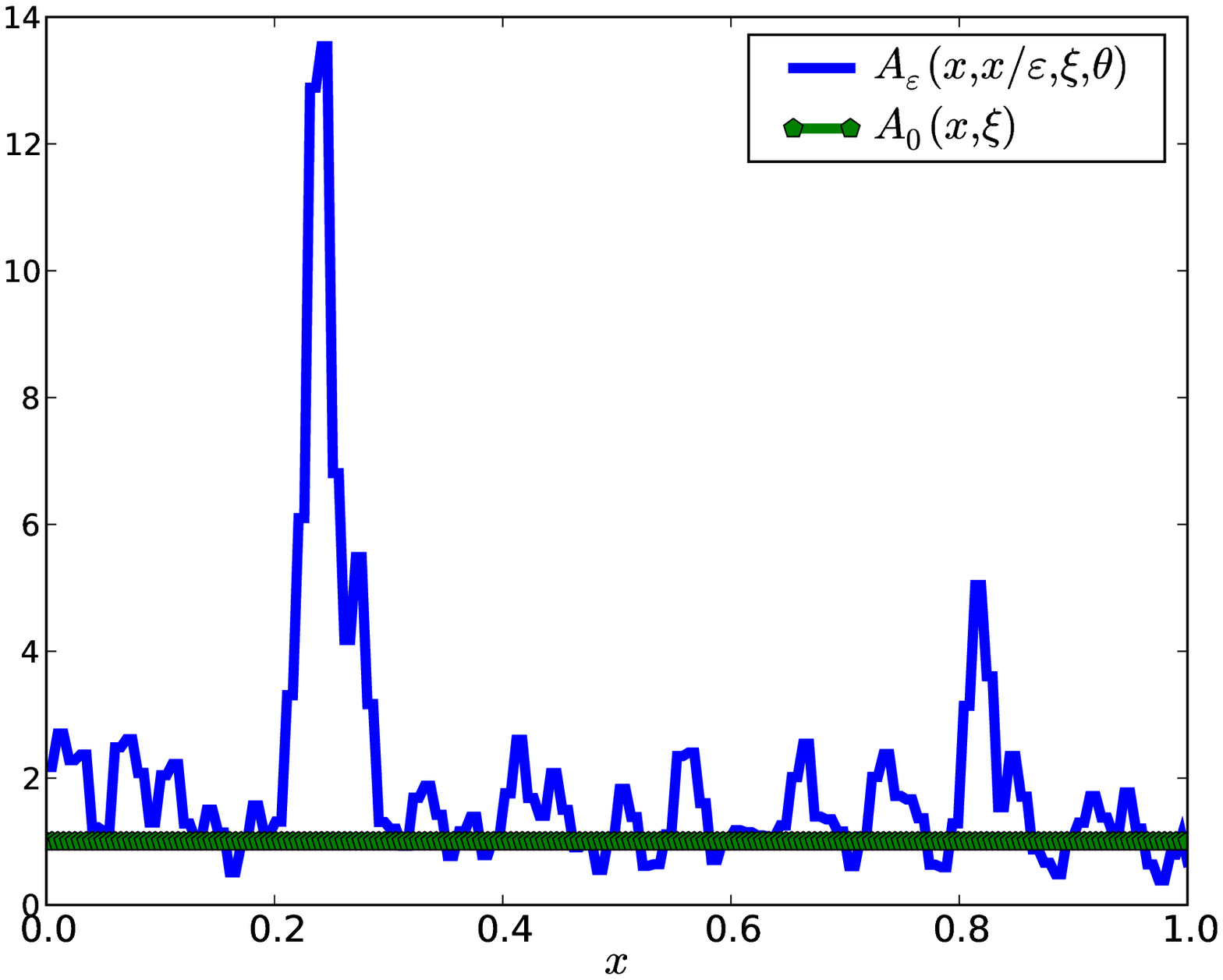}
  \includegraphics[width=0.5\textwidth]{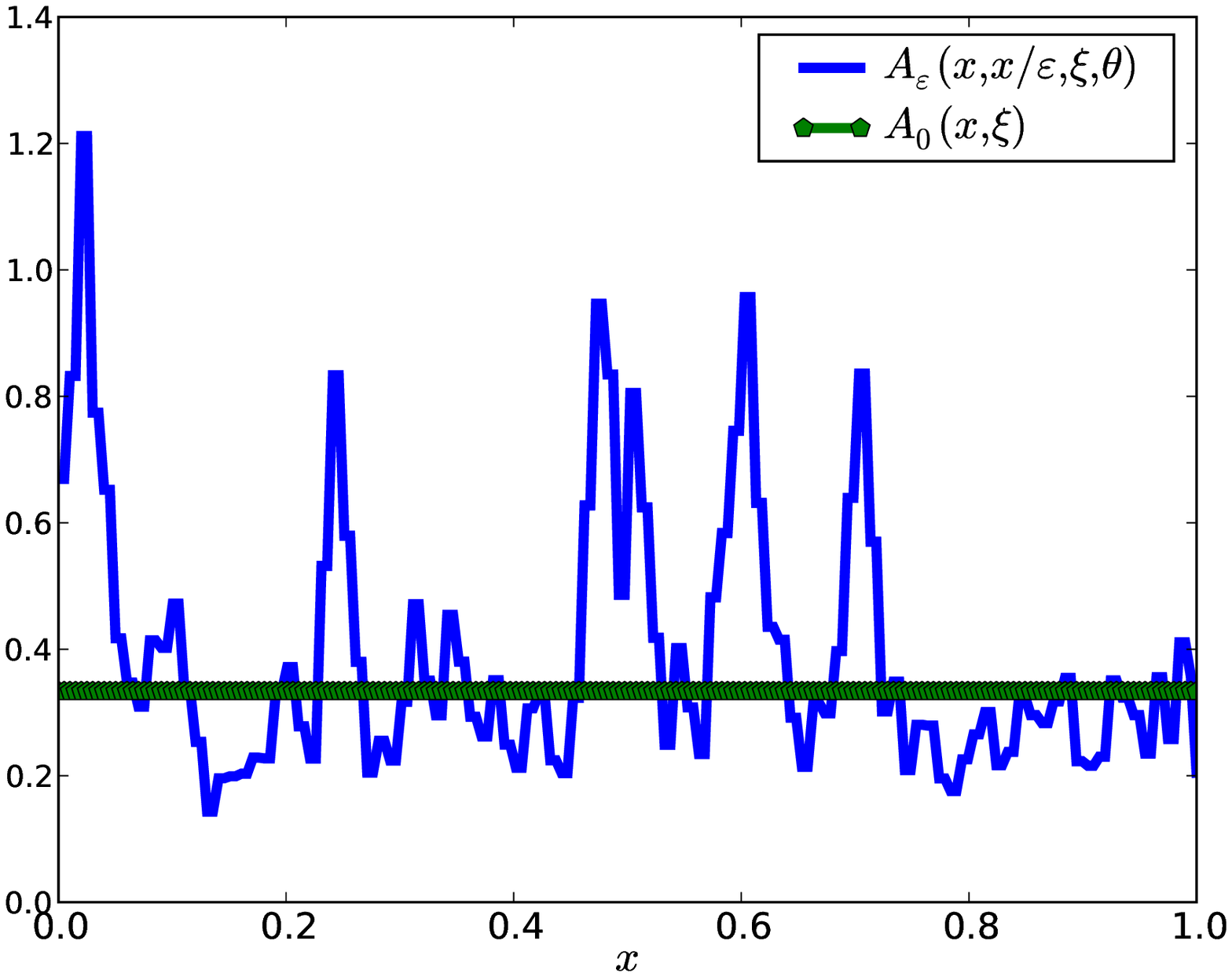}
  \caption{Typical realization of random media when $\xi=1$ (left) and when $\xi=3$ (right).}
  \label{fig:convolved_diffusion_coefficient_comparison}
\end{figure}
The moment generating and characteristic functions of every $\beta_n$ are
\begin{align}
  \Expxi{e^{\lambda\beta}} &= \frac{1}{(1-2\lambda)^{\xi/2}},  \quad \Expxi{e^{it\beta}} = \frac{1}{(1-2it)^{\xi/2}}.
  \label{eq:chi_squared_MGF}
\end{align}
The random variables $\gamma_n$ are well defined since the characteristic function $\phi_M(t):=\prod_{|k|<M}\Exp{e^{ith_{k}\beta}}$ has a continuous limit $\phi(t)$.  Indeed, 
\begin{align*}
  \log\phi_M(t) &= -\frac{\xi}{2}\sum_{|k|<M} \log\left( 1-2ith_k \right).
\end{align*}
This converges absolutely as can be seen using $|\log(1-2ith_k)|\leq C|2th_k|$ and $\|h\|_1<\infty$.

Note that $\Exi\beta_n = \xi$, $\Exi(\beta_n-\xi)^2 = 2\xi$, so 
\begin{align}
  \begin{split}
    \frac{1}{A_0(s)} :&= \Exi{\frac{1}{A(s)}}= \sum_{n=1}^\infty \one_{[n-1, n)}(s) \Exi{\gamma_n}\equiv \xi\|h\|_1\\
        \Covxi(\tau)&=\Expxi{\left( \frac{1}{A(0)}-\Exp{\frac{1}{A(0)}}\right)\left( \frac{1}{A(\tau)}-\Exp{\frac{1}{A(\tau)}}\right)}\\
                    &=2\xi\sum_{k\in\Zint}h_{-k}h_{\lfloor\tau\rfloor-k}.
    %\Expxi{\left( \frac{1}{A(s)}-\Expxi{\frac{1}{A(s)}} \right)^2}&= 2\xi \|h\|_2^2.
    %\sum_{n=1}^\infty \one_{[n-1, n)}(s) \Exi{(\gamma_n-h\xi)^2} \equiv 2\xi\|h\|_2^2.
  \end{split}
  \label{eq:mean_variance}
\end{align}
So the single random variable $\xi$ defines the coarse-scale randomness.  From realization to realization $\xi$ varies with a geometric distribution (with parameter $1/5$) i.e.
\begin{align*}
  \pi_\xi(\xi) &= \frac{1}{5}\left( \frac{4}{5} \right)^{\xi-1}.
\end{align*}

Truncation has no meaning in this context, so
we consider homogenization.  We have
\begin{align}
  A_0^{-1}=\Exp{\Aeps^{-1}}(x) &\equiv \xi\|h\|_1=\mbox{const.},
  \label{eq:A_0_chi_squared}
\end{align}
and then $u_0$ is the solution to
\begin{align*}
  -\xi\|h\|_1\frac{\d^2}{\dx^2} u_0 &= f(x), \quad u_0(0)=u_0(1)=0.
\end{align*}

Using \eqref{eq:mean_variance} we obtain
$\sigma^2(t) = 2\xi\|h\|_1^2$.  The Gaussian corrector is then given by theorem \ref{theorem:corrector}.
\begin{figure}[h]
  \includegraphics[width=0.5\textwidth]{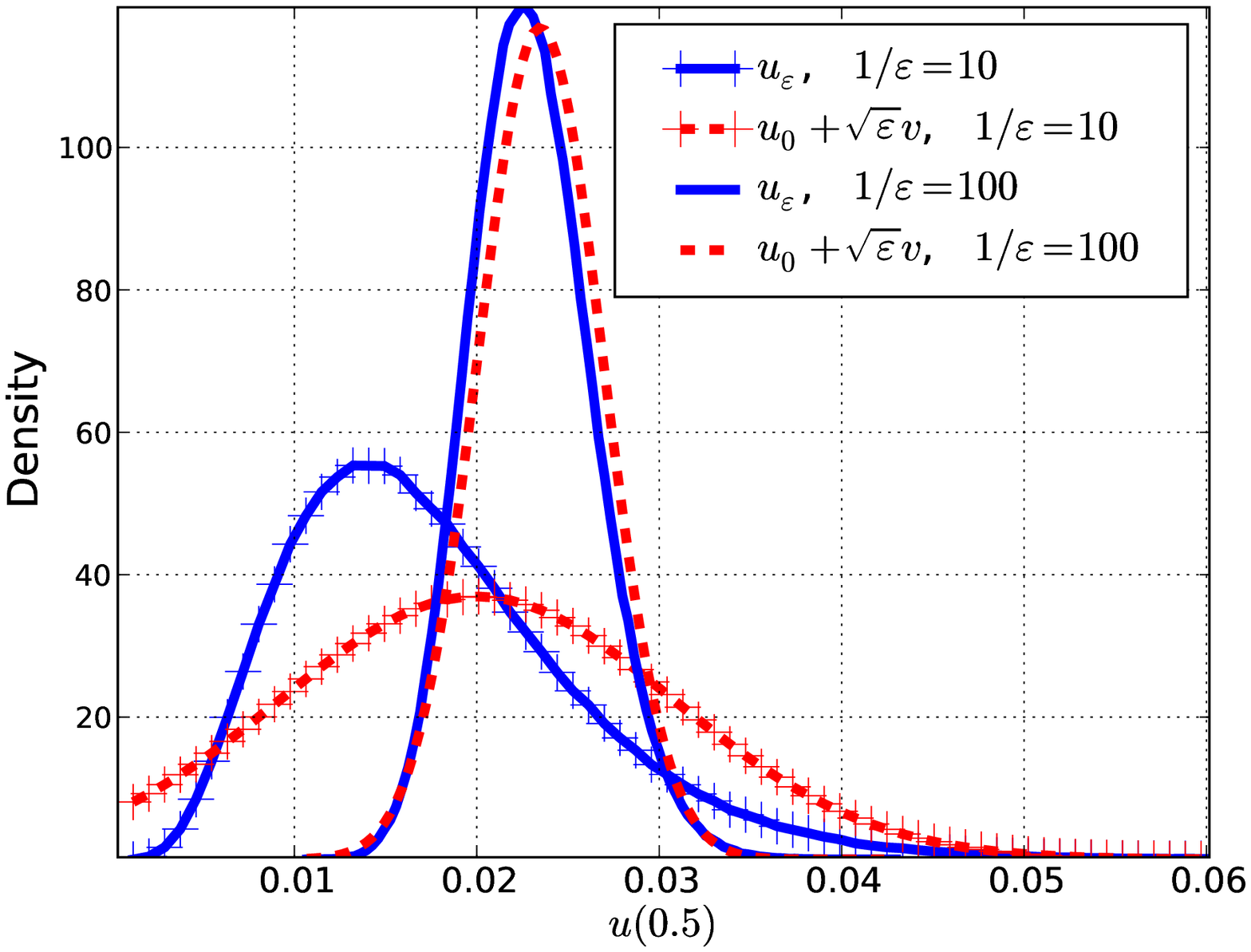}
  \includegraphics[width=0.5\textwidth]{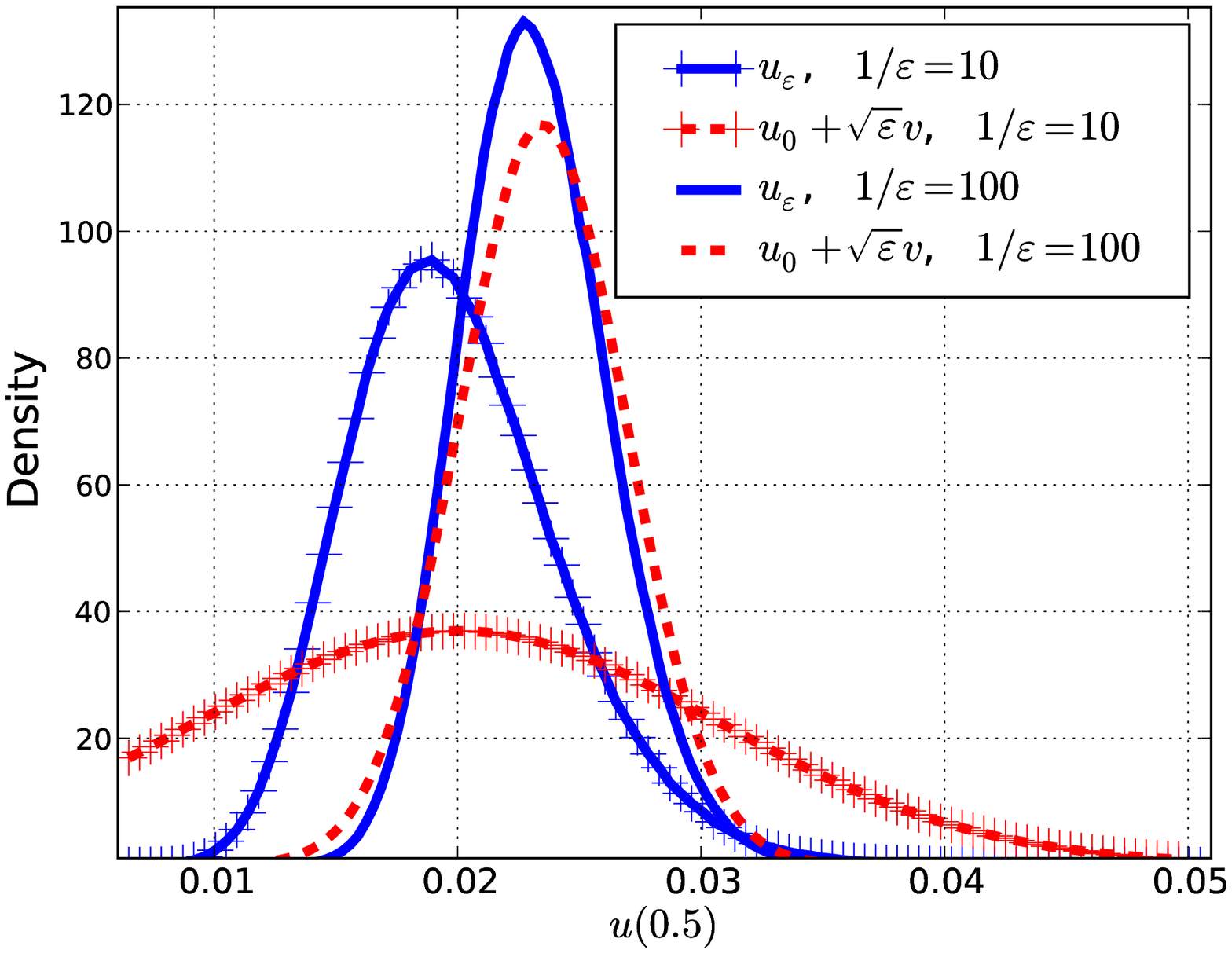}
  \caption{Pdf of $\ueps$ vs. that of $u_0 + \sqrt{\eps}v$ for two values of $\eps$ and the convolved media.  On the left, $\kappa=1$, and on the right $\kappa=10$.  In all cases $\xi=1$.  Shows good agreement once $\eps$ is small enough, although agreement is worse when $\kappa=10$.}
  \label{fig:convolved_CLT}
\end{figure}
The large deviations result and rate function is given by theorem \ref{theorem:LDP_for_convolved_media} with
\begin{align*}
  \Lambda(\lambda) :&= \int_0^1\frac{\xi}{2}\log\frac{1}{1-2\|h\|_1\lambda\cdot\bfH(s)} \ds
\end{align*}
and for the approximate LDP (proposition \ref{proposition:approximate_LDP})
\begin{align*}
  \Lambda(\lambda) :&= \lambda u_0+\int_0^1\left[ -\lambda\frac{G(s)}{A_0(s)} + \frac{\xi}{2}\log\frac{1}{1-2\|h\|_1\lambda G(s)} \right]\ds.
\end{align*}

Figure \ref{fig:convolved_diffusion_coefficient_comparison} shows typical realizations of the diffusion coefficient when the media ``building block'' $\beta$ has $\xi=1$ or $\xi=3$ degrees of freedom.  More degrees of freedom means larger $\ueps$.  The behavior however is not analogous to the ``mild/wild'' comparison of section \ref{subsection:independent_media}.  In particular, the corrector captures the bulk of the distribution (moderate deviations) for all values of $\xi$ so long as $\eps$ is small enough.  For this reason we only picture pdfs for $\xi=1$ (figure \ref{fig:convolved_CLT}).  This is expected since a $\chi^2_\xi$ random variable behaves similar to a Gaussian random variable $\calN(\xi,2\xi)$ when $\xi$ is large.  The media correlation length $\kappa$ does effect the results, and in particular $\kappa$ must be much smaller than $\eps^{-1}$ for the corrector theory to work well.  This is expected since sums of highly correlated random variables tend to a Gaussian at a slower rate than independent ones.
\begin{figure}[h!]
  \includegraphics[width=0.5\textwidth]{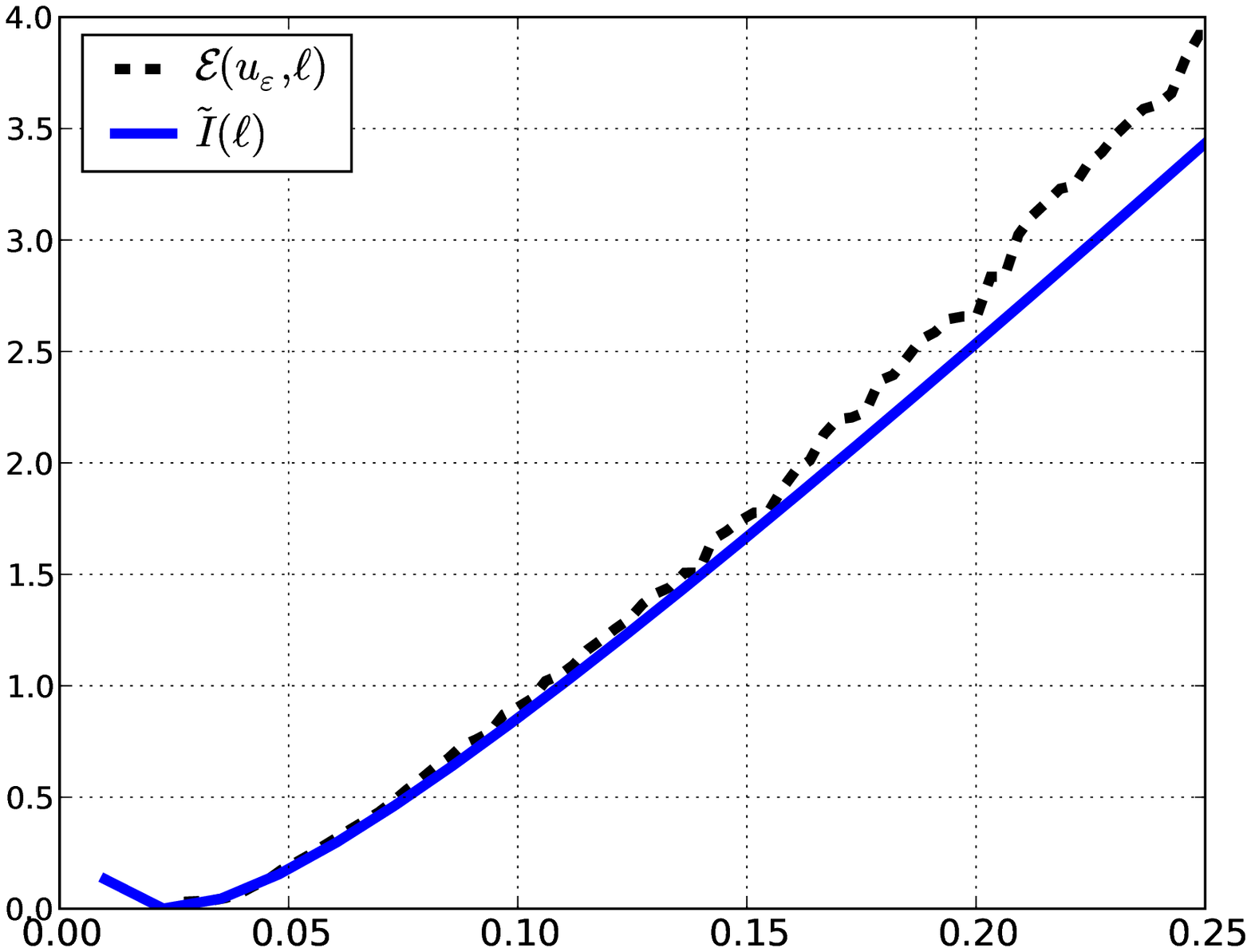}
  \includegraphics[width=0.5\textwidth]{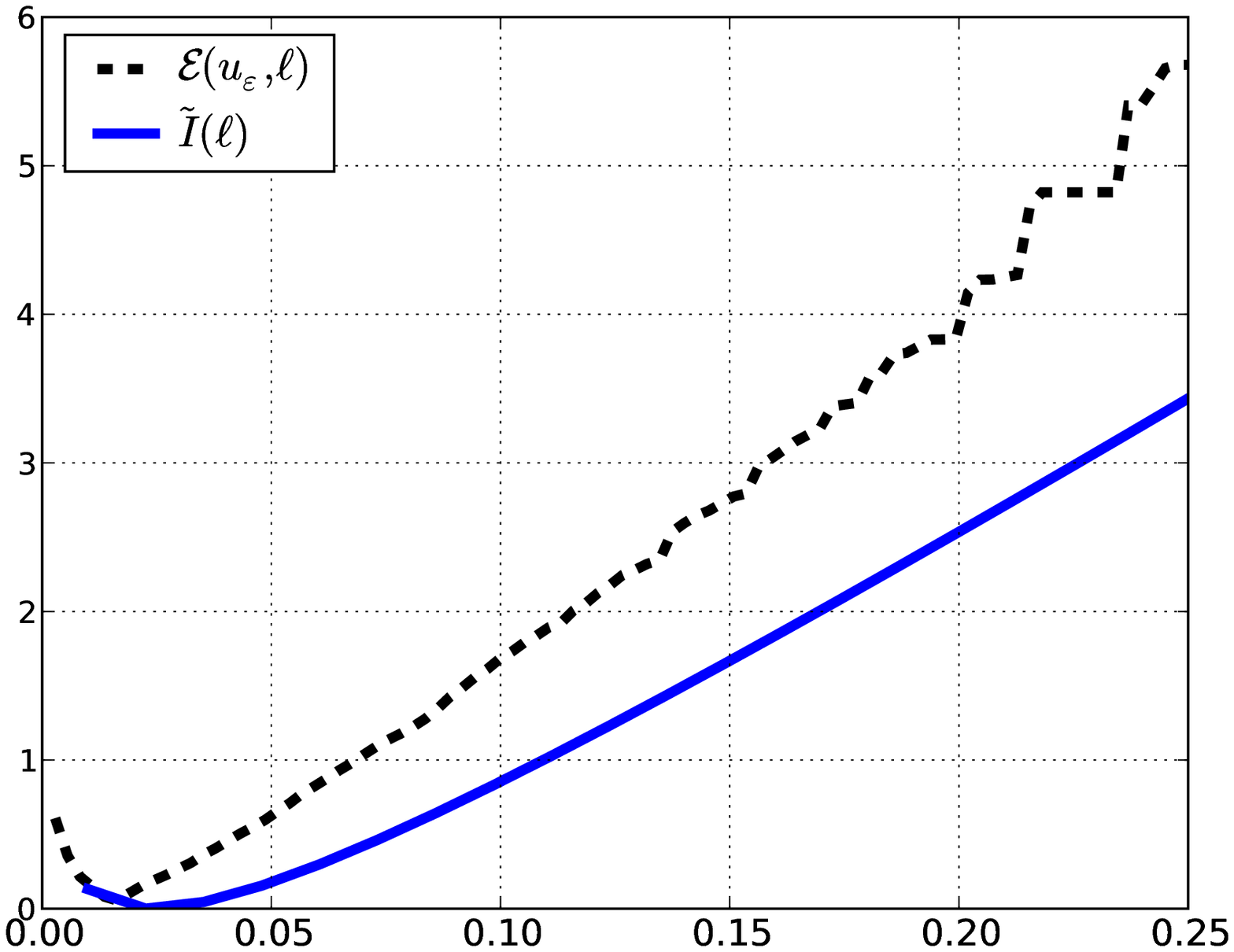}
  \caption{LDP results for the convolved media.  Comparison of empirical rate function $\calE(\ueps,\ell)$ with theoretical (approximate) rate function $\Itilde(\ell)$.  Here $\xi=1$, $\|h\|_1=1$, $\kappa=1$ and $\eps=1/100$ (left) or $\eps=1/10$ (right).}
  \label{fig:convolved_LDP1}
\end{figure}
Figures \ref{fig:convolved_LDP1} and \ref{fig:convolved_LDP2} show that the approximate rate function captures the large-deviation behavior well (for $\ell$ not too large).  As in the case of the wild medium from section \ref{subsection:independent_media} we consider $\ell$ a large deviation if the empirical rate function $\calE(\ueps,\cdot)$ differs significantly from the Gaussian rate function at that point.  In all cases, $\eps$ must be small enough, but when it is the match is good.  
\begin{figure}[h!]
  \includegraphics[width=0.5\textwidth]{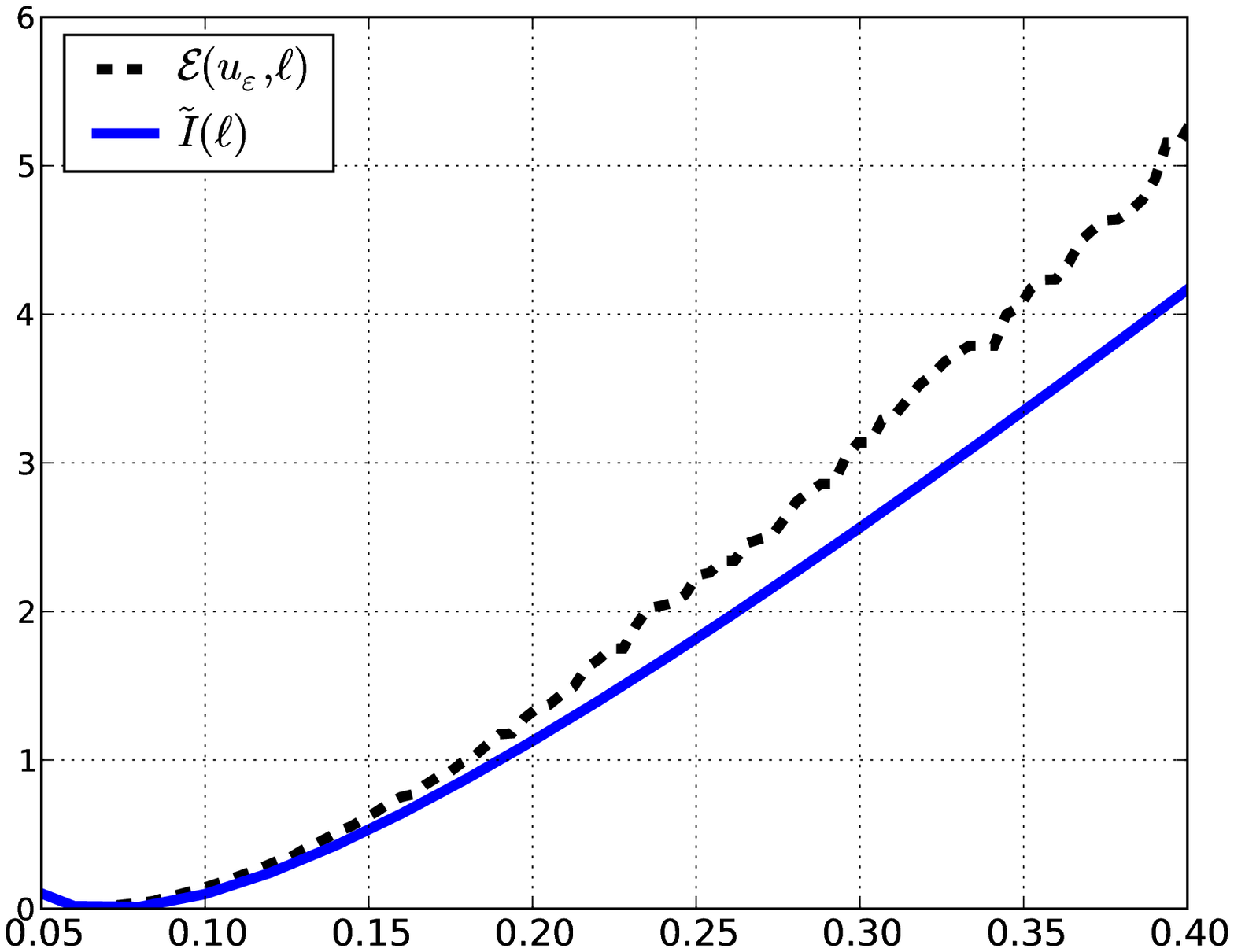}
  \includegraphics[width=0.5\textwidth]{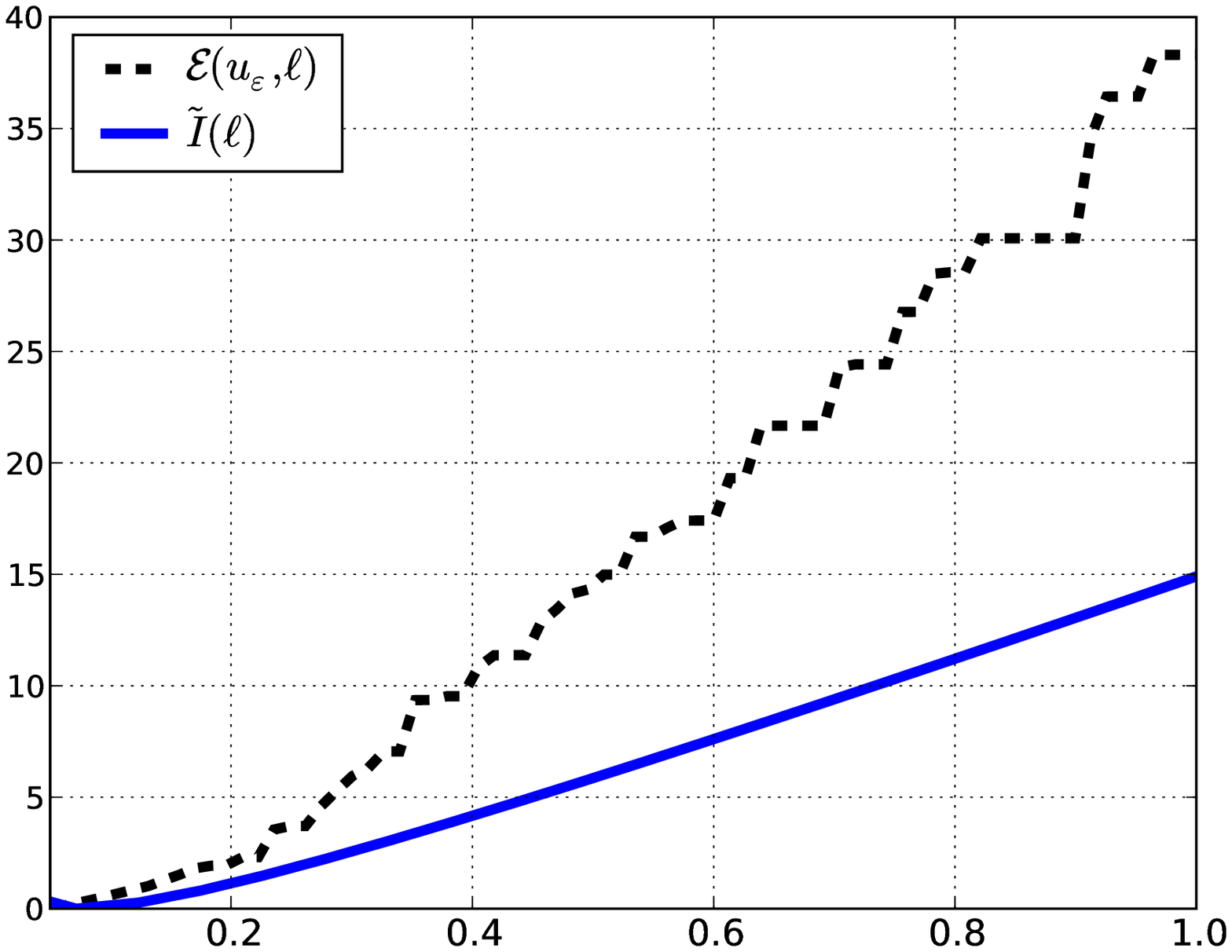}
  \caption{LDP results for the convolved media.  Comparison of empirical rate function $\calE(\ueps,\ell)$ with theoretical (approximate) rate function $\Itilde(\ell)$.  Here $\xi=3$, $\|h\|_1=1$, $\kappa=1$ and $\eps=1/100$ (left) or $\eps=1/10$ (right).  Note that the $\eps=1/100$ plot only has $\ell\in(0,0.4))$ due to the difficulty in sampling the extremely rare event $\ueps\geq0.5$.}
  \label{fig:convolved_LDP2}
\end{figure}
Figure \ref{fig:convolved_LDP3} shows that the corrector cannot capture the large deviation behavior of this media.  As is the case with the parameterized media of section \ref{subsection:independent_media} the approximate rate function does a much better job than the corrector (once $\eps$ is small enough). 
\begin{figure}[h!]
  \includegraphics[width=0.5\textwidth]{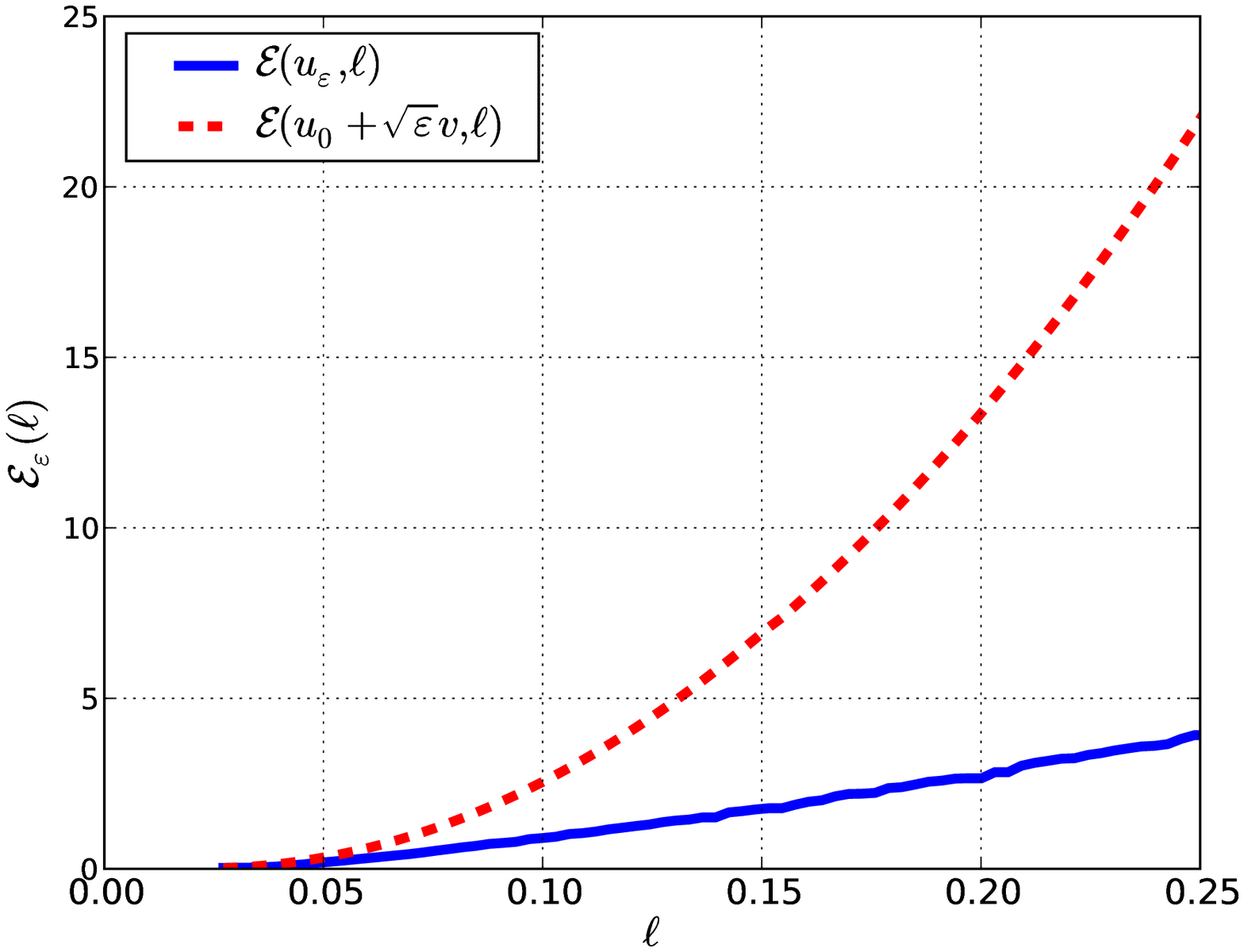}
  \includegraphics[width=0.5\textwidth]{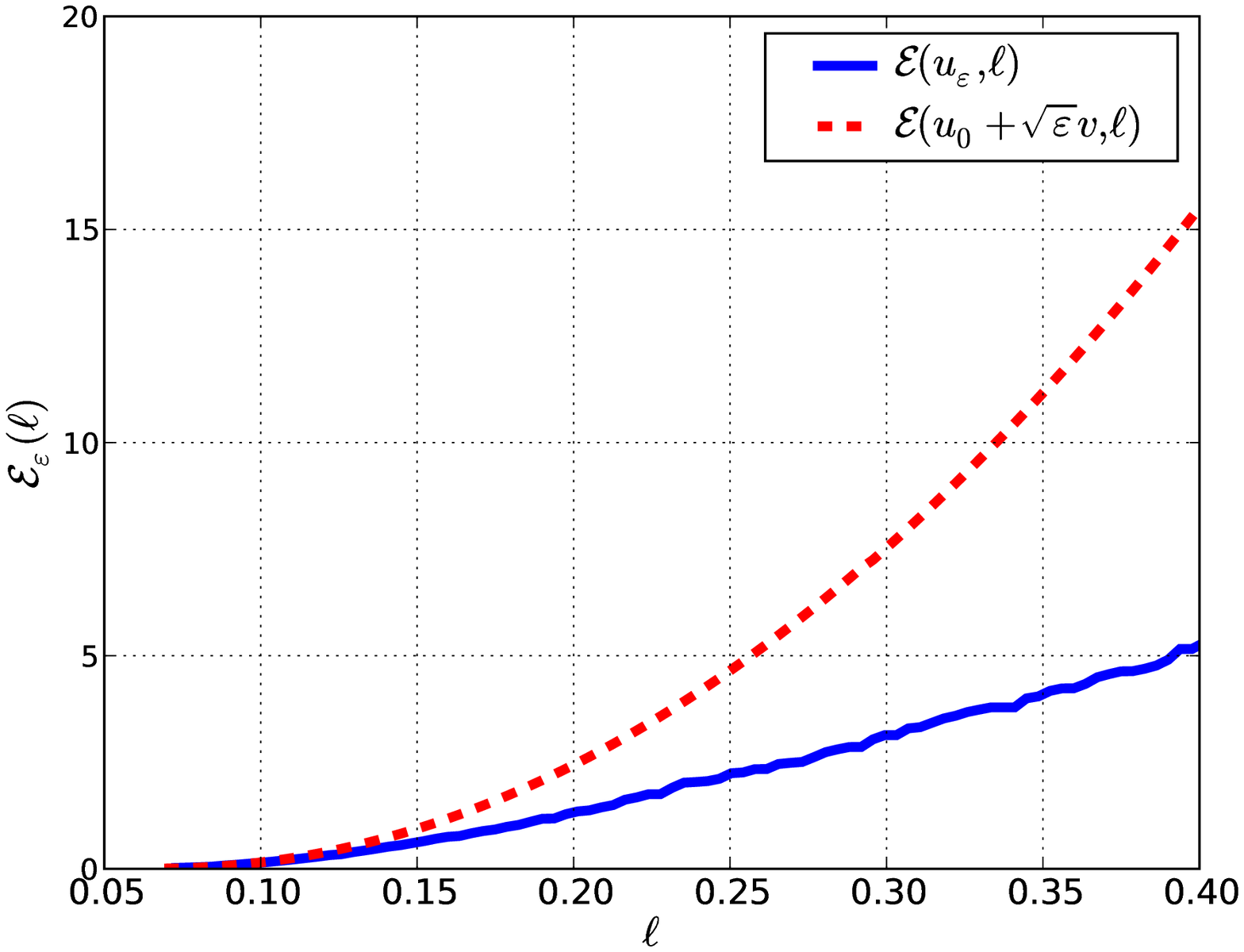}
  \caption{LDP results for the convolved media.  Comparison of empirical rate function $\calE(\ueps,\ell)$ with empirical corrector rate function $\calE(u_0+\eps v,\ell)$.  Here $\|h\|_1=1$, $\kappa=1$, $\eps=1/100$, and $\xi=1$ (left) or $\xi=3$ (right).  In both cases the true rate function is much different than that of the corrector.  The corrector's rate function performs better when $\xi=3$ as can be expected by comparison of the pdf for $\chi^2_3$ and $\chi^2_1$ random variables.}
  \label{fig:convolved_LDP3}
\end{figure}

% arXive modification line
\section{Proof of Homogenization/Gaussian Corrector Results}
\label{section:homogenization_corrector_proofs}
Here we prove extensions of known results.  
\subsection{Proof of theorem \ref{theorem:convergence}}
\label{subsection:homogenization_theorem_proof}
Here we prove one-dimensional homogenization results for our media, which has no uniform (in $\omega$) ellipticity lower bound and is stationary only in the second variable.

\subsubsection{$L^2$ convergence}
Solving \eqref{eq:basic_elliptic} with $\Aeps=\Aeps$ and again with $\Aeps=A_0$ obtain solutions such as \eqref{eq:basic_elliptic_solution}.  Subtracting them we have
\begin{align}
  \begin{split}
    u_0(x)-\ueps(x) &= \int_0^x F(s)\left[ \frac{1}{\Aeps(s)}-\frac{1}{A_0(s)} \right]\ds \\
    &\quad+ \int_0^x\left[ \frac{\brac{F/A_0}}{\brac{1/A_0}}\frac{1}{A_0(s)} - \frac{\brac{F/\Aeps}}{\brac{1/\Aeps}}\frac{1}{\Aeps(s)} \right]\ds \\
    &= I_1(x) + I_2(x).
  \end{split}
  \label{eq:I_decomposition}
\end{align}

We can re-write $I_1$ as
\begin{align*}
  I_1(x) &= \brac{F\one_{[0,x]}/\Aeps} - \brac{F\one_{[0,x]}/A_0}.
  %&= \left( \brac{F\chi_x/\Aeps}-\brac{F\chi_x/A_0} \right) + \left( \brac{F\chi_x/A_0}- \brac{F\chi_x/\Abarrho} \right).
\end{align*}
We are thus motivated to prove the following lemma
\begin{lemma}
  Let $H\in L^\infty[0,1]$ be deterministic, then
  \begin{align*}
    \Exi\left( \brac{H/\Aeps}-\brac{H/A_0} \right)^2 &\leq \eps\|H\|_{L^\infty}^2C_{A_{-1}}.
  \end{align*}
  \label{lemma:oscillatory_integral_bound}
\end{lemma}
\begin{proof}
  Write
  \begin{align*}
    &\Exi\left( \brac{H/\Aeps}-\brac{H/A_0} \right)^2 \\
    &= \eps^2\int_0^{\eps^{-1}}\int_0^{\eps^{-1}} \dt\ds H(s)H(t)
    \Expxi{\left( \frac{1}{A(s\eps,s)} - \Expxi{1/A}(s\eps) \right)\left( \frac{1}{A(t\eps,t)}-\Expxi{1/A}(t\eps) \right)}\\
    &\quad= \eps^2 \int_0^{\eps^{-1}}\int_0^{\eps^{-1}} H(s)H(t)\Covxi(s\eps,t\eps,s-t)\dt\ds 
    \leq \eps \|H\|_{L^\infty}^2 C_{A^{-1}}.
  \end{align*}
\end{proof}
Since $\|F\|_{L^\infty} \leq\|f\|_{L^2}$, lemma \ref{lemma:oscillatory_integral_bound} gives us
\begin{align}
  \Exi\|I_1\|_{L^2}^2 &\leq \eps\|f\|_{L^2}^2C_{A^{-1}},
  \label{eq:I1_bound}
\end{align}

We now consider $I_2$.  After repeated use of the equality $ab - \tilde a \tilde b = (a-\tilde a)b+\tilde a(b-\tilde b)$ we obtain
\begin{align*}
  I_2(x) &= \left( \brac{F/A_0}-\brac{F/\Aeps} \right)\left( \int_0^x\frac{1}{A_0(s)\brac{1/A_0}}\ds \right) \\
  &\quad+\left( \brac{1/\Aeps} - \brac{1/A_0} \right) \left( \frac{1}{\brac{1/A_0}\brac{1/\Aeps}}\int_0^x\frac{\brac{F/\Aeps}}{A_0}\ds \right)\\
  &\quad+\left( \int_0^x\left[ \frac{1}{A_0(s)}-\frac{1}{\Aeps(s)} \right]\ds \right)\left( \frac{\brac{F/\Aeps}}{\brac{1/\Aeps}} \right)\\
  &=I_{21}(x) + I_{22}(x) + I_{23}(x).
\end{align*}
Each term is the product of a term similar to $I_1$ and a bounded random variable (\emph{a priori} bounds obtained using the positivity of $A$).  We thus obtain
\begin{align}
  \begin{split}
    \sqrt{\Exi\|I_2\|_{L^2}^2} &\leq \sqrt{\Exi\|I_{21}\|_{L^2}^2} + \sqrt{\Exi\|I_{22}\|_{L^2}^2} + \sqrt{\Exi\|I_{23}\|_{L^2}^2} \\
    &\leq 3\sqrt{\eps}\|f\|_{L^2}\sqrt{C_{A^{-1}}}.
  \end{split}
  \label{eq:I2_bound}
\end{align}

The inequality in theorem \ref{theorem:convergence} thus follows from \eqref{eq:I_decomposition}, \eqref{eq:I1_bound}, and \eqref{eq:I2_bound}.

\subsubsection{\as convergence}
Here we use the same decomposition, 
\begin{align*}
  \ueps(x) - u_0(x) &= I_1(x) + I_{21}(x) + I_{22}(x) + I_{23}(x),
\end{align*}
and show that each term goes to zero \as  To that end, we notice that every term is the product of a bounded (sometimes random) variable, and a term like $\brac{H/\Aeps} - \brac{H/A_0}$ (with $H$ depending on $x$).  It will thus suffice to prove \as convergence of this latter term.  We thus obtain pointwise (in $x$) \as convergence.  The \as norm convergence then follows from an \emph{a-priori} bound on every realization of $\ueps-u_0$ and the bounded convergence theorem.  The \emph{a-priori} bound follows from \eqref{eq:ellipticity_condition} (which we assume here is independent of $\eps$) and \eqref{eq:I_decomposition}.  In other words, \as convergence in theorem \ref{theorem:convergence} is a corollary of the following lemma.

\begin{lemma}
  Suppose $H\in L^\infty$ is deterministic, then as $\tau\to\infty$, we have (almost surely $\Pxi$)
  \begin{align*}
    \brac{H/A_{\tau^{-1}}} - \brac{H/A_0} &\to 0.
  \end{align*}
  \label{lemma:as_convergence}
\end{lemma}
\begin{proof}
  The proof is more-or-less a standard trick where we show \as convergence on a sequence of $\tau$ values as well as the difference between the sequence values and ``nearby'' values.  See section 37.7 in \cite{Loeve_Probability_II}.

  We handle the first convergent term first.
  Defining
  \begin{align*}
    Y(\tau) :&= \brac{H/A_{\tau^{-1}}} - \brac{H/A_0} \\
    X(t,s) :&= H(t)\left( \frac{1}{A(t,s)}-\Expxi{1/A}(t) \right),
  \end{align*}
  we have
  \begin{align*}
    Y(\tau)&= \int_0^1 H(s)\left( \frac{1}{A(s,s\tau)}-\Expxi{1/A}(s) \right)\ds \\
    &=\frac{1}{\tau}\int_0^\tau H(s/\tau)\left( \frac{1}{A(s/\tau,s)} - \Expxi{1/A}(s/\tau) \right)\ds
    = \frac{1}{\tau}\int_0^\tau X(s/\tau,s)\ds.
  \end{align*}

  Now for $m\in\Nat$, $a>0$, and $m^a\leq\tau<(m+1)^a$,
  \begin{align*}
    \frac{\tau}{m^a}Y(\tau) &= Y(m^a) + Z(m^a,\tau),\qquad
    Z(m^a,\tau) := \frac{1}{m^a}\int_{m^a}^\tau X(s/\tau,s)\ds.
  \end{align*}

  Directly from lemma \ref{lemma:oscillatory_integral_bound} we have
  \begin{align*}
    \Expxi{Y(m^a)^2} &\leq \|H\|_{L^\infty}^2C_{A^{-1}} m^{-a},
  \end{align*}
  and then Chebyshev inequality gives us, for any $\delta>0$,
  \begin{align*}
    \Pxi[|Y(m^a)|>\delta] &\leq \frac{\Expxi{Y(m^a)^2}}{\delta^2}\leq\frac{\|H\|_{L^\infty}^2C_{A^{-1}}m^{-a}}{\delta^2}.
  \end{align*}
  We choose $a>1$ and then
  \begin{align*}
    &\sum_{m=1}^\infty \Pxi[|Y(m^a)|>\delta] < \infty.
  \end{align*}
  So by the Borel-Cantelli lemma, $Y(m^\alpha)\to0$ $\Pxi$ \as

  As for $Z(m^a,\tau)$, we set
  \begin{align*}
    U(m^a) :&= \sup_{m^a\leq\tau<(m+1)^a} |Z(m^a,\tau)| \leq \frac{1}{m^a}\int_{m^a}^{(m+1)^a}|X(s/\tau,s)|\ds,
  \end{align*}
  and note that
  \begin{align*}
    \Exi |U(m^a)|^2 &\leq \frac{\|H\|_\infty^2}{m^{2a}} \int_{m^a}^{(m+1)^a}\int_{m^a}^{(m+1)^a}C(s-s')\ds\ds'\\
    &\leq \|H\|_\infty^2 C_{A^{-1}}\frac{(m+1)^a-m^a}{m^a} \leq \|H\|_\infty^2C_{A^{-1}}\frac{\lceil a+1\rceil !}{m}.
    %&\frac{\|H\|_{L^\infty}}{\nu_1}\frac{(m+1)^a-m^a}{m^a} \leq \frac{\|H\|_{L^\infty}}{\nu_1}\frac{\lceil a+1\rceil !}{m}.
  \end{align*}
  Therefore $\Expxi{U(m^a)^2}\lesssim m^{-2}$, and by Chebyshev's inequality and the Borel-Cantelli lemma we obtain $U(m^a)\to0$ $\Pxi$ \as  The same conclusion thus holds for $Y(\tau)$, $m^a\leq\tau<(m+1)^a$, and therefore for $Y(\tau)=\brac{H/A_{\tau^{-1}}} - \brac{H/A_0}$.
\end{proof}

\subsection{Proof of theorem \ref{theorem:corrector}}
\label{subsection:corrector_proof}
Here we prove one-dimensional corrector results for our media, which has no uniform (in $\omega$) ellipticity lower bound and is stationary only in the second variable.
\subsubsection{One dimensional oscillatory integral}
Here we study the integral
\begin{align*}
  \veps(x) :&= \int_0^1 G(x,t)\qeps(t)\dt,\quad
  \qeps(t) := \frac{1}{A(t,\frac{t}{\eps})} - \Expxi{1/A}(t),
\end{align*}
where $G(x,s)$ is deterministic, piecewise continuous in $s$, and uniformly (in $s$) Lipschitz in $x$.

First note that 
\begin{align*}
  \Exi\left( \frac{\veps(x)}{\sqrt{\eps}}\right)^2 &= \int_0^1\int_0^1 G(x,t)G(x,s)\frac{\Covxi(t,s,\frac{t-s}{\eps})}{\eps}\ds\dt\\
  &= \int_0^1G(x,s)\left[ \int_{-s/\eps}^{(1-s)/\eps}G(x,\eps t + s)\Covxi(\eps t+s, s, t)\dt \right]\ds.
\end{align*}
Using \eqref{eq:cov_pointwise_bound} and dominated convergence, we therefore have
\begin{align}
  \begin{split}
   \Exi\left( \frac{\veps}{\sqrt{\eps}}\right)^2  &\to \int_0^1G(x,t)^2\sigma^2(t)\dt,\qquad
    \sigma^2(t) := \int_{-\infty}^\infty \Covxi(t,t,q)\d q.
  \end{split}
  \label{eq:convergent_variance}
\end{align}
The scaling (in $\eps$) and the fact that $\qeps$ is mean zero indicate that a central-limit type result should show convergence to a Gaussian random variable.  This is indeed the case.  
\begin{lemma}
  If $G(x,s)$ is deterministic, piecewise continuous in $s$, and uniformly (in $s$) Lipschitz in $x$, then 
  \begin{align*}
    \frac{1}{\sqrt{\eps}}\veps(x) = \frac{1}{\sqrt{\eps}}\int_0^1 G(x,t)\qeps(t)\dt & \xrightarrow{dist.} \int_0^1G(x,t)\sigma(t)\dW_t,
  \end{align*}
where $W_t$ is a standard Brownian motion.
  \label{lemma:one_d_integral_convergence}
\end{lemma}

The following result allows us to reduce the problem of proving convergence (of a stochastic process) to one of studying finite dimensional distributions \cite{Billingsley_Convergence}.
\begin{proposition}
  Suppose $(Z_n;1\leq n\leq \infty)$ are random variables with values in the space of continuous functions $C([0,1])$.  Then $Z_n$ converges in distribution to $Z_\infty$ provided that:
  \begin{enumerate}
    \item[(a)] any finite-dimensional joint distribution $(Z_n(x_1),\dots,Z_n(x_k))$ converges to the joint distribution $(Z_\infty(x_1),\dots,Z_\infty(x_k))$ as $n\to\infty$.
    \item[(b)] $(Z_n)$ is a tight sequence of random variables.  A sufficient condition for tightness of $(Z_n)$ is the following Kolmogorov criterion:  There exist positive constants $\nu$, $\beta$ and $\delta$ such that 
      \begin{enumerate}
        \item[(i)] $\sup_{n\geq1}\Exp{|Z_n(t)|^\nu}<\infty$, for some $t\in[0,1]$,
        \item[(ii)] $\Exp{|Z_n(s)-Z_n(t)|^\beta}\lesssim |t-s|^{1+\delta}$,
      \end{enumerate}
      uniformly in $n\geq1$ and $t,s\in[0,1]$.
  \end{enumerate}
  \label{proposition:tightness}
\end{proposition}
Tightness is easily verified.  Indeed, \eqref{eq:cov_pointwise_bound} and \eqref{eq:convergent_variance} show that 
\begin{align*}
  \Exi|\veps(x)-\veps(y)|^2 &=\Exi\left( \int_0^1[G(x,t)-G(y,t)]\qeps(t)\dt \right)^2 \\
  &\to \int_0^1 \left[ G(x,t)-G(y,t) \right]^2\sigma^2(t)\dt \leq C|x-y|^2.
\end{align*}
so condition (b-i) is met with $\nu=2$, and (b-ii) is met with $\beta=2$ and $\delta=1$.

This means that to prove the theorem we simply need to fix $(x_1,\dots,x_n)$ and show 
\begin{align}
  \begin{split}
    \eps^{-1/2}\vec\veps  :&= \eps^{-1/2}(v_\eps(x_1),\dots,v_\eps(x_n)) \xrightarrow{dist.} \vec v := (v(x_1),\dots,v(x_n)),\\
    v(x) :&= \int_0^1G(x,t)\sigma(t)\dW_t.
  \end{split}
  \label{eq:finite_dim_dist_limits}
\end{align}

\begin{proof}
  %We start by considering the case $G\equiv1$ and $\qeps(t) = B(\tovereps)$.  We then have to show 
  %\begin{align}
  %  \Ieps(x) :&= \frac{1}{\sqrt{\eps}}\int_0^x\qeps(t)\dt\xrightarrow{dist.} \int_0^x\dW_t=\sigma W_x.
  %  \label{eq:simple_limit}
  %\end{align}
  %We will first show that, for $x<y$,
  %\begin{align}
  %  \Ieps(x) &\xrightarrow{dist.} N_1, \quad (\Ieps(y)-\Ieps(x))\xrightarrow{dist.} N_2,
  %  \label{eq:single_point_convergence}
  %\end{align}
  %where $N_1$, $N_2$ are independent normal rv's with variance $\sigma^2x$ and $\sigma^2(y-x)$ respectively.  This in turn implies that
  %\begin{align*}
  %  (\Ieps(x), \Ieps(y)) &\xrightarrow{dist.} (N_1, N_1+N_2)=(v(x),v(y)).
  %\end{align*}
  %In other words, the weak limit has the same two-point distribution as $v$.  Note the importance of the independent increments coming from the independence of $N_1$, $N_2$ to the representation \eqref{eq:simple_limit}.  An easy extension then shows \eqref{eq:finite_dim_dist_limits}.

Any finite dimensional distribution $\eps^{-1/2}\vec\veps$ has characteristic function 
\begin{align*}
  \Phi_\eps(k) :&= \Expxi{e^{i\sum_{j=1}^n k_j \eps^{-1/2}\veps(x_j)}},\quad k = (k_1,\dots,k_n).
\end{align*}
The above characteristic function may be recast as
\begin{align*}
  \Phi_\eps(k) &= \Expxi{e^{i\int_0^1 m(s)\frac{1}{\sqrt{\eps}}\qeps(s)ds}},\quad m(s) := \sum_{j=1}^n k_j G(x_j,s).
\end{align*}
As a consequence, convergence of the finite dimensional distributions will be proved if we can show convergence of
\begin{align}
  \Ieps :&= \int_0^1 m(s)\frac{1}{\sqrt{\eps}}\qeps(s)\ds \xrightarrow{dist.} N_1 := \int_0^1 m(s)\sigma(s)\dW_s,
  \label{eq:single_point_convergence}
\end{align}
for piecewise continuous $m$.

Proceeding, we now show that $\Ieps\xrightarrow{dist.}N_1$.  We first consider the case of constant $m\equiv1$ and $\qeps(s) = q(s,\sovereps)=q(\sovereps)$.  To that end, set
\begin{align*}
  q_k :&= \int_{(k-1)\eps}^{k\eps}\frac{1}{\eps}q(\frac{s}{\eps})\ds = \int_{k-1}^k q(s)\ds,\quad k=1,\dots,\epsint.
\end{align*}
Then $\Ieps(x) = \sum_{k=1}^{\epsint}q_k+R(\eps)$, where $R(\eps)\to0$ in $L^2$ (and therefore in probability), and hence may be ignored.  We therefore consider the limit
\begin{align*}
  \sqrt{\eps}\sum_{k=1}^{\epsint}q_k &= \frac{1}{\sqrt{\epsint}}\sum_{k=1}^\epsint q_k + R'(\eps), \quad \|R'(\eps)\|_{L^2(\Omega)}\lesssim\eps.
\end{align*}
So we ignore $R'$ and consider the limit of the summation.  Following theorem 19.2 in \cite{Billingsley_Convergence}, we define the sigma fields $\calF_n$, $\calF^n$ generated by $\{q_k\st k\leq n\}$, $\{q_k\st k\geq n\}$ respectively and set
\begin{align*}
  \rho_n :&= \sup\left\{ |\Exp{\eta_1\eta_2}|\st \eta_1\in\calF_1, \eta_2\in\calF^n, \Exp{\eta_i}=0, \Exp{\eta_i^2}=1, i=1,2 \right\}.
\end{align*}
Then, so long as $\sum_{n=1}^\infty\rho_n<\infty$, 
\begin{align*}
  \frac{1}{\sqrt{\epsint}}\sum_{k=1}^\epsint q_k &\xrightarrow{dist.} \calN(0,\nu^2),\quad \nu^2 := \Exp{q_1^2} + 2\sum_{n=2}^\infty \Exp{q_1q_n}.
\end{align*}
Since $\rho_n\leq\varphi(n-1)$, the summability condition on $\rho_n$ is implied by assumptions \ref{assumptions:mixing_condition}.  We now show that $\nu^2=\sigma^2$.  Indeed,
\begin{align*}
  \Exp{q_1^2} &= \Exp{\left[ \int_0^1q(s)\ds\right]^2} = \int_0^1\int_0^1\Covxi(s-s')\ds\ds'.
\end{align*}
Also, using the symmetry of $\Covxi$, we have
\begin{align*}
  \sum_{n=2}^\infty\Exp{q_1q_n} &= \int_0^1\int_1^\infty\Covxi(s-s')\ds\ds'=\int_0^1\int_{-\infty}^0\Covxi(s-s')\ds\ds'.
\end{align*}
Therefore,
\begin{align*}
  \nu^2 &= \int_0^1\int_{-\infty}^\infty\Covxi(s-s')\ds\ds' = \sigma^2.
\end{align*}

This shows \eqref{eq:single_point_convergence}.

To prove the theorem in the case of non-constant $m(s)$, and $q=q(s,\sovereps)$ we note that if $m$ is replaced by $m_h$, and $q$ is replaced by $q_h$, giving us $\Ieps^h$, then 
\begin{align}
  \begin{split}
    \Exi|\Ieps^h(x) - \Ieps(x)|^2 &= \frac{1}{\eps}\Exi\int_0^1\int_0^1[m(s)q(s,\sovereps)-m_h(s)q_h(s,\sovereps)]\\
    &\quad\times[m(t)q(t,\tovereps)-m_h(t)q_h(t,\tovereps)]\ds\dt.
  \end{split}
  \label{eq:vh_minus_v}
\end{align}
We will choose $m_h$, $q_h$ such that the above expectation vanishes in the limit. 
Split $[0,1]$ into subintervals of size $h$ (with $1/h\subset\Nat$) with endpoints $t_j:= jh$, $j=1,\dots,1/h$.  Now set
\begin{align*}
  m_h(t) := m(t_j), \mbox{ and } q_h(t,\tovereps) := q(t_j, \tovereps),\mbox{ for } t\in[t_j,t_{j+1}).
\end{align*}
Taking expectation inside of the integral \eqref{eq:vh_minus_v}, changing $t\mapsto\eps t$, then $t\mapsto t+s/\eps$ we have (similar to \eqref{eq:convergent_variance})
\begin{align*}
  \Exi|\Ieps^h-\Ieps|^2&= \int_0^1\int_{-\infty}^\infty Q_{\eps,h}(s,t)\dt\ds,
\end{align*}
where $|Q_{\eps,h}(s,t)|\leq C\Gamma(t)$ (with $\Gamma$ from \eqref{eq:cov_pointwise_bound}).  Using the continuity (on the diagonal) of $\Covxi$, and the piecewise continuity of $G$, we also have $Q_{\eps,h}\to0$ as $(\eps,h)\to(0,0)$ (in any way whatsoever, for \emph{a.e.} (s,t)).  Therefore, it suffices to prove the lemma \ref{lemma:one_d_integral_convergence} with $m_h$, $q_h$ replacing $m$, $q$.

We proceed to split the integral defining $\Ieps$ up.
\begin{align*}
  \Ieps :&= \int_0^1m_h(t)q_h(t,\tovereps)\dt = \sum_{j=0}^{1/h-1}\int_{[t_j,t_{j+1})}m(t_j)q(t_j, \tovereps)\dt.
\end{align*}
Each subintegral is handled exactly as before, yielding 
\begin{align*}
  m(t_j)\int_{[t_j,t_{j+1})}q(t_j,\tovereps)\dt&\xrightarrow{dist.} \int_{[t_j,t_{j+1})}m(t_j)\sigma(t_j)\dW_t.
\end{align*}
It remains then to show that the limiting Gaussians (above) are independent.  We do this in the case where $\Ieps=I_{\eps,1}+I_{\eps,2}$ is split into two intervals, the general case following by induction.
To that end, we show that for all $k_1,k_2\in\Rone$
\begin{align}
  \calE :&= \left|\Expxi{e^{i(k_1I_{1,\eps} + k_2I_{2,\eps})}} - \Expxi{e^{ik_1I_{1,\eps}}}\Expxi{e^{ik_2I_{2,\eps}}}\right|\to0.
  \label{eq:independent_limit}
\end{align}
Define
\begin{align*}
  P^\eta_{\eps,1} :&= \int_{\eta}^{x-\eta}\frac{1}{\sqrt{\eps}}q(\tovereps)\dt,\quad Q^\eta_{\eps,1} := I_{\eps,1}- P^\eta_{\eps,1},\\
  P^\eta_{\eps,2} :&= \int_{x+\eta}^{y-\eta}\frac{1}{\sqrt{\eps}}q(\tovereps)\dt,\quad Q^\eta_{\eps,2} := I_{\eps,2}- P^\eta_{\eps,2}.
\end{align*}
Then the first term in $\calE$ decomposes as
\begin{align*}
  \Expxi{e^{i(k_1I_{\eps,1} + k_2I_{\eps,2})}} &= \Expxi{(e^{ik_1Q^\eta_{\eps,1}}-1)e^{ik_1P^\eta_{\eps,1} + ik_2I_{\eps,2}}} \\
  &\quad+ \Expxi{e^{ik_1P^\eta_{\eps,1} + ik_2I_{\eps,2}}}.
\end{align*}
The first term in the decomposition is small since
\begin{align}
  \left|\Expxi{(e^{ik_1Q^\eta_{\eps,1}}-1)e^{ik_1P^\eta_{\eps,1} + ik_2I_{\eps,2}}} \right|&\leq \left( \Exi|e^{ikQ^\eta_{\eps,1}}-1|^2 \right)^{1/2}\lesssim\eta^{1/2}.
  \label{eq:QminusOneExpectation}
\end{align}
As for the second term,
\begin{align*}
  &\Expxi{e^{ik_1P^\eta_{\eps,1}+ik_2I_{\eps,2}}} - \Expxi{e^{ik_1I_{\eps,1}}}\Expxi{e^{ik_2I_{\eps,2}}}\\
  &\quad=\left[ \Expxi{e^{ik_1P^\eta_{\eps,1} + ik_2I_{\eps,2}}} - \Expxi{e^{ik_1P^\eta_{\eps,1}}}\Expxi{e^{ik_2I_{\eps,2}}} \right] \\
  &\qquad- \left[ \Expxi{\left( e^{ik_1Q^\eta_{\eps,1}}-1 \right)e^{ik_1P^\eta_{\eps,1}}}\Expxi{e^{ik_2I_{\eps,2}}} \right].
\end{align*}
The first bracketed term is $\lesssim\varphi(\eta/\eps)$ by our mixing condition \eqref{assumptions:mixing_condition}, and the second is $\lesssim\eta^{1/2}$ in a manner similar to \eqref{eq:QminusOneExpectation}.  Therefore,
\begin{align*}
  \calE&\lesssim \left( \varphi\left( \frac{2\eta}{\eps}\right) + \eta^{1/2} \right) \to0,\mbox{ as }\eps\to0
\end{align*}
for say $\eta=\eps^{1/2}$.

We have thus shown
\begin{align*}
  \Ieps &\xrightarrow{dist.} \int_0^1 m_h(t)\sigma_h(t)\dW_t,
\end{align*}
and this completes the proof.
\end{proof}
%We split this last integral in $s$ into regions $0<s<x-\eps$, and $x-\eps<s<x$.  Then, assuming $|\Covxi(q)|\lesssim(1+|q|^\alpha)^{-1}$, $\alpha>1$ (TODO:  Put in main assumptions), the resultant integral is bounded by
%\begin{align*}
%  \int_0^{x-\eps} \int_{\frac{x-s}{\eps}}^{\frac{y-s}{\eps}} \frac{1}{|q|^\alpha}\d q\ds + \int_{x-\eps}^x\int_{\frac{x-s}{\eps}}^{\frac{y-s}{\eps}} \frac{1}{1+|q|^\alpha}\d q\ds.
%\end{align*}
%The first can be integrated explicitly and is $\lesssim\eps^{\alpha-1}\wedge\eps$.  The second integral is $\lesssim\eps$ due to the integrability of $(1+|q|^\alpha)^{-1}$.  

We also prove a complementary lemma that allows us to deal with the lack of a uniform (in $\omega$) ellipticity lower bound.  
\begin{lemma}
  Suppose $A$ satisfies the hypothesis of theorem \ref{theorem:corrector}.  Then with
  \begin{align*}
    q(s,\sovereps) :&= \frac{1}{A(s,\sovereps)} - \frac{1}{A_0(s)},
  \end{align*}
  we have for any $H\in L^\infty([0,1])$,
  \begin{align*}
    \calE := \Expxi{\left( \int_0^1H(s)q(s,\sovereps)\ds \right)^6} &\leq C_\xi \left( \int_0^1\int_0^1H(s)H(t)\rho^{1/3}(\frac{s-t}{\eps})\ds\dt \right)^3\\
    &\leq \eps^3C_\xi \|H\|_\infty^6 \|\rho^{1/3}\|_{L^1}.
  \end{align*}
  \label{lemma:fourth_moment}
\end{lemma}
\begin{proof}
  \begin{align*}
    \calE&\leq \int_{[0,1]^6}|H(s_1)\cdots H(s_6)|\left|\Expxi{q(s_1,\frac{s_1}{\eps})\cdots q(s_6,\frac{s_6}{\eps})}\right|\ds.
  \end{align*}
The expectation can be broken up using Cauchy-Schwartz into a product that looks like
\begin{align*}
  \left( \Expxi{X_1^3X_2^3} \Expxi{X_3^3X_4^3} \Expxi{X_5^3X_6^3} \right)^{1/3}.
\end{align*}
For each term $\Expxi{X_i^3X_{i+1}^3}$, the mixing condition, and our bound on $6^{th}$ moments gives
  \begin{align*}
    \Expxi{q(s,\sovereps)^3q(t,\tovereps)^3} &\leq \rho(\frac{s-t}{\eps})C_\xi.
  \end{align*}
  The result follows.
\end{proof}
\subsubsection{Asymptotic expansion of the solution $\ueps$}
\label{subsection:asymptotic_expansion}
We now complete the proof of theorem \ref{theorem:corrector}.
%\textbf{For my notes}
%\begin{align*}
%  &\brac{F/\Aeps}\frac{1}{\brac{1/\Aeps}}\int_0^x\frac{1}{\Aeps(s)}\ds \\
%  &= \left[ \brac{F/A_0}+Y_1^\eps \right]\left[ \frac{1}{\brac{1/A_0}} - \frac{X_1^\eps}{\brac{1/A_0}^2} + \frac{(X_1^\eps)^2}{\brac{1/A_0}^2\brac{1/\Aeps}} \right]\left[ \int_0^x\frac{1}{A_0(s)}\ds + X_x^\eps \right].
%\end{align*}
%\textbf{End of:  For my notes}\\
Starting from \eqref{eq:asymptotic_expansion}, \eqref{eq:veps_rewritten} we write
\begin{align*}
  \frac{\ueps(x) - u_0(x)}{\sqrt{\eps}} &= \frac{1}{\sqrt{\eps}}\veps(x) + \frac{1}{\sqrt{\eps}}R_\eps(x).
\end{align*}
%\begin{align*}
%  \ueps(x) :&= u_0(x) + \sqrt{\eps}C_\eps(x) + \eps R_\eps(x),\\
%  C_\eps(x) :&= -Y_x^\eps + \left( Y_1^\eps - X_1^\eps\frac{\brac{F/A_0}}{\brac{1/A_0}} \right)\frac{1}{\brac{1/A_0}}\int_0^x\frac{1}{A_0(s)}\ds + X_x^\eps\frac{\brac{F/A_0}}{\brac{1/A_0}},\\
%  R_\eps(x) :&= (X_1^\eps)^2\frac{\brac{F/\Aeps}}{\brac{1/A_0}^2\brac{1/\Aeps}}\int_0^x\frac{1}{\Aeps(x)}\ds\\
%  &\quad- \frac{X_1^\eps}{\brac{1/A_0}^2}\left[ Y_1^\eps\int_0^x\frac{1}{\Aeps(s)}\ds + \brac{F/\Aeps}X_x^\eps \right] + \frac{Y_1^\eps X_x^\eps}{\brac{1/A_0}}.
%\end{align*}
The convergence of $\eps^{-1/2}\veps$ is assured by lemma \ref{lemma:one_d_integral_convergence}.  We now show that $\eps^{-1/2}R_\eps$ is tight and converges pointwise to zero.  Then, proposition \ref{proposition:tightness} shows that $\eps^{-1/2}R_\eps\xrightarrow{dist.}0$ in the space of continuous paths.

To that end we write $\eps^{-1/2}R(x)$ as a sum of terms of the form
\begin{align*}
  Z\int_0^x\Beps(s)\ds,\qquad \Beps(s) := \frac{1}{\sqrt{\eps}}\left( \frac{1}{\Aeps(s)} - \frac{1}{A_0(s)} \right).
\end{align*}
for a constant random variable $Z$, with either
\begin{align*}
  |Z|&\leq C'_\xi (X_1^\eps)^2,\quad |Z|\leq C'_\xi |X_1^\eps Y_1^\eps|,\quad |Z|\leq C'_\xi|X_1^\eps|\brac{1/\Aeps},\quad\mbox{or}\quad |Z| \leq C'_\xi Y_1^\eps. 
\end{align*}
We first show that $\Expxi{(\eps^{-1/2}(R_\eps(y)-R_\eps(x)))^3}\lesssim\eps|y-x|^2$, meeting condition (b-ii) of proposition \ref{proposition:tightness} with $\beta=3$, $\delta=1$.  Then choosing $x=0$, we have (since $R_\eps(0)=0$) $\eps^{-1/2}R_\eps(y)\to0$ in $L^3(\Omega,\Pxi)$.  This meets condition (b-i) of proposition \ref{proposition:tightness} with $\nu=3$ and also shows that all finite dimensional distributions converge to the zero vector as well (meeting condition (a) the proposition).

Using H\"older, with $1/p + 1/p' = 1$, and $r>0$
\begin{align*}
  \Expxi{Z\int_x^y\frac{1}{\Beps(s)}\ds}^r&\leq \left( \Expxi{|Z|^{rp}} \right)^{1/p}\left( \Expxi{\left| \int_x^y\frac{1}{\Beps(x)}\ds \right|^{rp'}} \right)^{1/p'}.
\end{align*}
We choose $r=2$, $p=3/2$, $p'=3$ to get
\begin{align*}
  \Expxi{Z\int_x^y\frac{1}{\Beps(s)}\ds}^2&\leq \left( \Expxi{|Z|^3} \right)^{2/3}\left( \Expxi{\left| \int_x^y\frac{1}{\Beps(x)}\ds \right|^{6}} \right)^{1/3}.
\end{align*}
Now due to lemma \ref{lemma:fourth_moment}, 
\begin{align*}
  \left(\Expxi{\left| \int_x^y\frac{1}{\Beps(x)}\ds \right|^6} \right)^{1/3} &\leq C_\xi^{1/3}\|\rho\|_\infty |y-x|^2.
\end{align*}
Also, using H\"older and lemma \ref{lemma:fourth_moment}, $\Expxi{|Z|^3}$ is $\lesssim \eps$ in all cases of $Z$.  We therefore have
\begin{align*}
  \Expxi{\left| \frac{R_\eps(y)-R_\eps(x)}{\sqrt{\eps}} \right|^3} &\leq \eps C''_\xi |y-x|^2.
\end{align*}

\section*{Acknowledgment}
  The authors would like to thank Mark Adams and Yu Gu for many helpful discussions.  This work was supported in part by NSF grant 0904746 and NSF RTG grant DMS-0602235. Any opinions, findings, and conclusions or recommendations expressed in this material are those of the author(s) and do not necessarily reflect the views of the National Science Foundation.

%\bibliographystyle{plain}
%\bibliography{$HOME/sync/bibliography} 

\end{document}